\documentclass[11pt]{amsbook}
\usepackage{amssymb}
\usepackage{mathrsfs}
\usepackage[all]{xy}
\usepackage{hyperref}

\DeclareMathOperator{\flim}{``\mathrm{lim}''}
\DeclareMathOperator{\fcolim}{``\mathrm{colim}''}

\newtheorem{theorem}{Theorem}
\numberwithin{theorem}{section}
\newtheorem{lemma}[theorem]{Lemma}
\newtheorem{corollary}[theorem]{Corollary}
\newtheorem{proposition}[theorem]{Proposition}

\newtheorem{defprop}[theorem]{Definition/Proposition}

\theoremstyle{definition}
\newtheorem{remark}[theorem]{Remark}
\newtheorem{exercise}[theorem]{Exercise}
\newtheorem{warning}[theorem]{Warning}
\newtheorem{definition}[theorem]{Definition}

\newtheorem{example}[theorem]{Example}

\newtheorem{construction}[theorem]{Construction}

\date{\today}

\title{Six-Functor Formalisms}

\author{Peter Scholze}

\begin{document}

\maketitle

\tableofcontents

\chapter*{Six-Functor Formalisms}

\section*{Preface}

These are lecture notes for a course in Winter 2022/23, updated and completed in October 2025.\\

The goal of the lectures is to present some recent developments around six-functor formalisms, in particular
\begin{enumerate}
\item[{\rm (1)}] the abstract theory of $6$-functor formalisms;
\item[{\rm (2)}] the $2$-category of cohomological correspondences, and resulting simplifications in the proofs of Poincar\'e--Verdier duality results;
\item[{\rm (3)}] the relation between $6$-functor formalisms and ``geometric rings'';
\item[{\rm (4)}] many examples of $6$-functor formalisms, both old and new.
\end{enumerate}

In the last few years, there has been a surge of activity in the area. Let me mention some (personal) highlights, necessarily highly incomplete:
\begin{enumerate}
\item[{\rm (1)}] The construction principle of $6$-functor formalisms has seen multiple (new) proofs. On the one hand, Chowdhury \cite{ChowdhuryIII} has reworked Liu--Zheng's proof (who also updated their work). On the other hand, Cnossen--Lenz--Linskens \cite{CnossenLenzLinskens} have reproved and sharpened Liu--Zheng's construction principle using $(\infty,2)$-categories of spans, thereby also relating it to the Gaitsgory--Rozenblyum approach \cite{GaitsgoryRozenblyum}. Unlike the work of Liu--Zheng that heavily relies on the quasicategory model for $(\infty,1)$-categories, their argument is ``model-independent''. In a related vein, the Gaitsgory--Rozenblyum approach has now been completed in work of Loubaton--Ruit \cite{LoubatonRuit}, with proofs of all of their conjectures.
\item[{\rm (2)}] In a related direction, Dauser--Kuijper \cite{DauserKuijper} have proved a conjecture stated in the original version of these notes, that the Liu--Zheng construction is in fact unique (up to all higher coherences), when all morphisms are truncated. Unlike the work of Cnossen--Lenz--Linskens, their result does not require $(\infty,2)$-categorical upgrades of the notion of a six-functor formalism.
\item[{\rm (3)}] Heyer--Mann \cite{HeyerMann} have developed many of the ideas presented in these lectures in much more detail, and the reader should consult their work for a more thorough treatment.
\item[{\rm (4)}] The idea that motives are closely related to ring stacks has been turned into precise (higher categorical) theorems in Aoki's upcoming thesis \cite{AokiThesis}.
\end{enumerate}

\hfill{October 2025}

\newpage

\section{Lecture I: Introduction}

The first main goal of these lectures is to answer the question: What is a $6$-functor formalism?

For a long time, the answer was ``you know it when you see it'': The first such formalism was established during the development of \'etale cohomology of schemes, and it was soon realized that very similar formalisms also exist in other contexts, such as for (usual) cohomology of topological spaces, or for $D$-modules on algebraic varieties in characteristic $0$. But as I will discuss below, it is a highly nontrivial task to formalize all the structure implicit in a $6$-functor formalism, and the abstract notion was only formalized recently.

To get started, we will recall the $6$-functor formalism on (nice) topological spaces. For this lecture, let us work with the category $C$ of finite-dimensional locally compact Hausdorff spaces $X$, say (for simplicity) those that can be written as locally closed subsets of some $\mathbb R^n$. We are interested in studying the cohomology of $X$, say $H^i(X,\mathbb Z)$. In this generality, where $X$ might be a Cantor set, the good definition of cohomology is not singular cohomology, but instead the sheaf cohomology. To define it, one starts with the abelian category $\mathrm{Ab}(X)$ of abelian sheaves on $X$ and the global sections functor
\[
H^0(X,-): \mathrm{Ab}(X)\to \mathrm{Ab}.
\]
This is a left exact functor, and one can define its right derived functors
\[
H^i(X,-): \mathrm{Ab}(X)\to \mathrm{Ab}.
\]
Applied to the constant sheaf $\mathbb Z$ on $X$, this defines the cohomology groups $H^i(X,\mathbb Z)$ as desired.

Reflecting on this definition, one is led to also contemplate the full derived functor
\[
R\Gamma(X,-): D(\mathrm{Ab}(X))\to D(\mathrm{Ab}).
\]
In particular, one is led to contemplate the derived category $D(X,\mathbb Z):=D(\mathrm{Ab}(X))$ of abelian sheaves on $X$, for any such $X$. As a categorically minded person, one is also immediately led to contemplate its functoriality in $X$. Namely, for any $f: X\to Y$, there is an exact pullback functor $f^\ast: \mathrm{Ab}(Y)\to \mathrm{Ab}(X)$ inducing a functor
\[
f^\ast: D(Y,\mathbb Z)\to D(X,\mathbb Z)
\]
(our first functor!) which has a right adjoint
\[
f_\ast: D(X,\mathbb Z)\to D(Y,\mathbb Z)
\]
(our second functor!). In the case where $Y=\ast$ is a point, the functor $f^\ast$ is the ``constant sheaf functor'', while $f_\ast=R\Gamma(X,-)$ is the functor introduced above. In particular, the cohomology of $X$,
\[
R\Gamma(X,\mathbb Z)=f_\ast \mathbb Z=f_\ast f^\ast \mathbb Z\in D(\ast,\mathbb Z)=D(\mathbb Z),
\]
has a simple description in terms of these functors.

What is the functor $f_\ast$ in general? It is a relative version of cohomology. Ideally, one would like to say that it interpolates the cohomology of all the fibres. While this is not true in general, it is true when $f$ is proper (i.e.~the preimage of any compact subset is compact):

\begin{theorem}[Proper Base Change] Let $f: X\to Y$ be a proper map in $C$ and let $y\in Y$ with fibre $X_y=X\times_Y \{y\}$, with inclusions $i_X: X_y\to X$ and $i: \{y\}\to Y$. Take any $A\in D(X,\mathbb Z)$. Then the natural map
\[
(f_\ast A)_y\to R\Gamma(X_y,i^\ast A)
\]
is an isomorphism, where $(f_\ast A)_y := i^\ast f_\ast A\in D(\{y\},\mathbb Z)=D(\mathbb Z)$ is the stalk of $f_\ast A$ at $y$.

More generally, for any other map $g: Y'\to Y$ with base change $f':X'=X\times_Y Y'\to Y'$ (and $g': X'\to X$), the natural base change transformation
\[
g^\ast f_\ast\to f'_\ast g^{\prime\ast}
\]
of functors $D(X,\mathbb Z)\to D(Y',\mathbb Z)$ is an isomorphism.
\end{theorem}

Here, $g^\ast f_\ast\to f'_\ast g^{\prime\ast}$ is adjoint to
\[
f^{\prime\ast} g^\ast f_\ast = g^{\prime\ast} f^\ast f_\ast\to g^{\prime\ast}
\]
where the latter map comes from the counit map $f^\ast f_\ast\to \mathrm{id}$.

At this point, we have defined the cohomology $R\Gamma(X,\mathbb Z)$ of any $X\in C$. In fact, we have also encoded the pullback maps $R\Gamma(Y,\mathbb Z)\to R\Gamma(X,\mathbb Z)$ for $f: X\to Y$: There is a unit map $\mathbb Z\to f_\ast f^\ast \mathbb Z=f_\ast \mathbb Z$, and taking $R\Gamma(Y,-)$ produces the desired map as $R\Gamma(Y,f_\ast \mathbb Z)=R\Gamma(X,\mathbb Z)$ as the lower-shriek functors compose.

But there are further important structures on cohomology. For example:

\begin{theorem}[K\"unneth Formula] For proper $X$ and $Y$, there is a natural isomorphism
\[
R\Gamma(X,\mathbb Z)\otimes R\Gamma(Y,\mathbb Z)\cong R\Gamma(X\times Y,\mathbb Z).
\]
\end{theorem}

Here, $\otimes: D(\mathbb Z)\times D(\mathbb Z)\to D(\mathbb Z)$ is the tensor product on $D(\mathbb Z)$, i.e.~the derived tensor product.

Even the formulation of this theorem requires us to contemplate the tensor product on $D(\mathbb Z)$; and of course one is then led to also contemplate the tensor product on $D(X,\mathbb Z)$ for any $X$. (This arises from the usual tensor product on $\mathrm{Ab}(X)$ by deriving.) Thus, we consider
\[
-\otimes-: D(X,\mathbb Z)\times D(X,\mathbb Z)\to D(X,\mathbb Z)
\]
(our third functor!), which again has a (partial) right adjoint
\[
\underline{\mathrm{Hom}}(-,-): D(X,\mathbb Z)^{\mathrm{op}}\times D(X,\mathbb Z)\to D(X,\mathbb Z)
\]
(our fourth functor!) characterized by the adjunction
\[
\mathrm{Hom}(A,\underline{\mathrm{Hom}}(B,C))\cong \mathrm{Hom}(A\otimes B,C).
\]

How do these new functors interact with the previous ones? The pullback functors $f^\ast$ are symmetric monoidal, i.e.~commute with the tensor product, so in particular
\[
f^\ast(A\otimes B)\cong f^\ast A\otimes f^\ast B.
\]
Note that these isomorphisms here are really extra data, that ought to be subject to further compatibilities. Fortunately, there is a theory of symmetric monoidal categories and symmetric monoidal functors that makes it possible to express these compatibilities.

There is also a compatibility between tensor products and pushforward, again in the proper case.

\begin{theorem}[Projection Formula] Let $f: X\to Y$ be proper and $A\in D(X,\mathbb Z)$, $B\in D(Y,\mathbb Z)$. Then the natural map
\[
f_\ast A\otimes B\to f_\ast(A\otimes f^\ast B)
\]
is an isomorphism.
\end{theorem}

The map is adjoint to
\[
f^\ast(f_\ast A\otimes B)\cong f^\ast f_\ast A\otimes f^\ast B\to A\otimes f^\ast B.
\]

This is enough to prove the K\"unneth Formula. In fact, if $X$ and $Y$ are proper and $A\in D(X,\mathbb Z)$ and $B\in D(Y,\mathbb Z)$, we claim that there is a natural isomorphism
\[
R\Gamma(X,A)\otimes R\Gamma(Y,B)\cong R\Gamma(X\times Y,A\boxtimes B)
\]
where $A\boxtimes B := p_1^\ast A\otimes p_2^\ast B\in D(X\times Y,\mathbb Z)$. Here, we use the maps
\[\xymatrix{
X\times Y\ar[r]^{p_2}\ar[d]_{p_1}\ar[dr]^p & Y\ar[d]^{p_Y}\\
X\ar[r]^{p_X} & \ast.
}\]
Indeed,
\[\begin{aligned}
R\Gamma(X\times Y,A\boxtimes B)&=p_\ast(p_1^\ast A\otimes p_2^\ast B)\\
&=p_{X\ast}(p_{1\ast}(p_1^\ast A\otimes p_2^\ast B))\\
&\cong p_{X\ast}(A\otimes p_{1\ast} p_2^\ast B)\\
&\cong p_{X\ast}(A\otimes p_X^\ast p_{Y\ast} B)\\
&\cong p_{X\ast} A\otimes p_{Y\ast} B=R\Gamma(X,A)\otimes R\Gamma(Y,B).
\end{aligned}\]

Finally, there is one more important structural feature of cohomology: Duality.

\begin{theorem}[Poincar\'e Duality] Assume that $X$ is a compact oriented manifold of dimension $d$. Then there is a natural isomorphism
\[
R\Gamma(X,\mathbb Z)[d]\cong R\Gamma(X,\mathbb Z)^\vee
\]
where $-^\vee = \underline{\mathrm{Hom}}(-,\mathbb Z)$ is the duality functor on $D(\mathbb Z)$.
\end{theorem}

In other words, up to shift, cohomology is self-dual. In fact, one can state a more precise version applying to any sheaf $A\in D(X,\mathbb Z)$ with dual $A^\vee=\underline{\mathrm{Hom}}(A,\mathbb Z)$:
\[
R\Gamma(X,A^\vee)[d]\cong R\Gamma(X,A)^\vee.
\]
In fact, one can even replace the $\mathbb Z$ in the dual $\underline{\mathrm{Hom}}(-,\mathbb Z)$ by any $B\in D(\mathbb Z)$, thus leading to
\[
R\Gamma(X,\underline{\mathrm{Hom}}(A,f^\ast B))[d]\cong \mathrm{Hom}_{D(\mathbb Z)}(R\Gamma(X,A),B).
\]
In categorical language, this means precisely that the functor $B\mapsto f^\ast B[d]$ is right adjoint to $f_\ast = R\Gamma(X,-): D(X,\mathbb Z)\to D(\mathbb Z)$. Even more generally:

\begin{theorem}[Verdier duality] Let $f: X\to Y$ be a proper map that is a ``manifold bundle'' (i.e.~locally on $X$ and $Y$ of the form $Y\times \mathrm{ball}\to Y$) of relative dimension $d$. Then the functor $f_\ast: D(X,\mathbb Z)\to D(Y,\mathbb Z)$ admits a right adjoint which is given by
\[
f^\ast\otimes \omega_{X/Y}
\]
for some sheaf $\omega_{X/Y}\in D(X,\mathbb Z)$ that is locally isomorphic to $\mathbb Z[d]$.
\end{theorem}

The ``oriented'' assumption in Poincar\'e duality ensures that $\omega_{X/\ast}\cong \mathbb Z[d]$ globally on $X$.

The problem with this form of Verdier duality is that the assumptions on $f$ combine a global assumption (``proper'') with a local assumption (``manifold bundle''); this also makes a direct proof more complicated. Fortunately, there is a way to state a purely local form of Verdier duality.

Namely, for any map $f: X\to Y$ there is a further ``proper pushforward'' functor
\[
f_!: D(X\mathbb Z)\to D(Y,\mathbb Z)
\]
(our fifth functor!). In this setting, this can be defined as the derived functor of the functor of sections with proper support. In particular, there is a natural transformation $f_!\to f_\ast$ that is an isomorphism when $f$ is proper. Again, there is a right adjoint
\[
f^!: D(Y,\mathbb Z)\to D(X,\mathbb Z)
\]
(the final sixth functor!) called the ``exceptional inverse image'' functor. Now we can state the local version:

\begin{theorem}[Verdier duality] Let $f: X\to Y$ be a ``manifold bundle'' of relative dimension $d$. Then $f^!$ is isomorphic to $f^\ast\otimes \omega_{X/Y}$ where $\omega_{X/Y}=f^!\mathbb Z$ is locally isomorphic to $\mathbb Z[d]$.
\end{theorem}

The advantage of this local statement is that it can be formally reduced to the case that $X = Y\times \mathrm{ball}$ is the product of $Y$ with an (open $d$-dimensional) ball; and by induction one can also assume that $d=1$. Another advantage of introducing the functor $f_!$ is that it enables one to state more general versions of proper base change and the projection formula. In fact, if one replaces $f_\ast$ by $f_!$ in proper base change and the projection formula, they are true without the hypothesis that $f$ is proper! There is however a technical problem: In this more general version of proper base change and the projection formula, one does not have comparison maps a priori, and the isomorphisms are instead extra data that ought to be subject to further coherence isomorphisms.

Finally, we can say what a $6$-functor formalism is, roughly speaking: There is some category $C$ of geometric objects (topological spaces, schemes, stacks, analytic spaces, ...) and an association $X\mapsto D(X)$ from $C$ to (triangulated/stable $\infty$-/...) categories, together with functors $(f^\ast,f_\ast,\otimes,\underline{\mathrm{Hom}},f_!,f^!)$ satisfying several compatibilities. Notably, every second functor is the right adjoint of the previous one; pullback is symmetric monoidal; and $f_!$ satisfies general base change and projection formula. There are many examples: Etale cohomology (in various settings), but also $D$-modules, or mixed Hodge modules, or arithmetic $D$-modules, or...

However, it has been difficult to encode all the required compatibilities. Maybe the earliest formalization of the notion of a $6$-functor formalism, still in the language of triangulated categories, is due to Ayoub \cite[Definition 1.4.1]{AyoubSixFunctors}. He proved a nontrivial theorem using this formalism, that proper base change can be deduced from the other axioms in a nontrivial way. His definition contains a long list of compatibilities, and it is unclear how to upgrade this to setting where we treat $D(X)$ as a stable $\infty$-category (which leads to an implicit infinite system of higher coherences).

Our goal in this course is to first define and study abstract $6$-functor formalisms, following Mann \cite[A.5]{MannThesis}, whose work builds on the foundational work of Liu--Zheng \cite{LiuZhengArtin}; a slightly different approach is due to Gaitsgory--Rozenblyum \cite{GaitsgoryRozenblyum}. We will then set up some general machinery applying in this abstract generality that in particular reduces the proof of Poincar\'e--Verdier duality to a rather simple problem, building on work of Lu--Zheng \cite{LuZheng}. Afterwards, we will discuss many different examples of $6$-functor formalisms. A curious phenomenon is that in most cases the association $X\mapsto D(X)$ can be factored as a composite
\[
C\xrightarrow{F} \{\mathrm{analytic\ stacks}\}\xrightarrow{D_{\mathrm{qc}}} \mathrm{Cat}
\]
where the first functor $F$ takes any $X\in C$ to some other kind of geometric object $F(X)$, sometimes a scheme but often rather a stack or even an analytic stack, and the second functor is the functor of taking the derived category of quasicoherent sheaves on an analytic stack. This gives a more geometric perspective on a $6$-functor formalisms, as a functor $F$ between different kinds of geometric objects. This perspective also leads to a relation between $6$-functor formalisms and (analytic) ``ring stacks'' i.e.~(commutative) ring objects in (analytic) stacks (this was first observed by Drinfeld in relation to prismatic cohomology).

A somewhat curious case of this phenomenon is in the context of this introductory lecture:

\begin{exercise} For $X\in C$ as above a finite-dimensional locally compact Hausdorff space, consider the functor
\[
X_{\mathrm{Betti}}: \mathrm{Schemes}^{\mathrm{op}}\to \mathrm{Sets}
\]
taking any scheme $S$ to the continuous maps $|S|\to X$. Show that $X_{\mathrm{Betti}}$ is a ``pro-\'etale algebraic space'', i.e.~admits a pro-\'etale surjection from a scheme. Moreover, show that there is a natural equivalence
\[
D(X,\mathbb Z)\cong D_{\mathrm{qc}}(X_{\mathrm{Betti}}).
\]
\end{exercise}

\newpage

\section{Lecture II: Six-Functor Formalisms}

In the first part of this course, we want to develop some abstract theory of $6$-functor formalisms; the second part will then discuss many examples.

For this first part, we will work in the following setup: A category $C$, thought of as the category of geometric objects, and a class of morphisms $E$ of $C$, which will be the class of morphisms $f$ of $C$ for which the ``exceptional'' functors $f_!$ (and $f^!$) functors are defined. We will always assume the following hypotheses:\footnote{We previously did not always enforce stability under diagonals, but the theory behaves much better with that assumption.}
\begin{enumerate}
\item The category $C$ has all finite limits.
\item The class of morphisms $E$ contains all isomorphisms, and is stable under pullback, composition and diagonals.
\end{enumerate}

The naive idea of a $6$-functor formalism is the following:\footnote{In making the notions precise, we will be led to higher categorical notions, on which we will give a brief reminder below.}
\begin{enumerate}
\item An association $X\mapsto D(X)$ from $C$ to ($\infty$-)categories $D(X)$.
\item A symmetric monoidal structure $\otimes$ on $D(X)$ for each $X\in C$.
\item For each $f: X\to Y$, a pullback functor $f^\ast: D(Y)\to D(X)$, compatible with the symmetric monoidal structure, and compatible with composition of maps in $C$.
\item For each $f: X\to Y$ in $E$, a functor $f_!: D(X)\to D(Y)$, compatible with composition, and satisfying base change and projection formula isomorphisms.
\item Moreover, there should exist right adjoints to $-\otimes A$ (i.e., internal Hom's), $f^\ast$ (i.e.~$f_\ast$), and $f_!$ (i.e.~$f^!$) for $f\in E$.
\end{enumerate}

It is rather easy to formalize items (1)--(3): Indeed, this is precisely encoded in a functor
\[
C^{\mathrm{op}}\to \mathrm{CMon}(\mathrm{Cat}_\infty)
\]
where $\mathrm{Cat}_\infty$ denotes the ($\infty$-)category of $\infty$-categories, and $\mathrm{CMon}$ denotes the commutative monoids in it (for the Cartesian product), in other words $\mathrm{CMon}(\mathrm{Cat}_\infty)$ is the ($\infty$-)category of symmetric monoidal $\infty$-categories (with morphisms given by symmetric monoidal functors). Also, item (5) is just a condition. The whole difficulty in formalizing a $6$-functor formalism lies in formalizing the data in (4). Let us list some of the data one would like to have:

After defining the functor $f_!: D(X)\to D(Y)$ for each $f\in E$, one has to supply in addition:
\begin{enumerate}
\item For each $f: X\to Y$ and $g: Y\to Z$ in $E$, with composite $h: X\to Z$ thus also in $E$, an isomorphism $h_!\cong g_!f_!$ of functors $D(X)\to D(Z)$ (satisfying further associativity constraints for triple compositions).
\item For each cartesian diagram
\[\xymatrix{
X'\ar[r]^{g'}\ar[d]_{f'} & X\ar[d]^f\\
Y'\ar[r]^g & Y
}\]
with $f$ in $E$ (and thus also $f'\in E$), a base change isomorphism
\[
g^\ast f_!\cong f'_! g^{\prime\ast}
\]
of functors $D(X)\to D(Y')$. Moreover, these isomorphisms should be compatible with composing base change isomorphisms (horizontally and vertically).
\item For each $f: X\to Y$ in $E$, and $A\in D(X)$, $B\in D(Y)$, an isomorphism
\[
f_! A\otimes B\cong f_!(A\otimes f^\ast B),
\]
as functors
\[
D(X)\times D(Y)\to D(X).
\]
Moreover, this isomorphism should also be compatible with compositions in $f$ in a suitable sense, as well as the base change isomorphisms. It should also be compatible with further writing $B$ as a tensor product $B'\otimes B''$ and the symmetric monoidal structure of $f^\ast$, etc.pp.
\end{enumerate}

Especially with the projection formula, which simultaneously uses all of $\otimes$, $f^\ast$ and $f_!$, one sees that it quickly becomes tedious to explicitly write out all the compatibility isomorphisms (and associativity-type constraints) that have to be satisfied. This problem gets accentuated when $D(X)$ is actually an $\infty$-category, in which case even the associativity constraint is not just a condition, but a further datum that has to be subjected to a tower of higher associativity constraints (and similar remarks apply to the base change formula and projection formula).

In the context where $D(X)$ is an $\infty$-category, there have been at least two formalizations of the datum of a $6$-functor formalism:
\begin{enumerate}
\item By Liu--Zheng \cite{LiuZhengArtin} in the context of \'etale cohomology of schemes. Their formalization is firmly rooted in Lurie's foundational works, but it relies heavily on the combinatorics of specific simplicial sets.
\item By Gaitsgory--Rozenblyum \cite{GaitsgoryRozenblyum} in the context of coherent cohomology of schemes. They propose a very nice formal structure, but their formalization makes use of the formalism of $(\infty,2)$-categories which at the time was much less developed; in particular, they assume certain statements on faith. Fortunately, in recent years the foundations of $(\infty,2)$-categories have been developed and in particular all required statements have now been proven \cite{LoubatonRuit}.
\end{enumerate}

Very recently, Mann \cite[Appendix A.5]{MannThesis} has found a definition that combines the best of both worlds: Like Liu--Zheng's work, it is firmly rooted in Lurie's formalism, while like Gaitsgory--Rozenblyum's work it is a nice definition.

The goal of this lecture is to state Mann's definition, but first we want to say a few words about higher categories.

Very roughly, a higher category is something like a category that besides objects and morphisms also allows $2$-morphisms between morphisms, and possibly $3$-morphisms between $2$-morphisms, ad infinitum. There are basically two routes leading to higher categories:
\begin{enumerate}
\item The collection of all categories is naturally a $2$-category: The objects are categories, the morphisms are functors, and the $2$-morphisms are natural transformations. Likewise, the collection of all $2$-categories should naturally form a $3$-category, etc.
\item In homotopy theory, one studies the category of topological spaces up to homotopy. More precisely, the objects are topological spaces, while morphisms are homotopy classes of morphisms. However, passing to homotopy classes loses information -- often, it is important to know how, and not just that, two maps are homotopic. This can be recorded if instead of passing to homotopy classes, one instead introduces homotopies as $2$-morphisms. But then one is naturally led to introduce homotopies between homotopies as $3$-morphisms, and so on.
\end{enumerate}

A property of the second example is that $2$-morphisms (and all higher morphisms) are invertible (while this is not the case in the first example -- there are natural transformations that are not isomorphisms). Somewhat curiously, it turns out that it is easier to allow higher morphisms if one insists that they are all invertible. This leads to the following very vague definition.

{\bf Slogan.} For $\infty\geq n\geq m$, an $(n,m)$-category is a higher category having $i$-morphisms for $i\leq n$, which are invertible for $i>m$.

In particular through the works of Lurie (but building on previous work of Boardman--Vogt and Joyal), there is a well-developed notion of $(\infty,1)$-categories. In the literature, these are now often simply called $\infty$-categories, and we will (reluctantly) do the same. But be warned that they do not generalize $2$-categories (which are usually understood to mean $(2,2)$-categories, not $(2,1)$-categories)! Allowing non-invertible $2$-morphisms, to get a theory of $(\infty,2)$-categories, is part of the active area of higher category theory, and due to ignorance we will not say anything about it. Roughly speaking, the situation seems to be that there are now many different models of $(\infty,2)$-categories, and they have recently all been proved to be equivalent (in the relevant sense).

It turns out that the language of simplicial sets is a very convenient way to capture both the structure of categories, and the homotopy theory of topological spaces.

\begin{definition}\leavevmode
\begin{enumerate}
\item The simplex category $\Delta$ is the category of nonempty finite totally ordered sets, with weakly increasing morphisms. Equivalently, it has objects $\Delta^n\in \Delta$ for each $n\geq 0$ corresponding to the totally ordered set $\{0,1,\ldots,n\}$, with $\mathrm{Hom}(\Delta^n,\Delta^m)$ given by the set of weakly increasing maps $\{0,1,\ldots,n\}\to \{0,1,\ldots,m\}$.
\item The category of simplicial sets $\mathrm{sSet}$ is the category freely generated under colimits by $\Delta$. Equivalently, it is the functor category
\[
\Delta^{\mathrm{op}}\to \mathrm{Set}
\]
which contains $\Delta\hookrightarrow \mathrm{Fun}(\Delta^{\mathrm{op}},\mathrm{Set})$ via the Yoneda embedding.
\end{enumerate}
\end{definition}

We want to think of simplicial sets $C$ as modelling some kind of categories. Thinking of $C$ as a functor $\Delta^{\mathrm{op}}\to \mathrm{Set}$, we think of $C_n:=C(\Delta^n)$ as the set of $n$-simplices in $C$, which we think of as functors from the category $\{0\to 1\to \ldots\to n\}$ to $C$. In other words:
\begin{enumerate}
\item For $n=0$, $C_0$ is the set of objects of $C$.
\item For $n=1$, $C_1$ is the set of morphisms of $C$, i.e.~of pairs $(X,Y)$ of objects of $C$ together with a map $f: X\to Y$.
\item For $n=2$, $C_2$ is the set of commutative triangles in $C$, i.e.~of triples $(X,Y,Z)$ of $C$, maps $f: X\to Y$, $g: Y\to Z$ and $h: X\to Z$, as well as a witness that $h$ is the composite of $f$ and $g$.
\item $\ldots$
\end{enumerate}

In particular, this recipee defines a (fully faithful) functor from the (strict $1$-)category of categories towards simplicial sets.

A critical property of a category should be that composites of morphisms are defined. This means that whenever we have objects $X,Y,Z$ of $C$ as well as maps $f: X\to Y$ and $g: Y\to Z$, there is a way to extend this to a whole $2$-simplex of $C$ (and in particular a composite $h: X\to Z$). Equivalently, denoting
\[
\Lambda_1^2 = \Delta^1\sqcup_{\Delta^0} \Delta^1\subset \Delta^2
\]
the inner horn, any map $\Lambda_1^2\to C$ extends to $\Delta^2\to C$. (More generally, for any $0\leq i\leq n$, one defines the horn $\Lambda_i^n\subset \Delta^n$ obtained by removing the interior and the $i$-th face.) Moreover, composition should be unique, at least up to homotopy (which again, should be unique homotopy, ad infinitum). This condition of uniqueness up to all higher homotopies is that of being a trivial Kan fibration. This leads to the following definition.

\begin{definition}\leavevmode
\begin{enumerate}
\item A map $f: C\to D$ of simplicial sets is a trivial Kan fibration if for all $n\geq 0$ and any diagram
\[\xymatrix{
\partial\Delta^n\ar[r]\ar[d] & C\ar[d]\\
\Delta^n\ar[r] & D,
}\]
there is an extension $\Delta^n\to C$ making both triangles commute (where $\partial\Delta^n\subset \Delta^n$ is obtained by removing the interior). Equivalently, the same condition holds with $\partial\Delta^n\hookrightarrow \Delta^n$ replaced by any injection of simplicial sets.
\item An $\infty$-category is a simplicial set $C$ such that
\[
\mathrm{Map}(\Delta^2,C)\to \mathrm{Map}(\Lambda_1^2,C)
\]
is a trivial Kan fibration.
\end{enumerate}
\end{definition}

Here, for simplicial sets $C$ and $D$, $\mathrm{Map}(C,D)$ is the internal mapping object, with $n$-simplices given by $\mathrm{Hom}_{\mathrm{sSet}}(C\times \Delta^n,D)$. It turns out that $C$ is an $\infty$-category if and only if for all $0<i<n$, any map
\[
\Lambda_i^n\to C
\]
extends to the full $n$-simplex $\Delta^n$.

Let us list some basic facts and definitions about $\infty$-categories.

\begin{enumerate}
\item For any $\infty$-category $C$ and any simplicial set $D$, the mapping $\mathrm{Map}(D,C)$ is an $\infty$-category; we write $\mathrm{Fun}(D,C)$ and call it the functor $\infty$-category (especially when $D$ itself is an $\infty$-category).
\item For an $\infty$-category $C$, one can defines its homotopy category $\mathrm{Ho}(C)$. Its objects are the same as the objects of $C$, while morphisms are homotopy classes of morphisms in $C$. Here, for objects $X,Y$ of $C$ and morphisms $f,g:X\to Y$, they are homotopic if there is a $2$-simplex witnessing that $g:X\to Y$ is the composite of $f: X\to Y$ and the identity $Y\to Y$ (which is equivalent to the existence of a $2$-simplex witnessing $g: X\to Y$ as the composite of the identity $X\to X$ and $f: X\to Y$).
\item A map $f: X\to Y$ in $C$ is defined to be an isomorphism\footnote{sometimes called an equivalence} if it becomes an isomorphism in $\mathrm{Ho}(C)$. Equivalently, there is a morphism $g: Y\to X$ and $2$-simplices witnessing the identity $X\to X$ as the composite $f: X\to Y$ and $g: Y\to X$, and the identity $Y\to Y$ as the composite $g: Y\to X$ and $f: X\to Y$.
\item An $\infty$-groupoid is an $\infty$-category $C$ such that all morphisms are invertible; equivalently, $\mathrm{Ho}(C)$ is a groupoid. This is equivalent to asking that $C$ is a Kan complex, i.e.~for all $0\leq i\leq n$, $n>0$, any horn $\Lambda_i^n\to C$ extends to $\Delta^n\to C$.

At this point, recall that Kan complexes form a model for the theory of homotopy types; for any topological space $X$, one can define the Kan complex whose $n$-simplices are the continuous maps from the topological $n$-simplex towards $X$, and this captures all homotopy groups of $X$. As discussed below, Kan complexes naturally form an $\infty$-category, and we will generally to refer to them as anima when they are thought of as objects of that $\infty$-category.
\item For any $\infty$-category $C$, there is a maximal sub-$\infty$-groupoid $C^\simeq\subset C$, consisting of all objects of $C$ but restricting the morphisms to those that are isomorphisms in $C$.
\item Given objects $X,Y$ of $C$, one can form the pullback
\[\xymatrix{
\mathrm{Map}_C(X,Y)\ar[r]\ar[d] & \ast\ar[d]^{(X,Y)}\ar[d]\\
\mathrm{Fun}(\Delta^1,C)\ar[r]^{\mathrm{ev}_0,\mathrm{ev}_1} & C\times C;
}\]
intuitively speaking, $\mathrm{Map}_C(X,Y)$ is the simplicial set of maps $X\to Y$ in $C$. Then $\mathrm{Map}_C(X,Y)$ is a Kan complex. Roughly speaking, this makes $C$ into a ``category enriched in anima''. In fact, there is a natural way to treat any category enriched in anima as an $\infty$-category.
\item The $\infty$-category $\mathrm{Cat}_\infty$ of $\infty$-categories has objects $\infty$-categories, and mapping anima $\mathrm{Fun}(C,D)^{\simeq}\subset \mathrm{Fun}(C,D)$ (i.e.~one only allows invertible natural transformations as $2$-morphisms). The $\infty$-category $\mathrm{An}\subset \mathrm{Cat}_\infty$ of anima is the full sub-$\infty$-category on anima, i.e.~$\infty$-groupoids.
\end{enumerate}

Concerning the (non-set-theoretic) details, we note that everything in the above is precise, except for the implicit passage from categories enriched in Kan complexes to $\infty$-categories.\footnote{We note that with the given definitions, $\infty$-categories with mapping Kan complexes $\mathrm{Fun}^\simeq$ define a $1$-category enriched in the $1$-category of Kan complexes.} This latter procedure can be defined explicitly in terms of a homotopy-coherent nerve functor.

Another notion we will need is that of a symmetric monoidal structure on an $\infty$-category, and of (lax) symmetric monoidal functors; we will defer a more precise discussion of this to the next lecture. We only note here the following points:

\begin{enumerate}
\item In a symmetric monoidal $\infty$-category $(C,\otimes)$, one has a notion of commutative monoid, which is roughly speaking an object $X\in C$ together with a unit map $1\to X$ and a multiplication map $m: X\otimes X\to X$ together with (higher) unitality, commutativity, and associativity constraints.
\item If $C$ is an $\infty$-category admitting finite products, then $C$ has the Cartesian symmetric monoidal structure, where $X\otimes Y=X\times Y$.
\item The $\infty$-category $\mathrm{Cat}_\infty$ has finite products and thus the Cartesian symmetric monoidal structure. A symmetric monoidal $\infty$-category is a commutative monoid in $(\mathrm{Cat}_\infty,\times)$; we denote their $\infty$-category by $\mathrm{CMon}(\mathrm{Cat}_\infty)$.
\item If $(C,\otimes)$ and $(D,\otimes)$ are symmetric monoidal $\infty$-categories, a lax symmetric monoidal functor is a functor $F: C\to D$ together with (not necessarily invertible) maps $1_D\to F(1_C)$ and $F(X)\otimes F(Y)\to F(X\otimes Y)$ functorial in $X,Y\in C$, and compatible with the (higher) unitality, commutativity, and associativity constraints. A symmetric monoidal functor is one for which these natural transformations are equivalences.
\end{enumerate}

Finally, we can state Mann's definition of a $6$-functor formalism. Recall that we start with a geometric setup $(C,E)$, consisting of an ($\infty$-)category $C$ admitting finite limits, and a class of morphisms $E$ stable under pullback and composition (and containing all isomorphisms). The following definition will be made more precise in the next lecture; we note that even if $C$ is a $1$-category (as is usually the case), $\mathrm{Corr}(C,E)$ will be a $(2,1)$-category.

\begin{definition}\label{def:correspondences} The symmetric monoidal $\infty$-category of correspondences $\mathrm{Corr}(C,E)$ is given as follows.
\begin{enumerate}
\item The objects are the objects of $C$.
\item The symmetric monoidal structure is the Cartesian symmetric monoidal structure of $C$.
\item The morphisms are correspondences: $\mathrm{Hom}_{\mathrm{Corr}(C,E)}(X,Y)$ is given by the ($\infty$-)groupoid of objects $W\in C$ together with maps
\[\xymatrix{
& W \ar[dl]^f\ar[dr]^g\\
X && Y
}\]
where $g\in E$.
\item The composition of morphisms is given by the composition of correspondences, i.e.~the composite of two correspondences $X\leftarrow W\to Y$ and $Y\leftarrow W'\to Z$ is given by the long triangle in the diagram
\[\xymatrix{
&& W\times_Y W'\ar[dl]\ar[dr]\\
& W\ar[dl]\ar[dr] && W'\ar[dl]\ar[dr] \\
X && Y && Z.
}\]
\end{enumerate}
\end{definition}

As in the naive discussion above, we will first formalize just three functors $\otimes$, $f^\ast$ and $f_!$. In the following definition, $\mathrm{Cat}_\infty$ is endowed with the Cartesian symmetric monoidal structure.

\begin{definition}\label{def:3functors} A $3$-functor formalism is a lax symmetric monoidal functor
\[
D: \mathrm{Corr}(C,E)\to \mathrm{Cat}_\infty.
\]
\end{definition}

The reader should take a moment to appreciate how concise this definition is! The claim is that this encodes the $3$ functors $\otimes$, $f^\ast$, $f_!$, and (all of) their relations:

\begin{enumerate}
\item On objects, $D$ defines an association $X\mapsto D(X)$.
\item The lax symmetric monoidal structure defines a natural ``exterior tensor product'' functor $D(X)\otimes D(X)\to D(X\times X)$. Together with the (diagonal) pullback defined just below, this defines the tensor product $\otimes$ on $D(X)$.
\item For any map $f: X\to Y$, the correspondence $Y\xleftarrow{f} X=X$ defines the pullback functor $f^\ast: D(Y)\to D(X)$. 
\item For any map $f: X\to Y$ in $E$, the correspondence $X=X\xrightarrow{g} Y$ defines the functor $f_!: D(X)\to D(Y)$.
\end{enumerate}

In particular, a correspondence $X\xleftarrow{f} W\xrightarrow{g} Y$ gets sent to the functor $g_! f^\ast: D(X)\to D(Y)$. The compatibility of this with composition amounts to the base change formula. Next time, we will also discuss how to find the projection formula.

Finally, as promised we define a $6$-functor formalism.

\begin{definition}\label{def:6functors} A $6$-functor formalism is a $3$-functor formalism
\[
D: \mathrm{Corr}(C,E)\to \mathrm{Cat}_\infty
\]
for which the functors $-\otimes A$, $f^\ast$ and $f_!$ admit right adjoints.
\end{definition}

We note that no further coherences are necessary here: Adjoints automatically acquire all relevant coherences.

\newpage

\section{Lecture III: Symmetric monoidal $\infty$-categories}

Let us recall the definition of $3$-functor formalisms from the last lecture. Given a geometric setting, given by an $\infty$-category $C$ admitting finite limits, together with some class $E$ of morphisms of $C$ that is stable under pullback, composition and diagonals (and contains all isomorphisms), one defines a symmetric monoidal $\infty$-category $\mathrm{Corr}(C,E)$ of correspondences.

\begin{definition} A $3$-functor formalism is a lax symmetric monoidal functor
\[
D: \mathrm{Corr}(C,E)\to \mathrm{Cat}_\infty.
\]
\end{definition}

The goal of this lecture is to unpack this definition, and in particular give some background on symmetric monoidal $\infty$-categories and (lax) symmetric monoidal functors.

But first, let us give an honest definition of $\mathrm{Corr}(C,E)$ as an $\infty$-category. For any $n\geq 0$, let
\[
(\Delta^n)^2_+\subset (\Delta^n)^{\mathrm{op}}\times \Delta^n
\]
be the subset spanned by those simplices $(i,j)\in \{0,1,\ldots,n\}^2$ with $i\geq j$. For example, for $n=2$, this corresponds to the category
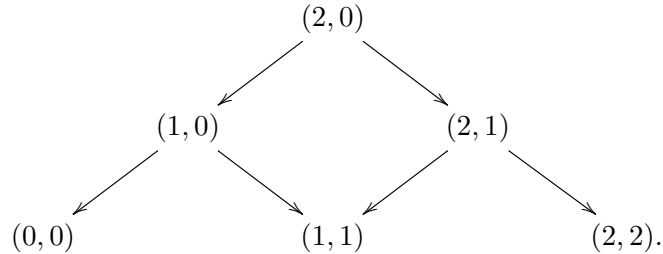
\[\xymatrix{
&& (2,0)\ar[dr]\ar[dl] \\
& (1,0)\ar[dr]\ar[dl] && (2,1)\ar[dr]\ar[dl]\\
(0,0) && (1,1) && (2,2).
}\]
Varying $n$, this defines a cosimplicial category $(\Delta^\bullet)^2_+$.

\begin{defprop} The correspondence $\infty$-category $\mathrm{Corr}(C,E)$ is the simplicial set whose $n$-simplices are maps from $(\Delta^n)^2_+$ to $C$ with the following two properties:
\begin{enumerate}
\item All arrows going down-right are in $E$.
\item All small squares are Cartesian.
\end{enumerate}
\end{defprop}

One has to prove that this is, indeed, an $\infty$-category, i.e.~that all inner horns $\Lambda_i^n\subset \Delta^n$ can be filled, $0<i<n$. For $i=1$, $n=2$, this amounts to completing a diagram
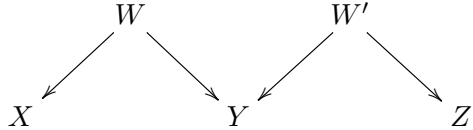
\[\xymatrix{
& W\ar[dl]\ar[dr] && W'\ar[dl]\ar[dr]\\
X && Y && Z
}\]
in $C$, with $W\to Y$ and $W'\to Z$ in $E$, to a diagram
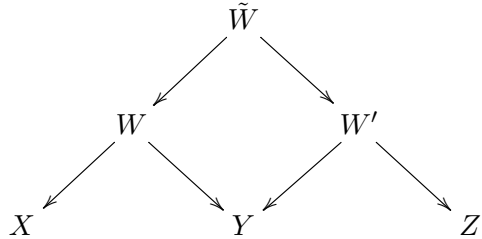
\[\xymatrix{
&& \tilde{W}\ar[dl]\ar[dr]\\
& W\ar[dl]\ar[dr] && W'\ar[dl]\ar[dr]\\
X && Y && Z
}\]
such that the square is cartesian, and also $\tilde{W}\to W'$ is in $E$. But this can be filled in, as $C$ has pullbacks, and $E$ is stable under pullbacks. The general assertion is \cite[Lemma 6.1.2]{LiuZhengArtin}.

Now we want to endow $\mathrm{Corr}(C,E)$ with a symmetric monoidal structure. This requires a digression on the general notion of symmetric monoidal $\infty$-categories.

Recall that classically, a symmetric monoidal category is a category $C$ together with:
\begin{enumerate}
\item A unit object $1_C\in C$;
\item A ``tensor product'' $-\otimes-: C\times C\to C$;
\item A ``unitality constraint'' $u_X: 1\otimes X\cong X$ functorially in $X$;
\item A ``commutativity constraint'' $c_{X,Y}: X\otimes Y\cong Y\otimes X$ functorially in $X,Y$;
\item An ``associativity constraint'' $a_{X,Y,Z}: (X\otimes Y)\cong Z\cong X\otimes (Y\otimes Z)$ functorially in $X,Y,Z$,
\end{enumerate}
subject to a number of commutative diagrams, like the pentagon axiom and the hexagon axiom. What these axioms express is that for all finite sets $I$, and all objects $X_i\in C$ enumerated by $i\in I$, there is a well-defined
\[
\bigotimes_{i\in I} X_i\in C.
\]
Of course, one way to obtain this is to order $I\cong \{1,\ldots,n\}$ and then inductively define the tensor product. Using the commutativity and associativity constraint, any two such choices can be related by an isomorphism, and the extra axioms ensure that if one goes around in a circle using these constraints, one will always end up with the identity. For example, the pentagon axioms says that, when filled in with the evident associativity constraints, the diagram
\[\xymatrix{
((X\otimes Y)\otimes Z)\otimes W\ar[r]\ar[d] & (X\otimes (Y\otimes Z))\otimes W\ar[r] & X\otimes ((Y\otimes Z)\otimes W)\ar[d]\\
(X\otimes Y)\otimes (Z\otimes W)\ar[rr] && X\otimes (Y\otimes (Z\otimes W))
}\]
commutes.

Clearly, it would quickly become tedious to adapt this definition to higher categories, and a more conceptual approach is required. But the idea should just be that for any finite set $I$ and any collection of objects $X_i\in C$ for $i\in I$, one can unambiguously define $\bigotimes_{i\in I} X_i$.

Let us first take a step back, and use this perspective to define commutative monoids in $\infty$-categories. Again, in a commutative monoid $X$, addition should give well-defined ``sum all elements'' maps from $X^I$ to $X$ for any finite set $I$. 

\begin{definition} A commutative monoid in an $\infty$-category $C$ is a functor
\[
X: \mathrm{Fin}^{\mathrm{part}}\to C
\]
from the category of finite sets with partially defined maps, such that for all finite sets $I$ the map
\[
X(I)\to \prod_{i\in I} X(\{i\}) = X(\ast)^I
\]
is an isomorphism (where the map is induced by the partially defined maps $I\to \{i\}$ sending $i$ to $i$, and the rest nowhere).
\end{definition}

Note that on objects, $X$ is determined by $X(\ast)$; indeed, $X(I)\cong X(\ast)^I$ by the condition. Thus, a commutative monoid $X$ is given by extra structure on $X(\ast)$, and we will sometimes abuse notation and write $X=X(\ast)$. The map $\emptyset\to \ast$ in $\mathrm{Fin}^{\mathrm{part}}$ defines a unit map $\ast=X(\emptyset)\to X(\ast)$; while for any finite set $I$, the projection map $I\to \ast$ defines a map $X(\ast)^I\cong X(I)\to X(\ast)$ ``summing all elements''. For a general partially defined map $f: I\dasharrow J$, the induced map $X^I=X(I)\to X^J=X(J)$ is on the $j$-th coordinate given by summing over $\prod_{i\in f^{-1}(j)} X_i$.

This definition can, in particular, be applied to $C=\mathrm{Cat}_\infty$, leading to the first definition of a symmetric monoidal $\infty$-category.

\begin{definition} A symmetric monoidal $\infty$-category is a commutative monoid in $\mathrm{Cat}_\infty$.
\end{definition}

Thus, giving a symmetric monoidal $\infty$-category requires us to write down functors
\[
\mathrm{Fin}^{\mathrm{part}}\to \mathrm{Cat}_\infty.
\]
It turns out that it is, in general, hard to write down functors towards $\mathrm{Cat}_\infty$ directly. This is an issue that already occurs in classical category theory, where it was solved by Grothendieck. A standard situation in algebraic geometry is to consider the functor
\[
\mathrm{Sch}^{\mathrm{op}}\to \mathrm{Cat}
\]
taking any scheme $S$ to the category $\mathrm{QCoh}(S)$ of quasicoherent sheaves on $S$ (or in the context of the first lecture, the functor
\[
\mathrm{Top}^{\mathrm{op}}\to \mathrm{Cat}
\]
taking any topological space $X$ to the category $\mathrm{Ab}(X)$ of abelian sheaves on $X$). Any map $f: S'\to S$ defines a pullback functor
\[
f^\ast: \mathrm{QCoh}(S)\to \mathrm{QCoh}(S').
\]
But when one composes two morphisms, $(fg)^\ast$ is not literally identical to $g^\ast f^\ast$. Instead, one only has a natural isomorphism between the two, and of course this needs to be subject to a coherence axiom.

Grothendieck's solution was to write down certain categorical fibrations instead. Namely, one can define the category
\[
\mathrm{QCoh}_{\mathrm{Sch}} = \{(S\in \mathrm{Sch}, \mathcal M\in \mathrm{QCoh}(S))\}\to \mathrm{Sch}
\]
consisting of pairs $(S,\mathcal M)$ of a scheme $S$ and a quasicoherent sheaf $\mathcal M$ on $S$; and where morphisms $(S',\mathcal M')\to (S,\mathcal M)$ are pairs of a morphism $f: S'\to S$ and a morphism $f^\ast \mathcal M\to \mathcal M'$. It is straightforward to define this category, and a functor to $\mathrm{Sch}$ which is a Cartesian fibration (for whose definition see below). But then a theorem says that Cartesian fibrations over a category $C$ (e.g.~$C=\mathrm{Sch}$) are equivalent to functors $C^{\mathrm{op}}\to \mathrm{Cat}$, giving the desired functor.

The following discussion can be found in \cite[Chapter 2, 3]{LurieHTT}, see in particular \cite[Proposition 2.4.2.8]{LurieHTT} for the equivalence of the following definition with other definitions.

\begin{definition} A functor $F: D\to C$ of $\infty$-categories is a coCartesian\footnote{This notion is dual to Cartesian fibrations, and classifies covariant functors to $\mathrm{Cat}_\infty$ as opposed to contravariant ones.} fibration if it is an inner fibration (i.e.~any inner horn can be lifted), and
\begin{enumerate}
\item Any morphism of $C$ admits a locally coCartesian lift;
\item Composites of locally coCartesian lifts are locally coCartesian.
\end{enumerate}
Here, a morphism $g: Y\to Y'$ in $D$ is a locally coCartesian lift of a morphism $f: X\to X'$ if it is initial among all morphisms with source $Y$ that lift $f$.
\end{definition}

\begin{theorem}[Lurie, ``Straightening/Unstraightening''] There is a natural equivalence between the $\infty$-categories of functors $C\to \mathrm{Cat}_\infty$ and the $\infty$-category of coCartesian fibrations over $C$.
\end{theorem}

One has to be a bit careful here which morphisms of coCartesian fibrations one allows. Indeed, one can already observe that functors $C\to \mathrm{Cat}_\infty$ form most naturally an $(\infty,2)$-category, and the above theorem restricts to the underlying $(\infty,1)$-category here (allowing only invertible natural transformations). But coCartesian fibrations know about the non-invertible transformations. Namely, consider two coCartesian fibrations $F: D\to C$, $F': D'\to C$ over $C$, and a functor $G: D\to D'$ over $C$. For each object $X$ of $C$, this induces a functor $G_X: D_X\to D'_X$ between the fibres. Moreover, for any map $f: X\to X'$, one can look at the diagram
\[\xymatrix{
D_X\ar[r]^{D_f}\ar[d]_{G_X} & D_{X'}\ar[d]^{G_{X'}}\\
D'_X\ar[r]^{D'_f} & D'_{X'},
}\]
where $D_f$ and $D'_f$ are the induced functors (of locally coCartesian lifts). This diagram may not commute, but the initiality condition on locally coCartesian lifts gives a natural transformation
\[
D'_f G_X\to G_{X'} D_f.
\]
This map is an isomorphism if and only if $G$ preserves locally coCartesian lifts.

In practice, all functors to $\mathrm{Cat}_\infty$ are constructed by constructing the corresponding coCartesian fibration instead. With this in mind, we redefine symmetric monoidal $\infty$-categories:

\begin{definition}[{\cite[Definition 2.0.0.7]{LurieHA}}]\label{def:symmetricmonoidal2} A symmetric monoidal $\infty$-category is a coCartesian fibration
\[
C^\otimes\to \mathrm{Fin}^{\mathrm{part}}
\]
such that, denoting $C=C^\otimes_\ast$ the fibre over the one-element set $\ast\in \mathrm{Fin}^{\mathrm{part}}$, for all finite sets $I$ the functor
\[
C^\otimes_I\to \prod_{i\in I} C,
\]
induced by the partially defined maps $I\dasharrow \{i\}$ sending $i$ to $i$ (and the rest nowhere), is an equivalence.
\end{definition}

Starting with a symmetric monoidal $\infty$-category $(C,\otimes)$ in the old sense, $C^\otimes$ is roughly given as follows. Objects of $C^\otimes$ are pairs $(I,(X_i)_{i\in I})$ of a finite set $I$ and objects $X_i\in C$ for $i\in I$. A map
\[
(I,(X_i)_{i\in I})\to (J,(Y_j)_{j\in J})
\]
in $C^\otimes$ is given by a partially defined map $f: I\dasharrow J$ together with maps $\bigotimes_{i\in f^{-1}(j)} X_i\to Y_j$ for all $j\in J$.

Finally, let us discuss (lax) symmetric monoidal functors. Recall that a lax symmetric monoidal functor $F: (C,\otimes)\to (D,\otimes)$ is, intuitively speaking, a functor $F: C\to D$ together with natural maps $\bigotimes_{i\in I} F(X_i)\to F(\bigotimes_{i\in I} X_i)$ for any finite $I$ and objects $X_i$, $i\in I$. A symmetric monoidal functor is a lax symmetric monoidal functor for which these maps are all isomorphisms.

It turns out that lax symmetric monoidal functors are best defined using this perspective of coCartesian fibrations $C^\otimes\to \mathrm{Fin}^{\mathrm{part}}$.

\begin{definition}\label{def:laxsymmetricmonoidal} Let $(C,\otimes)$ and $(D,\otimes)$ be symmetric monoidal $\infty$-categories, with corresponding coCartesian fibrations
\[
C^\otimes,D^\otimes\to \mathrm{Fin}^{\mathrm{part}}.
\]
A lax symmetric monoidal functor $(C,\otimes)\to (D,\otimes)$ is a functor $F^\otimes: C^\otimes\to D^\otimes$ over $\mathrm{Fin}^{\mathrm{part}}$ such that $F^\otimes$ preserves locally coCartesian lifts of the morphisms $I\dasharrow \{i\}$ sending $i\in I$ to $i$ and the rest nowhere. The functor $F^\otimes$ is symmetric monoidal if it preserves all locally coCartesian lifts.
\end{definition}

The condition on locally coCartesian lifts of the morphisms $I\dasharrow \{i\}$ is ensuring that $F^\otimes_I: C^\otimes_I\to D^\otimes_I$ is given by $F^I: C^I\to D^I$ under the equivalences $C^\otimes_I\cong C^I$, $D^\otimes_I\cong D^I$. Then the above remarks on maps between coCartesian fibrations is precisely ensuring that $F^\otimes$ induces maps $\bigotimes_{i\in I} F(X_i)\to F(\bigotimes_{i\in I} X_i)$.

\begin{remark} Arguably, a drawback of the definition of symmetric monoidal $\infty$-categories in terms of coCartesian fibrations is that the definition is not evidently selfdual -- it is not clear that if $(C,\otimes)$ is a symmetric monoidal $\infty$-category then so is $(C^{\mathrm{op}},\otimes)$. There is a dual notion of colax symmetric monoidal functors, and this is nontrivial to define in this language. See however \cite{HaugsengHebestreitLinskensNuiten} for a detailed discussion of this duality.
\end{remark}

If $C$ is any $\infty$-category that admits finite products, it acquires a unique symmetric monoidal $\infty$-category structure whose operation is the cartesian product.\footnote{This is actually easier to see for coproducts.} In particular, this applies to $\mathrm{Cat}_\infty$. The cartesian symmetric monoidal structure has a useful universal property.

\begin{theorem}[{\cite[Proposition 2.4.1.7]{LurieHA}}] Let $(C,\otimes)$ be any symmetric monoidal $\infty$-category, and let $D$ be an $\infty$-category admitting finite products, with its cartesian symmetric monoidal structure. Then lax symmetric monoidal functors $(C,\otimes)\to (D,\times)$ are equivalent to functors
\[
F: C^\otimes\to D
\]
such that for all finite sets $I$ and objects $X_i\in C$, $i\in I$, the map
\[
F((I,(X_i)_{i\in I}))\to \prod_{i\in I} F(X_i)
\]
is an equivalence (the map induced as usual by the forgetful maps $I\dasharrow \{i\}$ for $i\in I$).
\end{theorem}

To finish the discussion of the terms appearing in the definition of a $3$-functor formalism, it remains to define $\mathrm{Corr}(C,E)$ as a symmetric monoidal $\infty$-category. Conveniently, $\mathrm{Corr}(C,E)^\otimes$ can be defined itself as a correspondence category in
\[
C^{\times'} := ((C^{\mathrm{op}})^{\sqcup})^{\mathrm{op}}.
\]
Here $(C^{\mathrm{op}})^{\sqcup}$ is the coCartesian fibration corresponding to the symmetric monoidal $\infty$-category $C^{\mathrm{op}}$ with coproducts (i.e., products in $C$). Concretely, $C^{\times'}$ as objects given by pairs $(I,(X_i)_{i\in I})$ as usual, but maps
\[
(I,(X_i)_{i\in I})\to (J,(Y_j)_{j\in J})
\]
are given by partially defined maps $f:J\dasharrow I$ together with maps
\[
X_i\to \prod_{j\in f^{-1}(i)} Y_j
\]
for all $i\in I$ (instead of maps in the opposite direction). Thus, $C^{\times'}$ is not the coCartesian fibration over $\mathrm{Fin}^{\mathrm{part}}$ encoding the cartesian symmetric monoidal structure on $C$.

Let $E^\times$ be the class of morphisms of $C^{\times'}$ that lie over identity morphisms of $\mathrm{Fin}^{\mathrm{part}}$, and where all maps $Y_i\to X_i$ are in $E$.

\begin{defprop}[{\cite[Proposition 6.1.3]{LiuZhengArtin}}] Let $\mathrm{Corr}(C,E)^\otimes = \mathrm{Corr}(C^{\times'},E^\times)$. This is naturally a coCartesian fibration over $\mathrm{Fin}^{\mathrm{part}}$, defining a symmetric monoidal structure on $\mathrm{Corr}(C,E)$.
\end{defprop}

Objects of $\mathrm{Corr}(C,E)^\otimes$ are again pairs $(I,(X_i)_{i\in I})$. Morphisms $(I,(X_i)_{i\in I})\to (J,(Y_j)_{j\in J})$ are given by partially defined maps $f:I\dasharrow J$ together with correspondences
\[
\prod_{i\in f^{-1}(j)} X_i\leftarrow W_j\to Y_j
\]
for all $j\in J$, where the morphism $W_j\to Y_j$ lies in $E$. We note that we needed the contravariant maps in $C^{\times'}$ to cancel the contravariance of the first map in the correspondences. We note in particular the following three kinds of morphisms of $\mathrm{Corr}(C,E)^\otimes$:
\begin{enumerate}
\item For all $X\in C$, a map $(\{1,2\},(X,X))\to (\ast,X)$ given by the projection $\{1,2\}\to \ast$ and the correspondence $X\times X\leftarrow X=X$ where the first map is the diagonal.
\item For any map $f: X\to X'$ in $C$, a map $(\ast,X')\to (\ast,X)$ given by $\ast=\ast$ and the correspondence $X'\leftarrow X=X$.
\item For any map $f: X\to X'$ in $E$, a map $(\ast,X)\to (\ast,X')$ given by $\ast=\ast$ and the correspondence $X=X\to X'$.
\end{enumerate}

Finally, we can restate the definition of a $3$-functor formalism.

\begin{definition} A $3$-functor formalism is a lax symmetric monoidal functor $\mathrm{Corr}(C,E)\to \mathrm{Cat}_\infty$. Equivalently, it is a functor
\[
D: \mathrm{Corr}(C,E)^\otimes\to \mathrm{Cat}_\infty
\]
such that for all finite sets $I$ and $X_i\in C$, $i\in I$, the functor
\[
D((I,(X_i)_{i\in I}))\to \prod_{i\in I} D(X_i)
\]
is an equivalence.
\end{definition}

In particular, on objects this gives $X\mapsto D(X)$, and the three types of morphisms above induce the tensor product, the pullback $f^\ast$, and the exceptional pushforward $f_!$. In general, a map
\[
(I,(X_i)_{i\in I})\to (J,(Y_j)_{j\in J})
\]
given by $f:I\dasharrow J$ and correspondences
\[
\prod_{i\in f^{-1}(j)} X_i\xleftarrow{f_j} W_j\xrightarrow{g_j} Y_j
\]
induces the functor
\[
\prod_{i\in I} D(X_i)\to \prod_{j\in J} D(Y_j)
\]
whose $j$-th component is given by $g_{j!} f_j^\ast (\boxtimes_{i\in f^{-1}(j)})$ where
\[
\boxtimes_{i\in f^{-1}(j)}: \prod_{i\in f^{-1}(j)} D(X_i)\to D(\prod_{i\in f^{-1}(j)} X_i)
\]
denotes the exterior tensor product (i.e., the tensor product of the pullbacks from each factor). Indeed, this easily follows by writing this correspondence as a composite of smaller correspondences. Now the compatibility with composites of correspondences is encoding all the desired compatibilities between tensor, pullback, and exceptional pushforward.

For example, we can get the projection formula
\[
f_! A\otimes B\cong f_!(A\otimes f^\ast B)
\]
along a morphism $f: X\to Y$ in $E$, by writing the correspondence
\[
(\{1,2\},(X,Y))\to (\ast,Y)
\]
given by $X\times Y\leftarrow X\to Y$ as a composite in two ways. First, as the composite of
\[
(\{1,2\},(X,Y))\to (\{1,2\},(Y,Y))
\]
given by the correspondences $X=X\to Y$ and $Y=Y=Y$, and
\[
(\{1,2\},(Y,Y))\to (\ast,Y)
\]
given by the correspondence $Y\times Y\leftarrow Y=Y$; and as the composite of
\[
(\{1,2\},(X,Y))\to (\ast,X)
\]
given by the correspondence $X\times Y\leftarrow X=X$, and
\[
(\ast,X)\to (\ast,Y)
\]
given by the correspondence $X=X\to Y$. Here, the first composite unravels to
\[
(A,B)\mapsto (f_!A,B)\mapsto f_! A\otimes B,
\]
and the second composite to
\[
(A,B)\mapsto A\otimes f^\ast B\mapsto f_!(A\otimes f^\ast B).
\]

\begin{remark} In some other formulations of $6$-functor formalisms, it is stressed that the functor $f_!: \mathcal D(X)\to \mathcal D(Y)$ should be $\mathcal D(Y)$-linear (where $\mathcal D(Y)$ acts on $\mathcal D(X)$ by tensoring with the pullback); this is a ``homotopy coherent'' version of the projection formula. This structure of $f_!$ as a $\mathcal D(Y)$-linear functor is in fact encoded in the lax symmetric monoidal functor $\mathcal D: \mathrm{Corr}(C,E)\to \mathrm{Cat}_\infty$. In fact, lax symmetric monoidal functors preserve algebras and modules, and the relevant structure can be defined already on $\mathrm{Corr}(C,E)$. Namely, any $X$ defines a commutative algebra object in $C^{\mathrm{op}}$ (this is always true for the disjoint union symmetric monoidal structure, as there is a canonical map $X\sqcup X\to X$), and any map $f: X\to Y$ defines a map $f^\ast$ of commutative algebras in $C^{\mathrm{op}}$, and in particular makes $X$ a module over $Y$. The symmetric monoidal functor $C^{\mathrm{op}}\to \mathrm{Corr}(C,E)$ induces the same structure in $\mathrm{Corr}(C,E)$. Now any map $f: X\to X'$ over $Y$ in $E$ induces a covariant map of the associated modules over $Y$ in $\mathrm{Corr}(C,E)$. Thus, the same structure is present after applying $\mathcal D$, making $f_!: \mathcal D(X)\to \mathcal D(X')$ a map of $\mathcal D(Y)$-modules, i.e.~$\mathcal D(Y)$-linear. In particular, taking $X'=Y$, the functor $f_!: \mathcal D(X)\to \mathcal D(Y)$ is naturally $\mathcal D(Y)$-linear.
\end{remark}

In the next lecture, we will discuss a construction principle for $3$-functor formalisms, showing how to get all these intricate coherences in practice.\newpage

\section{Lecture IV: Construction of Six-Functor Formalisms}

Let us again fix a geometric setting $(C,E)$ as before. We are interested in constructing a lax symmetric monoidal functor
\[
\mathcal D: \mathrm{Corr}(C,E)\to \mathrm{Cat}_\infty.
\]
Equivalently, this is described by a functor
\[
\mathcal D: \mathrm{Corr}(C,E)^\otimes\to \mathrm{Cat}_\infty
\]
such that for any finite set $I$ and $X_i\in C$ for $i\in I$, the functor
\[
\mathcal D((I,(X_i)_i))\to \prod_{i\in I} \mathcal D(X_i)
\]
is an equivalence.

In practice, it is easy to construct a functor
\[
\mathcal D_0: C^{\mathrm{op}}\to \mathrm{CMon}(\mathrm{Cat}_\infty)
\]
to symmetric monoidal $\infty$-categories, encoding $\otimes$ and $f^\ast$. Note that $\mathcal D_0$ is equivalent to a lax symmetric monoidal functor
\[
C^{\mathrm{op}}\to \mathrm{Cat}_\infty
\]
by \cite[Theorem 2.4.3.18]{LurieHA}. The problem is now to extend this from $C^{\mathrm{op}}\cong \mathrm{Corr}(C,\mathrm{isom})$ to $\mathrm{Corr}(C,E)$.

In practice (at least in contexts of algebraic geometry), there are two special classes of morphisms $I,P\subset E$ of ``open immersions'' and ``proper'' maps such that for $f\in I$, the functor $f_!$ is the left adjoint of $f^\ast$, while for $f\in P$, the functor $f_!$ is the right adjoint of $f^\ast$. Moreover, any $f\in E$ admits a factorization $f=\overline{f} j$ with $j\in I$ and $\overline{f}\in P$, i.e.~a compactification. As $f_!$ should be compatible with composition, we need to have $f_! = \overline{f}_! j_!$ where the latter two functors are defined (as the right resp.~left adjoint of pullback). It is, however, far from clear that this is well-defined, especially as a functor of $\infty$-categories!

\begin{example}\leavevmode
\begin{enumerate}
\item If $C$ is the category of locally compact Hausdorff topological spaces, one can take for $I$ the open immersions, and for $P$ the proper maps (i.e., preimages of compact subsets are compact). Any map $f: X\to Y$ admits a compactification $X\hookrightarrow \overline{X}\to Y$, for example $\overline{X}=\beta X\times_{\beta Y} Y$ using the Stone-\v{C}ech compactification.
\item If $C$ is the category of qcqs schemes and $E$ is the class of separated morphisms of finite type, then one can again take for $I$ the open immersions and for $P$ the proper maps. Indeed, the Nagata compactifion theorem ensures that any separated map of finite type $f: X\to Y$ between qcqs schemes admits a compactification $X\hookrightarrow \overline{X}\to Y$.
\end{enumerate}
\end{example}

One might expect that in order for $f_!$ to be defined uniquely as a functor of $\infty$-categories, one has to put some rather strong assumptions on $C$ (and the classes $E$, $I$, $P$) -- for example, the existence of some canonical compactification, or at least a canonical ``cofinal'' collection of such. Somewhat surprisingly, it turns out that minimal hypotheses ensure that $f_!$ is canonically defined. In fact, we will only make the following assumptions. (Only the assumptions on $I$ and $P$ are not obviously necessary, but they are very weak and satisfied in all practical situations.)

\begin{enumerate}
\item Assumptions on $I$ and $P$: The classes of morphisms $I$ and $P$ are stable under pullback, composition and diagonals and contain all isomorphisms. If $f\in I\cap P$, then $f$ is $n$-truncated for some $n$.\footnote{This is automatic if $C$ is a category (as opposed to an $\infty$-category). It means that for any $g: Z\to X$, the anima of lifts of $g$ to $Y$ is $n$-truncated, i.e.~has $\pi_i=0$ for $i>n$.} Any $f\in E$ is a composite $f=\overline{f}j$ with $j\in I$ and $\overline{f}\in P$.
\item For any $f\in I$, the functor $f^\ast$ admits a left adjoint $f_!$ which satisfies base change and the projection formula.
\item For any $f\in P$, the functor $f^\ast$ admits a right adjoint $f_\ast$ which satisfies base change and the projection formula.
\item For any Cartesian diagram
\[\xymatrix{
X'\ar[r]^{j'}\ar[d]^{g'} & X\ar[d]^g\\
Y'\ar[r]^j & Y
}\]
with $j\in I$ (hence $j'\in I$) and $g\in P$ (hence $g'\in P$), the natural map $j_! g'_\ast\to g_\ast j'_!$ is an isomorphism.
\end{enumerate}

\begin{remark}\label{rem:lastconditionvacuousexcision} In the situation of (4), assume that $j$ is a monomorphism. Then this follows from (2) and (3).\footnote{This was observed by Heyer--Mann \cite[Lemma 1.2.6]{HeyerMann}, improving on a weaker statement originally made in these lectures.} Indeed, by (2) and $j$ being a monomorphism, we have $j^\ast j_!=\mathrm{id}$ (and $j'^\ast j'_! =\mathrm{id}$); in particular $j'^\ast$ is essentially surjective. Thus, it suffices to see that the map
\[
j_! g'_\ast j'^\ast\to g_\ast j'_! j'^\ast
\]
is an isomorphism. In fact, it is a priori clear that the map in (4) is an isomorphism after applying $j^\ast$ (using base change from (2) and (3), and noting that after pullback along $j$ the cartesian square has isomorphisms as horizontal maps); so it is enough to show that the functor $g_\ast j'_! j'^\ast$ takes values in the essential image of $j_!$. (Via this reduction, we do not have to track commutative diagrams in the following.) But $j'_! j'^\ast M\cong j'_! 1\otimes M$ by the projection formula in (2), and $j'_! 1\cong g^\ast j_! 1$ by the base change in (2). Thus
\[
g_\ast j'_! j'^\ast M\cong g_\ast(g^\ast j_! 1\otimes M)\cong j_! 1\otimes g_\ast M\cong j_!j^\ast g_\ast M
\]
using the projection formula from (3) for $g$ and from (2) for $j$. 
\end{remark}

We note that in (2), (3) and (4), one is only asking that certain natural maps (defined by adjunctions) are isomorphisms. For example, the last map $j_! g'_\ast\to g_\ast j'_!$ is by definition adjoint to
\[
g^\ast j_! g'_\ast\cong j'_! g^{\prime\ast} g'_\ast\to j'_!
\]
using the base change isomorphism for $j_!$ from (2).

In fact, in (4) one would expect to have a similar isomorphism $j_! g'_\ast\cong g_\ast j'_!$ even if the diagram is not Cartesian. However, in that case one cannot a priori define a natural isomorphism. Fortunately, it turns out that this is a consequence of the axioms. To see this, we make the following construction.

\begin{construction} For any $f\in I\cap P$, there is a natural isomorphism $f_!\cong f_\ast$ between the left and the right adjoint of $f^\ast$.

Indeed, we argue by induction on $n$ such that $f$ is $n$-truncated. If $n=-2$, then $f$ is an isomorphism, and the claim is clear. In general, consider the cartesian diagram
\[\xymatrix{
X\times_Y X\ar[r]^g\ar[d]^h & X\ar[d]^f\\
X\ar[r]^f & Y
}\]
as well as $\Delta: X\to X\times_Y X$. Then $g\in I\cap P$ by pullback-stability, and hence $\Delta\in I\cap P$ as it is a map between $X$ and $X\times_Y X$, which have compatible projections to $X$ that are in $I\cap P$. Moreover, $\Delta$ is $n-1$-truncated, so by induction we have constructed an identification $\Delta_!\cong \Delta_\ast$. Now
\[
f_! = f_! g_\ast \Delta_\ast\cong f_! g_\ast \Delta_!\cong f_\ast h_! \Delta_! = f_\ast
\]
using respectively $g \Delta = \mathrm{id}_X$, the identification $\Delta_!\cong \Delta_\ast$, assumption (4), and $h\Delta = \mathrm{id}_X$.
\end{construction}

\begin{construction} For any commutative diagram
\[\xymatrix{
X'\ar[r]^{j'}\ar[d]^{g'} & X\ar[d]^g\\
Y'\ar[r]^j & Y
}\]
with $j,j'\in I$ and $g,g'\in P$, there is a natural isomorphism $j_!g'_\ast\cong g_\ast j_!$.

Indeed, consider the induced map $h: X'\to X\times_Y Y'$ with $j'': X\times_Y Y'\to X$ and $g'': X\times_Y Y'\to Y'$. Then $h\in I$ as it is a map between $X'$ and $X\times_Y Y'$ both of whose projections to $X$ are in $I$; and $h\in P$ as it is a map between $X'$ and $X\times_Y Y'$ both of whose projections to $Y'$ are in $P$. Thus, the previous construction gives us an isomorphism $h_!\cong h_\ast$. Now
\[
j_!g'_\ast= j_! g''_\ast h_\ast\cong j_! g''_\ast h_!\cong g_\ast j''_! h_! = g_\ast j'_!
\]
using respectively $g'=g''h$, the identification $g_!\cong g_\ast$, assumption (4), and $j''h=j'$.
\end{construction}

Now we can also see that if $f\in E$ has two different compactifications, then the two induced possible definitions of $f_!$ can be identified.

\begin{construction} Let $f: X\to Y$ in $E$ be written in two ways as $f=\overline{f} j$ and $\overline{f}' j'$, with $j,j'\in I$ and $\overline{f},\overline{f}'\in P$. Then there is a natural isomorphism
\[
\overline{f}_\ast j_!\cong \overline{f}'_\ast j'_!.
\]

Indeed, let $\overline{X}$ and $\overline{X}'$ be the implicit compactifications of $X$ over $Y$, and consider $Z=\overline{X}\times_Y \overline{X}'$. The projection to $Y$ lies in $P$, and the diagonal map $X\to X\times_Y X\to \overline{X}\times_Y \overline{X}'$ is a morphism in $E$ (one can check that morphisms of $E$ also satisfy the 2-out-of-3 property). Thus, it can be factored as $X\to \overline{X}''\to \overline{X}\times_Y \overline{X}'$ where $\overline{f}'': \overline{X}''\to Y$ still lies in $P$, and $j'': X\to \overline{X}''$ lies in $I$. It now suffices to construct isomorphisms
\[
\overline{f}_\ast j_!\cong \overline{f}''_\ast j''_!\cong \overline{f}'_\ast j'_!.
\]
Restricting attention to one half, we can (renaming $\overline{X}''$ as $\overline{X}'$) assume there is a map $g: \overline{X}'\to \overline{X}$ over $Y$. But then $g_\ast j'_!\cong j_!$ by the previous construction (in the degenerate case where one of the maps in $P$ is the identity), and post-composing with $\overline{f}_\ast$ gives the desired isomorphism.
\end{construction}

The previous constructions merely serve as an indication that the statement of the following theorem (which is really a construction) is sensible.

\begin{theorem}[Mann {\cite[Proposition A.5.10]{MannThesis}}, Liu--Zheng]\label{thm:construct6functors} Under the above assumptions, there is a canonical extension of $\mathcal D_0$ to a lax symmetric monoidal functor
\[
\mathcal D: \mathrm{Corr}(C,E)\to \mathrm{Cat}_\infty
\]
such that for $f\in I$, $f_!$ is the left adjoint of $f^\ast$, while for $f\in P$, $f_!$ is the right adjoint of $f^\ast$.
\end{theorem}

\begin{proof}[Sketch of Construction] In this sketch, we will ignore the lax symmetric monoidal structure; this can be taken care of by thinking about the functor $\mathrm{Corr}(C,E)^\otimes\to \mathrm{Cat}_\infty$ instead, where the source is itself a correspondence $\infty$-category to which the constructions below apply. (Indeed, one only has to replace $C$ below by $C^{\times'}$.)

An important component of the construction are multisimplicial sets; more precisely, bi- and trisimplicial sets. Note that $\mathrm{Corr}(C,E)$ can be constructed in two steps. First, one can define a bisimplicial set $\mathrm{Corr}(C,E)_{\mathrm{contra},\mathrm{co}}$ whose $(n,m)$-simplices are maps
\[
(\Delta^n)^{\mathrm{op}}\times \Delta^m\to C
\]
satisfying an analogue of the condition defining $\mathrm{Corr}(C,E)$; i.e., sending morphisms of the second variable to morphisms in $E$, and making all squares cartesian. Now for any bisimplicial set $X$, one can define a simplicial set $\delta^\ast_+ X$ whose $n$-simplices are given by maps $\Delta^{(n,n)}_+\to X$ where $\Delta^{(n,n)}$ is the bisimplicial set given by an $(n,n)$-simplex, and $\Delta^{(n,n)}_+\subset \Delta^{(n,n)}$ denotes the sub-bisimplicial set spanned by vertices $(i,j)$ with $i\geq j$. Unraveling the definitions, we have
\[
\mathrm{Corr}(C,E) = \delta^\ast_+ \mathrm{Corr}(C,E)_{\mathrm{contra},\mathrm{co}}.
\]
There is a more standard way to pass from a bisimplicial set $X$ to a simplicial set $\delta^\ast X$, via taking as $n$-simplices the maps $\Delta^{(n,n)}\to X$. The inclusion $\Delta^{(n,n)}_+\subset \Delta^{(n,n)}$ induces a comparison map
\[
\delta^\ast X\to \delta^\ast_+ X.
\]
Note that in general neither $\delta^\ast X$ nor $\delta^\ast_+ X$ is an $\infty$-category; and in our case, $\delta^\ast \mathrm{Corr}(C,E)_{\mathrm{contra},\mathrm{co}}$ is not an $\infty$-category. Thus, we are here really using the interpretation of $\infty$-categories as simplicial sets, and ambient simplicial sets that are not $\infty$-categories.

\begin{remark} Vertices of $\delta^\ast \mathrm{Corr}(C,E)_{\mathrm{contra},\mathrm{co}}$ are objects of $C$, while edges (from $X$ to $Y$) are given by cartesian squares
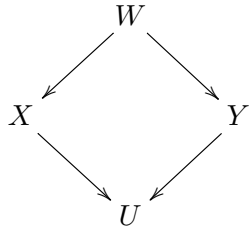
\[\xymatrix{
& W\ar[dr]\ar[dl] \\
X\ar[dr] && Y\ar[dl]\\
& U
}\]
where the right-vertical arrows are in $E$. But given two such composable edges, one cannot in general compose them to a $2$-simplex, which would be a diagram
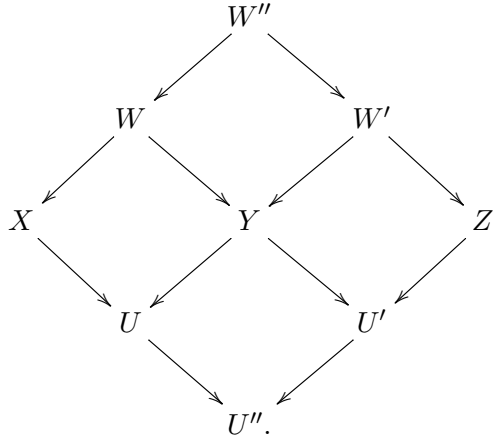
\[\xymatrix{
&& W''\ar[dr]\ar[dl]\\
& W\ar[dr]\ar[dl] && W'\ar[dr]\ar[dl]\\
X\ar[dr] && Y\ar[dr]\ar[dl] && Z\ar[dl]\\
 & U\ar[dr] && U'\ar[dl]\\
&& U''.
}\]
Indeed, it is in general hard to find $U''$.
\end{remark}

\begin{theorem}[{\cite[Theorem 4.27]{LiuZhengGluing}}] For any bisimplicial set $X$, the map $\delta^\ast X\to \delta^\ast_+ X$ is a categorical equivalence. In other words, for any $\infty$-category $\mathcal E$, the functor
\[
\mathrm{Fun}(\delta^\ast_+ X,\mathcal E)\to \mathrm{Fun}(\delta^\ast X,\mathcal E)
\]
of $\infty$-categories is an equivalence.
\end{theorem}

\begin{remark} One might wonder whether $\delta^\ast \mathrm{Corr}(C,E)_{\mathrm{contra},\mathrm{co}}$ has enough morphisms for this to be true, as few correspondences $X\leftarrow W\to Y$ extend to a cartesian square. But any correspondence for which one of $W\to X$ and $W\to Y$ is an isomorphism naturally lifts to an edge of $\delta^\ast \mathrm{Corr}(C,E)_{\mathrm{contra},\mathrm{co}}$ (by taking $U=X$ or $U=Y$), and any correspondence can be canonically written as the composite of two such. As composites of morphisms are unique (up to contractible choice) in an $\infty$-category, any morphism in $\mathrm{Corr}(C,E)$ lifts canonically (up to contractible choice) to any fibrant replacement of $\delta^\ast \mathrm{Corr}(C,E)_{\mathrm{contra},\mathrm{co}}$ (i.e.~any $\infty$-category equipped with a categorical equivalence from this simplicial set).
\end{remark}

Using the theorem, it is thus sufficient to construct a map of simplicial sets
\[
\delta^\ast \mathrm{Corr}(C,E)_{\mathrm{contra},\mathrm{co}}\to \mathrm{Cat}_\infty.
\]

At this point, tri-simplicial sets enter the picture. Namely, we consider the trisimplicial set
\[
\mathrm{Corr}(C,I,P)_{\mathrm{contra},\mathrm{co},\mathrm{co}}
\]
whose $(n,m,k)$-simplices are maps
\[
(\Delta^n)^{\mathrm{op}}\times \Delta^m\times \Delta^k\to C
\]
sending morphisms of the second variable to morphisms in $I$, morphisms of the third variable to morphisms in $P$, and making all squares and cubes cartesian. We can turn this into a bisimplicial set by contracting the last two factors, getting a natural map to $\mathrm{Corr}(C,E)_{\mathrm{contra},\mathrm{co}}$; and then we can turn both into simplicial sets by further diagonal restriction. In summary, denoting by $\delta^\ast$ full diagonal restriction for both tri- and bi-simplicial sets, we get a map of simplicial sets
\[
\delta^\ast \mathrm{Corr}(C,I,P)_{\mathrm{contra},\mathrm{co},\mathrm{co}}\to \delta^\ast \mathrm{Corr}(C,E)_{\mathrm{contra},\mathrm{co}}.
\]

\begin{theorem}[{\cite[Theorem 5.4]{LiuZhengGluing}}]\label{thm:decomposemorphisms} This map is a categorical equivalence.
\end{theorem}

\begin{remark} This theorem guarantees, roughly speaking, that decomposing a map in $E$ into maps in $I$ and $P$ is ``unique up to contractible choice'', and encompasses the explicit constructions above.
\end{remark}

In other words, it is enough to construct a map of simplicial sets
\[
\delta^\ast \mathrm{Corr}(C,I,P)_{\mathrm{contra},\mathrm{co},\mathrm{co}}\to \mathrm{Cat}_\infty.
\]
What is easy to construct is a map
\[
\delta^\ast \mathrm{Corr}(C,I,P)_{\mathrm{contra},\mathrm{contra},\mathrm{contra}}\to \mathrm{Cat}_\infty.
\]
Indeed, there is a natural functor
\[
\delta^\ast \mathrm{Corr}(C,I,P)_{\mathrm{contra},\mathrm{contra},\mathrm{contra}}\to C^{\mathrm{op}}
\]
recording the diagonal composition of morphisms; and we can compose this with the given functor $\mathcal D_0$ to $\mathrm{Cat}_\infty$.

The idea now is to pass to adjoints in the second and third variable. In general, for a trisimplicial set $X$, giving a map $\delta^\ast X\to \mathrm{Cat}_\infty$ means to give for each $(n,m,k)$-simplex $\Delta^{(n,m,k)}$ of $X$ a map
\[
\Delta^n\times \Delta^m\times \Delta^k\to \mathrm{Cat}_\infty,
\]
functorially in the simplex of $X$. This is equivalent to
\[
\Delta^n\times \Delta^m\to \mathrm{Fun}(\Delta^k,\mathrm{Cat}_\infty)
\]
where the latter is (up to equivalence) given by $\infty$-categories $\mathcal C_0,\ldots,\mathcal C_k$ with functors $F_i: \mathcal C_i\to \mathcal C_{i+1}$ for $i=0,\ldots,k-1$. One can define a sub-$\infty$-category $\mathrm{Fun}^L(\Delta^k,\mathrm{Cat}_\infty)$ whose objects are those where each $F_i$ is a left adjoint, i.e.~admits a right adjoint $G_i$, and where the morphisms are those maps $\mathcal C_i\to \mathcal C_i'$ (commuting with the $F_i$) which also, after passing to adjoints, commute with the $G_i$. On this subcategory, we can pass to adjoints, giving a functor
\[
\mathrm{Fun}^L(\Delta^k,\mathrm{Cat}_\infty)\to \mathrm{Fun}((\Delta^k)^{\mathrm{op}},\mathrm{Cat}_\infty).
\]
In the context of the trisimplicial set, this gives a map
\[
\Delta^n\times \Delta^m\times (\Delta^k)^{\mathrm{op}}\to \mathrm{Cat}_\infty,
\]
i.e.~one has reversed the arrows in the third direction. Of course, the same discussion applies if one passes to left adjoints instead.

We apply these remarks to first pass to left adjoints in the second direction (which is made possible by condition (2) on base change (and projection formula when keeping track of the symmetric monoidal structure)) and then pass to right adjoints in the third variable (which is made possible by condition (3) on base change (and projection formula when keeping track of the symmetric monoidal structure), as well as condition (4)), to turn the functor
\[
\delta^\ast \mathrm{Corr}(C,I,P)_{\mathrm{contra},\mathrm{contra},\mathrm{contra}}\to \mathrm{Cat}_\infty
\]
into a functor
\[
\delta^\ast \mathrm{Corr}(C,I,P)_{\mathrm{contra},\mathrm{co},\mathrm{co}}\to \mathrm{Cat}_\infty
\]
as desired.
\end{proof}

\begin{remark} The construction shows that if one formulates the requirement that $f_!$ is the left (resp.~right) adjoint of $f^\ast$ for $f\in I$ (resp.~$f\in P$) in a suitably strong form, then the extension is in fact unique. Indeed, the only parts of the construction that were not of propositional nature were the parts about passing to left resp.~right adjoints in the context of this tri-simplicial set, but doing so was essentially dictated by the desired outcome.
\end{remark}

\begin{remark} To get a full $6$-functor formalism, it remains to ask that all $\mathcal D(X)$ are closed symmetric monoidal $\infty$-categories; that all $f^\ast$ admit right adjoints $f_\ast$; and that for $f\in P$, the functor $f_\ast$ admits a further right adjoint $f^!$.
\end{remark}\newpage

\section{Lecture V: Poincar\'e Duality}

The goal of this lecture is to discuss, in the generality of abstract $6$-functor formalisms, a notion of Poincar\'e Duality, and give a simple means for establishing it.

As discussed in the first lecture, Poincar\'e Duality should say that for a ``smooth'' morphism $f: X\to Y$, the right adjoint $f^!$ of $f_!$ (exists and) agrees with $f^\ast$ up to a twist. Axiomatizing this desiratum leads to the following definition. Here, we fix $(C,E)$ as before, and a $3$-functor formalism
\[
\mathcal D: \mathrm{Corr}(C,E)\to \mathrm{Cat}_\infty.
\]

\begin{definition}\label{def:cohomsmooth} Let $f: X\to Y$ be a morphism in $E$. Then $f$ is ($\mathcal D$-)cohomologically smooth if the following properties are satisfied.
\begin{enumerate}
\item The right adjoint $f^!$ of $f_!$ exists, and the natural transformation
\[
f^!(1_Y)\otimes f^\ast(-)\to f^!(-)
\]
of functors $\mathcal D(Y)\to \mathcal D(X)$ is an isomorphism.
\item The object $f^!(1_Y)\in \mathcal D(X)$ is $\otimes$-invertible.
\item For any $g: Y'\to Y$ with base change $f':X'=X\times_Y Y'\to Y'$ of $f$, properties (1) and (2) also hold for $f'$, and moreover the natural map
\[
g^{\prime\ast} f^!(1_Y)\to f^{\prime !}(1_{Y'})
\]
is an isomorphism, where $g': X'\to X$ is the base change of $g$.
\end{enumerate}
\end{definition}

Here, the map $f^!(1_Y)\otimes f^\ast\to f^!$ is adjoint to
\[
f_!(f^!(1_Y)\otimes f^\ast(-))\cong f_!f^!(1_Y)\otimes (-)\to (-)
\]
using the projection formula and the counit map $f_!f^!(1_Y)\to 1_Y$; and $g^{\prime\ast} f^!(1_Y)\to f^{\prime !}(1_{Y'})$ is adjoint to
\[
f'_! g^{\prime\ast} f^!(1_Y)\cong g^\ast f_! f^!(1_Y)\to g^\ast(1_Y) = 1_{Y'}
\]
using base change and the same counit map.

\begin{remark} This notion was defined in \cite{EtCohDiamonds} when I failed to define a notion of smoothness for maps of perfectoid spaces. The issue there was that the topological definition (``locally looks like euclidean space'') did not work as not all examples were locally isomorphic; and the algebraic-geometric definition (``lifting against nilpotent thickenings'') did not work as perfectoid spaces have no nilpotents. Thus, in the end I simply characterized the desired cohomological properties, and then showed that all relevant examples have this property.
\end{remark}

\begin{remark} In general, cohomological smoothness is weaker than smoothness; for example, any universal homeomorphism of schemes is cohomologically smooth with respect to \'etale cohomology, as \'etale sheaves are insensitive to universal homeomorphisms.
\end{remark}

\begin{remark} It follows from the definition that the class of cohomologically smooth morphisms is stable under base change, composition, and contains all isomorphisms. Moreover, in case the $3$-functor formalism is defined from classes of morphisms $I$ and $P$ as in the last lecture, all morphisms in $I$ are cohomologically smooth. In fact, in that case $f_!$ is left adjoint to $f^\ast$, i.e.~$f^\ast=f^!$, and the properties are clear.
\end{remark}

Checking that $f: X\to Y$ is cohomologically smooth seems highly nontrivial. Indeed, for any base change of $f$, one needs to prove that some map is an isomorphism for all $B\in \mathcal D(Y)$; and the map involves $f^! B$ which is abstractly defined as an adjoint, so one has to compute the morphisms from any $A\in \mathcal D(X)$ towards $f^! B$.

The goal of this lecture is to show that in fact, it is enough to construct a surprisingly small amount of data (and check the commutativity of two diagrams), and this data involves only some very simple sheaves on $X$, $Y$ and $X\times_Y X$.

To simplify the situation, we will in the following replace $C$ by the slice $C_{/Y}$, so that we can assume that $Y$ is the final object of $C$. Moreover, we will restrict to the subcategory of objects $X\in C$ such that $X\to Y$ lies in $E$, and assume that all morphisms $X\to X'$ over $Y$ are still in $E$ -- this is notably satisfied in the setup of the last lecture. Thus, in fact all morphisms of $C$ lie in $E$.

\begin{theorem}\label{thm:critcohomsmooth} Assume that all morphisms of $C$ lie in $E$, and let $f: X\to Y$ be a map to the final object $Y\in C$. Let $\Delta: X\to X\times_Y X$ be the diagonal. Then $f$ is cohomologically smooth if and only if there is a $\otimes$-invertible object $L\in D(X)$ and maps
\[
\alpha: \Delta_! 1_X\to p_2^\ast L, \beta: f_! L\to 1_Y
\]
such that the composite
\[
1_X\cong p_{1!}\Delta_! 1_X\xrightarrow{p_{1!}\alpha} p_{1!} p_2^\ast L\cong f^\ast f_! L\xrightarrow{f^\ast\beta} 1_X
\]
is the identity, as well as the composite
\[
L\cong p_{2!}(p_1^\ast L\otimes \Delta_! 1_X)\xrightarrow{p_{2!}(p_1^\ast L\otimes \alpha)} p_{2!}(p_1^\ast L\otimes p_2^\ast L)\cong p_{2!} p_1^\ast L\otimes L\cong f^\ast f_! L\otimes L\xrightarrow{f^\ast \beta\otimes L} L.
\]
\end{theorem}

Let us first prove the forward direction. For this, we take $L=f^!(1_Y)$, and $\beta: f_! L = f_! f^!(1_Y)\to 1_Y$ the counit of the adjunction. For $\alpha$, we note that $p_2^\ast L = p_1^!(1_X)$ by property (3) in Definition~\ref{def:cohomsmooth}, and then $\alpha$ is adjoint to $1_X=\Delta^! p_2^!(1_X)$.

We need to prove that these two composites are the identity. The first one is actually straightforward from the definition, but the second is more subtle. It could be done by a direct, but elaborate, diagram chase. Let us give a more abstract argument instead, one that will actually introduce the techniques that will be useful in the converse direction. Namely, note that if $X_1,X_2\in C$ and $K\in D(X_1\times_Y X_2)$ is an object, then it induces a ``Fourier--Mukai functor with kernel $K$'':
\[
\mathcal D(X_1)\to \mathcal D(X_2): A\mapsto p_{2!}(p_1^\ast A\otimes K).
\]
In particular, taking $X_1=X$ and $X_2=Y$, the object $K=1_X$ corresponds to the functor
\[
F=f_!: \mathcal D(X)\to \mathcal D(Y),
\]
while taking $X_1=Y$ and $X_2=X$, the object $K=L$ corresponds to the functor
\[
G=f^\ast\otimes L: \mathcal D(Y)\to \mathcal D(X).
\]
Cohomological smoothness ensures that these two functors are adjoint, so there are natural transformations
\[
\alpha_0: \mathrm{id}_{\mathcal D(X)}\to GF
\]
and
\[
\beta_0: FG\to \mathrm{id}_{\mathcal D(Y)}.
\]
Now note that a composite of two functors given by a kernel is itself induced by a kernel, through convolution of kernels. In particular $FG$ is given by $f_! L\in \mathcal D(Y)$ and $\mathrm{id}_{\mathcal D(Y)}$ by $1_Y\in \mathcal D(Y)$, and $\beta_0$ is given by $\beta: f_! L\to 1_Y$. On the other hand, $GF$ is given by the kernel $p_2^\ast L\in \mathcal D(X\times_Y X)$, while $\mathrm{id}_{\mathcal D(X)}$ is given by the kernel $\Delta_!(1_X)\in \mathcal D(X\times_Y X)$. We claim that
\[
\alpha: \Delta_!(1_X)\to p_2^\ast L
\]
induces
\[
\alpha_0: \mathrm{id}_{\mathcal D(X)}\to GF.
\]
Let temporarily $\alpha_0'$ denote the transformation induced by $\alpha$. Then
\[
(\mathrm{Cat},D(X),D(Y),f_!,f^\ast\otimes L,\beta_0,\alpha_0,\alpha_0')
\]
satisfy the hypothesis of the following lemma, concerning a version of ``left and right inverse agree'' in the context of adjunctions in $2$-categories.

\begin{lemma}\label{lem:adjunctionleftright} Let $\mathcal C$ be a $2$-category and $F: X\to Y$, $G: Y\to X$ be $1$-morphisms in $\mathcal C$, together with $2$-morphisms
\[
\beta: FG\to 1, \alpha: 1\to GF, \alpha': 1\to GF
\]
such that the composites
\[
F\xrightarrow{F\alpha'} FGF\xrightarrow{\beta F} F, G\xrightarrow{\alpha G} GFG\xrightarrow{G\beta} G
\]
are the identity. Then $\alpha=\alpha'$.
\end{lemma}

\begin{proof} We have
\[\begin{aligned}
1\xrightarrow{\alpha} GF &= 1\xrightarrow{\alpha} GF\xrightarrow{GF\alpha'} GFGF\xrightarrow{G\beta F} GF\\
&=1\xrightarrow{\alpha'} GF\xrightarrow{\alpha GF} GFGF\xrightarrow{G\beta F} GF\\
&=1\xrightarrow{\alpha'} GF.
\end{aligned}\]
\end{proof}

Thus, the unit transformation $\alpha_0: \mathrm{id}_{\mathcal D(X)}\to GF$ is indeed induced by $\alpha: \Delta_!(1_X)\to p_2^\ast L$. Now the composite
\[
G\xrightarrow{\alpha_0 G} GFG\xrightarrow{G\beta_0} G
\]
is the identity, and by unraveling the definitions one computes that this composite is induced by the map on kernels
\[
L\cong p_{2!}(p_1^\ast L\otimes \Delta_! 1_X)\xrightarrow{p_{2!}(p_1^\ast L\otimes \alpha)} p_{2!}(p_1^\ast L\otimes p_2^\ast L)\cong p_{2!} p_1^\ast L\otimes L\cong f^\ast f_! L\otimes L\xrightarrow{f^\ast \beta\otimes L} L.
\]
In general, passing from kernels to functors loses information, but in this case $L$ encodes the functor $G=f^\ast\otimes L: \mathcal D(Y)\to \mathcal D(X)$ which sends $1_Y\in \mathcal D(Y)$ to $L\in \mathcal D(X)$. Thus, knowing that the induced transformation of functors is the identity implies that the transformation on kernels is the identity.

To prove the converse direction of Theorem~\ref{thm:critcohomsmooth}, we formalize the preceding arguments through the introduction of a certain $2$-category; some version of this idea is in work of Lu--Zheng \cite{LuZheng}, and this was adapted slightly in \cite{FarguesScholze}, \cite{HansenScholze}. Informally, given a geometric setup $(C,E)$ where $E$ consists of all morphisms, and a $3$-functor formalism $\mathcal D$, let
\[
\mathrm{LZ}_{\mathcal D}
\]
be the $2$-category with:
\begin{itemize}
\item Objects given by objects of $C$;
\item Morphism categories: $\mathrm{Hom}_{\mathrm{LZ}_{\mathcal D}}(X,X') = D(X\times X')$;
\item Identity morphisms: $\mathrm{id}_X\in \mathrm{Hom}_{\mathrm{LZ}_{\mathcal D}}(X,X)=D(X\times X)$ is given by $\Delta_!(1_X)$;
\item Composition:
\[
\mathrm{Hom}_{\mathrm{LZ}_{\mathcal D}}(X,X')\times \mathrm{Hom}_{\mathrm{LZ}_{\mathcal D}}(X',X'')\to \mathrm{Hom}_{\mathrm{LZ}_{\mathcal D}}(X,X'')
\]
is given by the convolution
\[
(A,B)\mapsto A\star B = p_{X,X''!}(p_{X,X'}^\ast A\otimes p_{X',X''}^\ast B).
\]
\end{itemize}

To turn this into a definition, one has to supply isomorphisms between $(A\star B)\star C$ and $A\star (B\star C)$ satisfying the pentagon axiom for fourfould convolution (and similar isomorphisms for the identity morphisms). This becomes a bit cumbersome, so let us sketch a high-tech construction of $\mathrm{LZ}_{\mathcal D}$; see \cite{Zavyalov} for an elaboration of this construction. One starts with the symmetric monoidal $\infty$-category $\mathrm{Corr}(C,E)$. All of its objects are dualizable (in fact, self-dual, using diagonal correspondences). Thus, this is a closed symmetric monoidal $\infty$-category, and hence enriched over itself, with internal mapping objects
\[
\underline{\mathrm{Hom}}_{\mathrm{Corr}(C,E)}(X,X') = X\times X'.
\]
Enrichments can be transferred along lax symmetric monoidal functors. Applying this to the functor $\mathcal D$, we turn $\mathrm{Corr}(C,E)$ into a $\mathrm{Cat}_\infty$-enriched $\infty$-category, with internal Hom-objects given by $\mathcal D(X\times X')$. But $\mathrm{Cat}_\infty$-enriched $\infty$-categories are a model for $(\infty,2)$-categories, and in particular the homotopy $2$-category gives the desired $2$-category $\mathrm{LZ}_{\mathcal D}$. We outline another approach based on presentable $(\infty,2)$-categories in the appendix to this lecture.

\begin{remark} It is worth pointing out how well the definition of a $3$-functor formalism fits into this construction.
\end{remark}

The idea behind the definition of $\mathrm{LZ}_{\mathcal D}$ is that there is a natural functor
\[
\mathrm{Hom}_{\mathrm{LZ}_{\mathcal D}}(Y,-): \mathrm{LZ}_{\mathcal D}\to \mathrm{Cat}
\]
taking any $X$ to $D(X)=D(Y\times_Y X)$, and any $K\in D(X\times_Y X')$ into the functor
\[
D(X)\to D(X'): A\mapsto p_{2!}(K\otimes p_1^\ast A),
\]
i.e.~the ``Fourier--Mukai functor with kernel $K$''. In general, this passage from $K$ to the induced functor $D(X)\to D(X')$ is very lossful, and working in $\mathrm{LZ}_{\mathcal D}$ amounts to working with kernels of functors directly, instead of with the induced functors.

To prove cohomological smoothness, we want to show that (up to twist) $f^\ast$ is a right adjoint of $f_!$. What we will do instead is to prove the adjointness already in $\mathrm{LZ}_{\mathcal D}$. This makes sense, as adjunctions can be defined in any $2$-category:

\begin{definition} Let $\mathcal C$ be a $2$-category and $F: X\to Y$ a $1$-morphism in $\mathcal C$. A right adjoint of $F$ is a triple $(G,\alpha,\beta)$ of a $1$-morphism $G: Y\to X$ and $2$-morphisms $\alpha: 1\to GF$ and $\beta: FG\to 1$ such that the composites
\[
F\xrightarrow{F\alpha} FGF\xrightarrow{\beta F} F, G\xrightarrow{\alpha G} GFG\xrightarrow{G\beta} G
\]
are the identity.
\end{definition}

Some basic properties of adjunctions are:

\begin{enumerate}
\item Adjunctions are unique up to unique isomorphism. More precisely, if $(G,\alpha,\beta)$ and $(G',\alpha',\beta')$ are right adjoints of $F$, then we get a unique isomorphism $(G,\alpha,\beta)\cong (G',\alpha',\beta')$ making the obvious diagrams commute. For the construction, for example the map $G\to G'$ is defined as
\[
G\xrightarrow{\alpha' G} G'FG\xrightarrow{G'\beta} G'.
\]
\item Any functor of $2$-categories preserves adjunctions.
\item Adjunctions in $\mathrm{Cat}$ are usual adjunctions of categories.
\end{enumerate}

Properties (2) and (3) are immediate from the definition. With this preparation, we can give the proof of Theorem~\ref{thm:critcohomsmooth}.

\begin{proof}[Proof of Theorem~\ref{thm:critcohomsmooth}] We need to prove the backwards direction, so take $(L,\alpha,\beta)$. Consider $X$ and $Y$ as objects of $\mathrm{LZ}_{\mathcal D}$, and the morphisms
\[
F=1_X\in D(X)=\mathrm{Hom}_{\mathrm{LZ}_{\mathcal D}}(X,Y), G=L\in D(X)=\mathrm{Hom}_{\mathrm{LZ}_{\mathcal D}}(Y,X).
\]
(We recall that we assumed that $Y$ is the final object of $C$.) Then $F$ encodes the functor $f_!: D(X)\to D(Y)$ and $G$ encodes the functor $f^\ast\otimes L: D(Y)\to D(X)$, and we wish to see that $f_!$ is the left adjoint of $f^\ast\otimes L$. We will in fact show that $F$ is a left adjoint of $G$ in $\mathrm{LZ}_{\mathcal D}$. Indeed, $\alpha$ and $\beta$ translate into maps
\[
\mathrm{id}_X=\Delta_!(1_X)\to GF=p_2^\ast L\in D(X\times_Y X)=\mathrm{Hom}_{\mathrm{LZ}_{\mathcal D}}(X,X)
\]
and
\[
FG=f_! L\to \mathrm{id}_Y = 1_Y\in D(Y) = \mathrm{Hom}_{\mathrm{LZ}_{\mathcal D}}(Y,Y),
\]
and the required commutative diagrams translate into the ones in Theorem~\ref{thm:critcohomsmooth}.

In particular, it follows that the right adjoint $f^!$ is given by a kernel, and in particular it is $D(Y)$-linear, so $f^!(1_Y)\otimes f^\ast\to f^!$ is an isomorphism. Moreover, we must have $L=f^!(1_Y)$, which we assumed to be $\otimes$-invertible.

Finally, for any base change of $f$ along $g: Y'\to Y$, base change along $g$ defines a functor of $2$-category $\mathrm{LZ}_{\mathcal D}\to \mathrm{LZ}_{\mathcal D|_{C_{/Y'}}}$, which hence gives the same adjunction for the base change of $f$, with the base change of $L$ as dualizing complex.
\end{proof}

Let us end the lecture by discussing the example of locally compact Hausdorff topological spaces.

\begin{example} Consider the category $C$ of finite-dimensional locally compact Hausdorff topological spaces, with the functor $\mathcal D: X\mapsto \mathcal D(\mathrm{Ab}(X))$. For the classes of open immersions $I$ and proper maps $P$, this satisfies the conditions of the last lecture, giving us a $6$-functor formalism (where $E$ consists of all maps).

\begin{proposition}\label{prop:realssmooth} The map $f: \mathbb R\to \ast$ is cohomologically smooth.
\end{proposition}

Note that by compatibility with composites and base change and open immersions, this implies that all manifold bundles are cohomologically smooth.

\begin{proof} We take $L=\mathbb Z[1]$. We need to construct maps $\alpha$ and $\beta$. But $R\Gamma_c(\mathbb R,\mathbb Z[1])\cong \mathbb Z$ giving us $\alpha$; and using the triangle $0\to \Delta_! \mathbb Z\to \mathbb Z\to j_! \mathbb Z$ for $j: \mathbb R^2\setminus \Delta\to \mathbb R^2$ the complementary open, one computes
\[
\mathrm{RHom}(\Delta_! \mathbb Z[-1],\mathbb Z)\cong \mathbb Z,
\]
giving us $\beta$. Indeed, the triangle shows that $\mathrm{RHom}(\Delta_! \mathbb Z[-1],\mathbb Z)$ is the cone of
\[
R\Gamma(\mathbb R^2,\mathbb Z)\to R\Gamma(\mathbb R^2\setminus \Delta,\mathbb Z),
\]
and this map is the diagonal embedding $\mathbb Z\to \mathbb Z^2$.

To finish the proof, one would like to see that an endomorphism of the constant sheaf $\mathbb Z$ on $\mathbb R$ is the identity (if the signs of $\alpha$ and $\beta$ are chosen compatibly). It is helpful that in this case, all intermediate maps are isomorphisms, so it has to be true up to sign. A priori, one might worry that these signs cannot be chosen compatibly, but here Lemma~\ref{lem:adjunctionleftright} comes to the rescue. Even better, one has the following general principle about adjunctions.
\end{proof}
\end{example}

\begin{lemma}\label{lem:adjunctioncheat} Let $\mathcal C$ be a $2$-category, $F: X\to Y$, $G: Y\to X$ be $1$-morphisms in $\mathcal C$ as well as $2$-morphisms $\alpha: 1\to GF$ and $\beta: FG\to 1$. Assume that the $2$-morphisms
\[
F\xrightarrow{F\alpha} FGF\xrightarrow{\beta F} F, G\xrightarrow{\alpha G} GFG\xrightarrow{G\beta} G
\]
are isomorphisms. Then one can find some $\alpha': 1\to GF$ such that $(G,\alpha',\beta)$ is a right adjoint of $F$.
\end{lemma}

\begin{proof} Changing $\alpha$ by an isomorphism $G\cong G$, we can arrange that $G\xrightarrow{\alpha G} GFG\xrightarrow{G\beta} G$ is the identity. We can also find some $\alpha': 1\to GF$ such that $F\xrightarrow{F\alpha'} FGF\xrightarrow{\beta F} F$ is the identity. By Lemma~\ref{lem:adjunctionleftright}, we necessarily have $\alpha=\alpha'$, which then does the trick.
\end{proof}\newpage

\section*{Appendix to Lecture V: Presentable $(\infty,2)$-categories of kernels}

The construction of this lecture also admits a large ``presentable'' version. Fix some geometric setup $(C,E)$ and a six-functor formalism
\[
\mathcal D: \mathrm{Corr}(C,E)\to \mathrm{Cat}_\infty.
\]
We assume that $D(X)\in \mathrm{Pr}^L$ is a presentable $\infty$-category for all $X\in C$, i.e.~admits all colimits and for some cardinal $\kappa$ it is freely generated under $\kappa$-filtered colimits by a set of $\kappa$-compact objects. We note that for any geometric setup $(C,E)$, the functor category
\[
\mathrm{Fun}(\mathrm{Corr}(C,E),\mathrm{Pr}^L)
\]
admits a symmetric monoidal Day convolution structure. A standard result about Day convolutions is that the commutative algebra objects are precisely the lax symmetric monoidal functors. Thus,
\[
\mathcal D\in \mathrm{CAlg}(\mathrm{Fun}(\mathrm{Corr}(C,E),\mathrm{Pr}^L)).
\]
We can in particular form the $\infty$-category
\[
\mathrm{Mod}_{\mathcal D}(\mathrm{Fun}(\mathrm{Corr}(C,E),\mathrm{Pr}^L))
\]
of $\mathcal D$-modules.

It turns out that it is better to do this when all maps are in $E$, as in the lecture. Thus, fix some object $X\in C$ and restrict to $C^E_{/X}$, the slice category of maps $Y\to X$ in $E$. Let
\[
\mathcal D_{/X}: \mathrm{Corr}(C^E_{/X})\to \mathrm{Pr}^L
\]
be the restriction of $\mathcal D$, and let
\[
\mathrm{Pr}^L_{\mathcal D,X} = \mathrm{Mod}_{\mathcal D_{/X}}(\mathrm{Fun}(\mathrm{Corr}(C^E_{/X}),\mathrm{Pr}^L).
\]
This is a symmetric monoidal $\infty$-category equipped with a symmetric monoidal functor from $\mathrm{Pr}^L$, taking any $\mathcal C\in \mathrm{Pr}^L$ to the module
\[
(Y\to X)\mapsto \mathcal D(Y)\otimes \mathcal C.
\]
In particular, for $F,G\in \mathrm{Pr}^L_{\mathcal D}$, we can hope to define an internal Hom object $\mathrm{Fun}_{\mathrm{Pr}^L_{\mathcal D}}(F,G)\in \mathrm{Pr}^L$ representing $\mathcal C\mapsto \mathrm{Hom}_{\mathrm{Pr}^L_{\mathcal D}}(F\otimes \mathcal C,G)$, thereby in particular defining some $(\infty,2)$-categorical structure. In general, it is not clear that this exists (as $\mathrm{Pr}^L$ is not itself presentable).

Thus, to simplify the discussion, we fix some regular cardinal $\kappa$ so that
\[
\mathcal D\in \mathrm{CAlg}(\mathrm{Fun}(\mathrm{Corr}(C,E),\mathrm{Pr}_\kappa^L))
\]
where $\mathrm{Pr}_\kappa^L$ is the non-full subcategory of $\mathrm{Pr}_\kappa^L$ consisting of $\kappa$-compactly generated presentable $\infty$-categories, and functors preserving $\kappa$-compact objects. (In practice, $\kappa=\omega_1$ works always.) Then $\mathrm{Pr}_\kappa^L$ is, itself, a $\kappa$-compactly generated presentable $\infty$-category. Indeed, $\mathrm{Pr}_\kappa^L$ is equivalent to the $\infty$-category of small $\infty$-categories admitting all $\kappa$-small colimits, and functors preserving $\kappa$-small colimits. This is generated by $\mathrm{Fun}(\Delta^n,\mathrm{Ani}^\kappa)$ for integers $n$, and these are $\kappa$-compact.

Now, leaving the choice of $\kappa$ implicit,
\[
\mathrm{Pr}_{\mathcal D,X} := \mathrm{Mod}_{\mathcal D_{/X}}(\mathrm{Fun}(\mathrm{Corr}(C^E_{/X}),\mathrm{Pr}_\kappa^L))
\]
is a symmetric monoidal $\infty$-category equipped with a symmetric monoidal functor $\mathrm{Pr}_\kappa^L\to \mathrm{Pr}_{\mathcal D,X}$. As $\mathrm{Pr}_\kappa^L$ is presentable, we get in particular a $\mathrm{Pr}_\kappa^L$-enrichment of $\mathrm{Pr}_{\mathcal D,X}$.

Note that by Yoneda
\[
\mathrm{Fun}(\mathrm{Corr}(C^E_{/X}),\mathrm{Pr}_\kappa^L)=\mathrm{Fun}(\mathrm{Corr}(C^E_{/X}),\mathrm{Ani})\otimes_{\mathrm{Ani}} \mathrm{Pr}_\kappa^L
\]
is the free $\mathrm{Pr}_\kappa^L$-linear presentable $\infty$-category with a functor from $\mathrm{Corr}(C^E_{/X})^{\mathrm{op}}$. In particular, there is a natural symmetric monoidal functor
\[
\mathrm{Corr}(C^E_{/X})^{\mathrm{op}}\to \mathrm{Fun}(\mathrm{Corr}(C^E_{/X}),\mathrm{Pr}_\kappa^L).
\]
Explicitly, this takes any $Y\to X$ to the functor taking any $Y'$ to the free $\kappa$-presentable $\infty$-category on the anima $\mathrm{Hom}_{\mathrm{Corr}(C^E_{/X})}(Y',Y)$, i.e.
\[
Y'\mapsto \mathrm{Fun}(\mathrm{Hom}_{\mathrm{Corr}(C^E_{/X})}(Y',Y),\mathrm{Ani}).
\]
We can also pass further to $\mathcal D_{/X}$-modules and get a symmetric monoidal functor
\[
\mathrm{Corr}(C^E_{/X})^{\mathrm{op}}\to \mathrm{Pr}_{\mathcal D,X}.
\]
As all objects of $\mathrm{Corr}(C^E_{/X})^{\mathrm{op}}$ are selfdual, we see that the mapping object (in $\mathrm{Pr}^\kappa$) between the images of $Y$ and $Y'$ is the same as the one for mapping from $Y\times_X Y'$ to $X$. Now the object corresponding to $X$ is just $\mathcal D$, and mapping out of the object corresponding to $Y\times_X Y'$ amounts to evaluation, so this mapping object is exactly $\mathcal D(Y\times_X Y')$. We see that the full subcategory of $\mathrm{Pr}_{\mathcal D,X}$ on the image of $\mathrm{Corr}(C^E_{/X})^{\mathrm{op}}\to \mathrm{Pr}_{\mathcal D,X}$ is precisely the $2$-category considered in this lecture. (We did not give a precise definition in the lecture, so the reader is invited to take this as the official definition.)

The construction $X\mapsto \mathrm{Pr}_{\mathcal D,X}$ is contravariantly functorial. In fact, pullback along $f: X'\to X$ gives a functor
\[
\mathrm{Corr}(C^E_{/X})\to \mathrm{Corr}(C^E_{/X'})
\]
and this passes to the free $\mathrm{Pr}_\kappa^L$-linear presentable $\infty$-categories. In particular, we get a symmetric monoidal functor
\[
f^\ast: \mathrm{Fun}(\mathrm{Corr}(C^E_{/X}),\mathrm{Pr}_\kappa^L)\to \mathrm{Fun}(\mathrm{Corr}(C^E_{/X'}),\mathrm{Pr}_\kappa^L).
\]
This functor is in general inexplicit, but its right adjoint is given by the functor
\[
f_\ast: \mathrm{Fun}(\mathrm{Corr}(C^E_{/X'}),\mathrm{Pr}_\kappa^L)\to \mathrm{Fun}(\mathrm{Corr}(C^E_{/X}),\mathrm{Pr}_\kappa^L): F\mapsto ((Y\to X)\mapsto F(Y\times_X X')).
\]
There is a natural map $\mathcal D_{/X}\to f_\ast \mathcal D_{/X'}$ of commutative algebras: When evaluated on $Y\to X$, this is $\mathcal D(Y)\to \mathcal D(Y\times_X X')$, which is just the pullback along $Y\times_X X'\to Y$. By adjunction, we get a functor $f^\ast \mathcal D_{/X}\to \mathcal D_{/X'}$, and this induces the desired functor
\[
f^\ast_{\mathrm{Pr}_{\mathcal D}}: \mathrm{Pr}_{\mathcal D,X} = \mathrm{Mod}_{\mathcal D_{/X}}(\mathrm{Fun}(\mathrm{Corr}(C^E_{/X'}),\mathrm{Pr}_\kappa^L))\to \mathrm{Pr}_{\mathcal D,X'} = \mathrm{Mod}_{\mathcal D_{/X'}}(\mathrm{Fun}(\mathrm{Corr}(C^E_{/X'}),\mathrm{Pr}_\kappa^L)).
\]

If $f: X'\to X$ lies in $E$, the pullback functor
\[
\mathrm{Corr}(C^E_{/X})\to \mathrm{Corr}(C^E_{/X'})
\]
has a simultaneous left and right adjoint, given by sending $Y\to X'$ to the composite $Y\to X'\to X$. It follows that in this case
\[
f^\ast: \mathrm{Fun}(\mathrm{Corr}(C^E_{/X}),\mathrm{Pr}_\kappa^L)\to \mathrm{Fun}(\mathrm{Corr}(C^E_{/X'}),\mathrm{Pr}_\kappa^L)
\]
also has a simultaneous left and right adjoint $f_\ast$, which is also linear, and $f^\ast$ is given explicitly by $F\mapsto ((Y\to X')\mapsto F(Y))$; in particular it follows that the map $f^\ast \mathcal D_{/X}\to \mathcal D_{/X'}$ is an isomorphism. We can then pass to modules over $\mathcal D_{/X}$ to see that
\[
f^\ast_{\mathrm{Pr}_{\mathcal D}}: \mathrm{Pr}_{\mathcal D,X}\to \mathrm{Pr}_{\mathcal D,X'}
\]
has a simultaneous linear left and right adjoint $f_{\ast,\mathrm{Pr}_{\mathcal D}}$.

\begin{theorem}\label{thm:shriekimpliesstar} Let $f: X'\to X$ be a morphism in $E$. The following conditions are equivalent.
\begin{enumerate}
\item[{\rm (1)}] For all $Y\to X$ in $E$, the map $X'\times_X Y\to Y$ satisfies $!$-descent, i.e.
\[
\mathcal D(Y)\to \mathrm{lim}^! \mathcal D(X'^{\times_X \bullet}\times_X Y)
\]
is an equivalence.
\item[{\rm (2)}] The functor
\[
\mathrm{Pr}_{\mathcal D,X}\to \mathrm{lim}^* \mathrm{Pr}_{\mathcal D,X'^{\times_X \bullet}}
\]
is fully faithful.
\item[{\rm (3)}] The functor
\[
\mathrm{Pr}_{\mathcal D,X}\to \mathrm{lim}^* \mathrm{Pr}_{\mathcal D,X'^{\times_X \bullet}}
\]
is an equivalence.
\item[{\rm (4)}] For all $Y\to X$ in $C$, the map $X'\times_X Y\to Y$ satisfies $*$-descent and $!$-descent, i.e. both functors
\[
\mathcal D(Y)\to \mathrm{lim}^! \mathcal D(X'^{\times_X \bullet}\times_X Y)
\]
and
\[
\mathcal D(Y)\to \mathrm{lim}^\ast \mathcal D(X'^{\times_X \bullet}\times_X Y)
\]
are equivalences.
\end{enumerate}

More generally, a similar statement holds if $X'\to X$ is replaced by a family of maps $\{X'_i\to X\}_i$.
\end{theorem}

\begin{proof} First, we show that (1) is equivalent to (2). The functor $\mathrm{Pr}_{\mathcal D,X}\to \mathrm{Pr}_{\mathcal D,X'}$ admits a linear left adjoint (commuting with any base change); these left adjoints exist also for all $X'^{\times_X \bullet}$ and combine to yield a linear left adjoint $F_\sharp$ to
\[
F^\ast: \mathrm{Pr}_{\mathcal D,X}\to \mathrm{lim}^* \mathrm{Pr}_{\mathcal D,X'^{\times_X \bullet}}
\]
explicitly given by taking $M_\bullet\in \mathrm{lim}^* \mathrm{Pr}_{\mathcal D,X'^{\times_X \bullet}}$ to the geometric realization of the left (and right) adjoints applied to $M_\bullet$. (We warn the reader that even while at each individual level, the left and right adjoints agree, it does not follow that the left adjoint $F_\sharp$ and right adjoint $F_\ast$ of $F^\ast$ agree -- $F_\ast$ uses a totalization instead of a geometric realization.)

Fully faithfulness in (2) is now equivalent to $F_\sharp F^\ast\to \mathrm{id}$ being an equivalence. But by linearity, it suffices to check this on the unit of $\mathrm{Pr}_{\mathcal D,X}$. Then $F_\sharp F^\ast 1$ is given by the object of
\[
\mathrm{Pr}_{\mathcal D,X} = \mathrm{Mod}_{\mathcal D_{/X}}(\mathrm{Fun}(\mathrm{Corr}(C^E_{/X}),\mathrm{Pr}_\kappa^L))
\]
given by $Y\mapsto \mathrm{colim}_! \mathcal D(X'^{\times_X \bullet}\times_X Y)$. Thus, we have to see that
\[
\mathrm{colim}_! \mathcal D(X'^{\times_X \bullet}\times_X Y)\to \mathcal D(Y)
\]
is an isomorphism for all $Y\to X$ in $E$; but this is exactly condition (1) under identifying colimits in $\mathrm{Pr}_\kappa^L$ with limits along right adjoints.

On the other hand, fully faithfulness is also equivalent to $\mathrm{id}\to F_\ast F^\ast$ being an equivalence. Applying this to the unit, we see that $\ast$-descent along $f$ follows from (2). Also, as the left adjoints on the level of $\mathrm{Pr}_{\mathcal D}$ commute with any base change, it follows that this condition is stable under any base change. Thus, (2) implies that $f$ satisfies universal $\ast$-descent and universal $!$-descent, i.e.~(4). But (4) evidently implies (1), so (1), (2) and (4) are equivalent.

Finally, it is clear that (3) implies (2), so it remains to see that (2) implies (3). For this, we have to see that the transformation $\mathrm{id}\to F^\ast F_\sharp$ is an isomorphism. But this can be checked after pullback along $X'\to X$, by the fully faithfulness of (2). As everything commutes with base change, this reduces us to the case that the map $X'\to X$ admits a section. But then descent is automatic.

Note that all the arguments work as well for a family of maps $\{X'_i\to X\}_i$.
\end{proof}

In fact, one can see that if we let $2\mathrm{Pr}_\kappa^L = \mathrm{Mod}_{\mathrm{Pr}_\kappa^L}(\mathrm{Pr}_\kappa^L)$ be the $\infty$-category of $\mathrm{Pr}_\kappa^L$-linear $\kappa$-presentable $\infty$-categories, the construction $X\mapsto \mathrm{Pr}_{\mathcal D,X}\in 2\mathrm{Pr}_\kappa^L$ assembles into a lax symmetric monoidal functor
\[
\mathrm{Pr}_{\mathcal D}: \mathrm{Corr}(C,E)\to 2\mathrm{Pr}_\kappa^L.
\]
In fact, we have already constructed a lax symmetric monoidal functor
\[
C^{\mathrm{op}}\to 2\mathrm{Pr}_\kappa^L: X\mapsto \mathrm{Pr}_{\mathcal D,X}
\]
and observed that for all $f\in E$, the pullback functors admit linear left adjoints satisfying base change. We can thus apply the construction of six-functor formalisms with respect to $I=E$ (and $P$ being just the isomorphisms). (We could also swap $I$ and $P$ here, as left and right adjoints agree in this formalism.)

One can then repeat the same construction one categorical level higher, and inductively define lax symmetric monoidal functor
\[
n\mathrm{Pr}_{\mathcal D}: \mathrm{Corr}(C,E)\to (n+1)\mathrm{Pr}_\kappa^L
\]
where inductively $(n+1)\mathrm{Pr}_\kappa^L = \mathrm{Mod}_{n\mathrm{Pr}_\kappa^L}(\mathrm{Pr}_\kappa^L)$.\footnote{This notion of $\kappa$-presentable $(\infty,n)$-categories is due to Aoki \cite{AokiPresentable}, building on ideas of Stefanich \cite{StefanichCatSp}.} For a fixed $X\in C$, these assemble into a categorical spectrum
\[
(\mathcal D(X),\mathrm{Pr}_{\mathcal D,X},2\mathrm{Pr}_{\mathcal D,X},\ldots)
\]
in the sense of Stefanich \cite{StefanichCatSp}, see also \cite{StefanichSheafGestalt} for a detailed account of this constrution.

\newpage

\section*{Appendix to Lecture V: Passage to stacks}

The main goal of the work of Liu--Zheng \cite{LiuZhengArtin} was to extend the \'etale $6$-functor formalism to stacks. Some of their work has been streamlined by Mann \cite[Appendix A.5]{MannThesis}. Let us present a slightly different take on this passage to stacks. Our discussion here strongly relies on extended discussions of the author with Clausen.\footnote{Heyer--Mann give a version of this construction in \cite[Theorem 3.4.11]{HeyerMann}, noting some oversights of the first version here in \cite[Remark 3.4.12]{HeyerMann}. These are corrected now. On the other hand, using the presentable $(\infty,2)$-categories we have been able to simplify the discussion.}

Start with any $\infty$-category $C$ with all finite limits, and a class of morphisms $E$ as usual (in particular, stable under diagonals). We make the additional assumption that $C$ has finite disjoint unions, and that $E$ is stable under taking finite disjoint unions, and contains the maps $\emptyset\to X$ and $X\sqcup X\to X$ for any $X\in C$. Let
\[
D: \mathrm{Corr}(C,E)\to \mathrm{Pr}^L
\]
be a presentable $6$-functor formalism. We also assume that $D$ takes finite disjoint unions to products.

Our goal is to introduce a Grothendieck topology on $C$ such that $D$ extends to a $6$-functor formalism on all sheaves (of anima) on $C$. Our principle will be to work with the finest possible topology, and try to find the largest possible class of morphisms for which the $!$-functors are defined. In the process, some of the conditions we impose may look a bit artificial; we will later discuss how to verify some of the conditions (like being a $D$-cover) in practice, see the second appendix to Lecture VI.

First, we extend $D$ to all (small) presheaves of anima on $C$ via right Kan extension, i.e.~for any presheaf of anima $\tilde{X}$, we let
\[
D(\tilde{X}) = \mathrm{lim}_{X\to \tilde{X},X\in C} D(X).
\]
This is still a symmetric monoidal presentable stable $\infty$-category. By \cite[Proposition A.5.16]{MannThesis}, one can moreover extend this to a $6$-functor formalism with respect to the maps of presheaves of anima any of whose pullbacks to $C$ are representable by a morphism in $E$.

\begin{definition}\label{def:Dtopology} Consider a family of maps $\{f_i: X_i\to Y\}_i$ in $C$.

The $f_i$ \emph{form a $D$-cover} if all $f_i\in E$, they form a cover in the canonical topology, and satisfy universal $!$-descent.
\end{definition}

Recall that forming a cover in the canonical topology means that for all $Z\in C$, the presheaf $\mathrm{Hom}(-,Z)$ satisfies the sheaf condition with respect to any pullback of $\{f_i: X_i\to Y\}_i$ along a map $Y'\to Y$ in $C$.

As $C$ admits finite disjoint unions, the maps $\{X_i\to Y\}_i$ for finite disjoint unions $Y=\bigsqcup_i X_i$ form a $D$-cover.

\begin{remark} It seems artificial that one has to impose that the $D$-topology is subcanonical -- one often has morphisms in $C$ which are inverted by $D$, for example the morphism $X_{\mathrm{red}}\to X$ of schemes under the \'etale $6$-functor formalism. In that case, one would expect that one can allow $X_{\mathrm{red}}\to X$ as a cover in the $D$-topology, but this is precluded by the requirement of being subcanonical. One way around is to restrict to reduced schemes from the start (or even absolutely weakly normal schemes, giving a canonical representative in each universal homeomorphism class).
\end{remark}

\begin{remark} By Theorem~\ref{thm:shriekimpliesstar}, universal $!$-descent implies universal $\ast$-descent, and to prove universal $!$-descent, it is enough to prove $!$-descent after pullback along maps $Y'\to Y$ in $E$.
\end{remark}

Now let $\tilde{C}$ be the $\infty$-category of sheaves of anima on $C$ for the $D$-topology. The functor $X\mapsto D(X)$ to symmetric monoidal presentable $\infty$-categories is a sheaf with respect to the $D$-topology, and thus gives a functor on $\tilde{C}$; for all $\tilde{X}\in \tilde{C}$, one still has
\[
D(\tilde{X}) = \mathrm{lim}_{X\to \tilde{X},X\in C} D(X).
\]
Let $\tilde{E}_0$ be the class of maps $\tilde{f}: \tilde{X}\to \tilde{Y}$ such that for any $Y\in C$ and $Y\to \tilde{Y}$, the pullback $f: X=\tilde{X}\times_{\tilde{Y}} Y\to Y$ satisfies $X\in C$ and $f\in E$.\footnote{For this condition to make sense, we need to make the $D$-topology subcanonical.} The following result is an application of right Kan extension.

\begin{proposition}[{\cite[Proposition A.5.16]{MannThesis}}]\label{prop:extendtostacks} The $6$-functor formalism $D$ on $(C,E)$ extends uniquely to a $6$-functor formalism on $(\tilde{C},\tilde{E}_0)$ such that the associated functor from $\tilde{C}^{\mathrm{op}}$ to symmetric monoidal presentable $\infty$-categories is given by the above recipe
\[
D(\tilde{X}) = \mathrm{lim}_{X\to \tilde{X},X\in C} D(X).
\]
\end{proposition}

We want to extend from $\tilde{E}_0$ to a larger class of morphisms $\tilde{E}$. Here are some desirable conditions.

\begin{definition} Let $\tilde{E}\supset \tilde{E}_0$ be a class of morphisms of $\tilde{C}$ that is stable under pullback and composition.
\begin{enumerate}
\item The class $\tilde{E}$ is \emph{stable under disjoint unions} if whenever $\tilde{f}_i: \tilde{X}_i\to \tilde{Y}$ are morphisms in $\tilde{E}$, then also $\bigsqcup_i \tilde{f}_i: \bigsqcup_i \tilde{X}_i\to \tilde{Y}$ is in $\tilde{E}$.
\item The class $\tilde{E}$ is \emph{local on the target} if whenever $\tilde{f}: \tilde{X}\to \tilde{Y}$ is a morphism in $\tilde{C}$ such that for all $Y\in C$ with a map $Y\to \tilde{Y}$, the pullback $\tilde{X}\times_{\tilde{Y}} Y\to Y$ lies in $\tilde{E}$, then $\tilde{f}\in \tilde{E}$.
\item Assume that $D$ extends uniquely from $(\tilde{C},\tilde{E}_0)$ to $(\tilde{C},\tilde{E})$. The class $\tilde{E}$ is \emph{local on the source} if whenever $\tilde{f}: \tilde{X}\to \tilde{Y}$ is a morphism in $\tilde{C}$ such that there is some surjective map $\tilde{g}: \tilde{X}'\to \tilde{X}$ in $\tilde{E}$ satisfying universal $!$-descent, and such that $\tilde{f}\tilde{g}$ lies in $\tilde{E}$, then $\tilde{f}\in \tilde{E}$.
\item Assume that $D$ extends uniquely from $(\tilde{C},\tilde{E}_0)$ to $(\tilde{C},\tilde{E})$. The class $\tilde{E}$ is \emph{tame} if whenever $Y\in C$ and $\tilde{f}: \tilde{X}\to Y$ is a map in $\tilde{E}$, then there are morphisms $h_i: X_i\to Y$ in $E$ and a surjective morphism $\bigsqcup_i X_i\to \tilde{X}$ over $Y$ that lies in $\tilde{E}$ and satisfies universal $!$-descent.
\end{enumerate}
\end{definition}

\begin{remark} The condition of being local on the target may seem stricter than expected, as after pullback to $Y\in C$, no further localization in $Y$ is allowed. But if $\tilde{E}$ is both local on the target and local on the source, then to check whether a morphism lies in $\tilde{E}$ it suffices to check it after pullback along an $\tilde{E}$-map $Y'\to Y$ that is of universal $!$-descent. Thus, for the formulation of the following theorem, one could also take this stronger condition of being ``local on the target''.
\end{remark}

\begin{theorem}\label{thm:stacky6functors} There is a minimal collection of morphisms $\tilde{E}\supset \tilde{E}_0$ of $\tilde{C}$ such that $D$ extends uniquely from $(\tilde{C},\tilde{E}_0)$ to $(\tilde{C},\tilde{E})$, and such that $\tilde{E}$ is stable under disjoint unions, local on the target, local on the source, and tame.
\end{theorem}

\begin{proof} Consider the class $A$ of all possible collections of morphisms $\tilde{E}$ such that $D$ extends uniquely to $(\tilde{C},\tilde{E})$, and which are tame. Note that $\tilde{E}_0$ is an example (and minimal). We note that any filtered union of such collections is again such a collection. Given any such collection $\tilde{E}$, one can consider the collection $\tilde{E}'$ of all morphisms $\tilde{f}: \tilde{X}\to \tilde{Y}$ such that $\tilde{X}$ can be written as a disjoint union of $\tilde{X}_i$ with all $\tilde{X}_i\to \tilde{Y}$ being in $\tilde{E}$. This class is stable under pullback and composition, it stays tame, and by \cite[Proposition A.5.12]{MannThesis} the $6$-functor formalism $D$ extends uniquely to $(\tilde{C},\tilde{E}')$. Thus, any class $\tilde{E}\in A$ can be minimally enlarged to be stable under disjoint unions.

Also, if $\tilde{E}\in A$, we can define a new class of morphisms $\tilde{E}'$ as being those morphisms $\tilde{f}: \tilde{X}\to \tilde{Y}$ such that there is some surjective $\tilde{E}$-map $\tilde{g}: \tilde{X}'\to \tilde{X}$ of universal $!$-descent for which $\tilde{f}\tilde{g}$ lies in $\tilde{E}$. This class is stable under pullback and composition, and by \cite[Proposition A.5.14]{MannThesis} (an application of left Kan extension), the $6$-functor formalism $D$ extends uniquely to $(\tilde{C},\tilde{E}')$; here, the special class of covers is taken to be those of universal $!$-descent. Moreover, $\tilde{E}'$ stays tame. Iterating this procedure countably many times, one can enlarge any $\tilde{E}\in A$ in a minimal way to make it local on the source. Combined with the procedure of making it stable under disjoint unions, we can also arrange both properties in a minimal way.

Now assume that $\tilde{E}\in A$ is stable under disjoint unions and local on the source. Consider the class $\tilde{E}'$ of morphisms $\tilde{f}: \tilde{X}\to \tilde{Y}$ such that for all $Y\in C$ with a map $Y\to \tilde{Y}$, the pullback $\tilde{X}\times_{\tilde{Y}} Y\to Y$ lies in $\tilde{E}$. This class is clearly stable under pullback; we claim it is also stable under composition. For this, we show more generally that if $\tilde{Y}'\in \tilde{C}$ is an object that admits an $\tilde{E}$-map $\tilde{Y}'\to Y$ to an object of $C$, then for any map $\tilde{Y}'\to \tilde{Y}$, the pullback $\tilde{X}\times_{\tilde{Y}} \tilde{Y}'\to \tilde{Y}'$ lies in $\tilde{E}$. To see this, note that by tameness there is a cover of $\tilde{Y}'$ by an $\tilde{E}$-map $\bigsqcup_i Y_i\to \tilde{Y}'$ of universal $!$-descent with $Y_i\in C$ and $Y_i\to X$ in $E$. As $\tilde{E}$ is local on the source and stable under disjoint unions, to check that $\tilde{X}\times_{\tilde{Y}} \tilde{Y}'\to \tilde{Y}'$ lies in $\tilde{E}$, it suffices to check after pullback to the $Y_i$. But here it is true by assumption.

Now let $C'\subset \tilde{C}$ be the full subcategory of those objects that admit an $\tilde{E}$-map towards an object of $C\subset \tilde{C}$. Let $E'$ be the restriction of $\tilde{E}$ to $C'$. Then $D$ restricts to a $6$-functor formalism on $(C',E')$. Now we can apply \cite[Proposition A.5.16]{MannThesis} to extend $D$ uniquely from $(C',E')$ back to $(\tilde{C},\tilde{E}')$, as desired. (The uniqueness assertion of this extension, applied to the extension from $(C',E')$ to $(\tilde{C},\tilde{E})$, ensures that this really extends the formalism on $(\tilde{C},\tilde{E})$.) It is clear that $\tilde{E}'$ stays tame.

Thus, given some $\tilde{E}\in A$ that is stable under disjoint unions and local on the source, we can find a minimal $\tilde{E}'\in A$ containing $\tilde{E}$ that is local on the target. This process may destroy the other stability properties, but these can be retained through further transfinitely iterating these construction principles.
\end{proof}

We note that in the proof of the theorem, the $6$-functor formalism is alternatingly extended via left Kan extension and via right Kan extension; this makes it difficult to formulate a direct universal property for this extended $6$-functor formalism.

Still, the constructions of this appendix are functorial in maps of six-functor formalisms. More precisely, consider a second presentable six-functor formalism
\[
D': \mathrm{Corr}(C,E')\to \mathrm{Pr}^L
\]
with $E'\supset E$, and a lax symmetric monoidal transformation $F: D\to D'$.

\begin{proposition} In this situation, any $D$-cover is also a $D'$-cover, and hence one gets a functor
\[
\mathrm{Shv}_D(C)\to \mathrm{Shv}_{D'}(C)
\]
commuting with small colimits and finite limits. Moreover, for the extended $6$-functor formalism of Theorem~\ref{thm:stacky6functors}, any morphism in $\tilde{E}$ gets sent to a morphism in $\tilde{E'}$, and there is a lax symmetric monoidal transformation from the extended $6$-functor formalism
\[
\mathrm{Corr}(\mathrm{Shv}_D(C),\tilde{E})\to \mathrm{Pr}^L
\]
towards the composite
\[
\mathrm{Corr}(\mathrm{Shv}_D(C),\tilde{E})\to \mathrm{Corr}(\mathrm{Shv}_{D'}(C),\tilde{E'})\to \mathrm{Pr}^L.
\]
\end{proposition}

\begin{proof} The condition of being subcanonical is clearly preserved as the category is unchanged. We have to see that the property of universal $!$-descent is preserved. For this, let $C''=C\times \Delta^1$ and $E''=E\times\{1\}\sqcup E'\times\{0\}$, with
\[
D'': \mathrm{Corr}(C'',E'')\to \mathrm{Pr}^L
\]
encoding the lax symmetric monoidal transformation $D\to D'$. Then the preservation of universal $!$-descent follows from Theorem~\ref{thm:shriekimpliesstar}: This property for maps in $(C,E)$ implies it when these maps are embedded in $(C'',E'')$ via $(C,E)=(C\times \{1\},E\times \{1\})\subset (C'',E'')$, as one only has to check it for base changes along maps in $E''$, and maps in $E''$ with target in $C\times \{1\}$ already lie in $C\times \{1\}$. But then it also holds for the base change along $X\times \{0\}\to X\times \{1\}$, which is what we wanted to see.

Now the rest follows from the construction of Theorem~\ref{thm:stacky6functors} --- we apply the theorem first to $D'$, and then when we construct the extension of $D$, we note that at each step of the construction, the lax symmetric monoidal transformation towards the six-functor formalism on $\mathrm{Shv}_{D'}(\tilde{E'})$ gets preserved.
\end{proof}
\newpage

\section{Lecture VI: Complements on the abstract formalism}

This is the final lecture on the abstract theory of $6$-functor formalisms; starting from the next lecture, we will discuss various examples. In this lecture, I want to discuss some complements that the $2$-category from the previous lecture naturally leads to, and that give answers to the following natural questions:
\begin{enumerate}
\item For which morphisms is $f^\ast$ essentially $f^!$? For which morphisms is $f_!$ essentially $f_\ast$?
\item What are reasonable finiteness conditions that one can put on objects $A\in \mathcal D(X)$? For example, which conditions ensure that $A$ maps isomorphically to its double Verdier dual?
\end{enumerate}

As in the last lecture, we will fix some $\infty$-category $C$ admitting all finite limits, and assume that $E$ consists of all morphisms of $C$. Let $\mathrm{Corr}(C):=\mathrm{Corr}(C,\mathrm{all})$ be the resulting symmetric monoidal $\infty$-category of correspondences, and
\[
\mathcal D: \mathrm{Corr}(C)\to \mathrm{Cat}_\infty
\]
a lax symmetric monoidal functor. In this lecture, we will assume that it gives a $6$-functor formalism, i.e.~internal Homs, and the right adjoints $f_\ast$ and $f^!$ (of $f^\ast$ and $f_!$) exist for all $f$.

As in the last lecture, we get the $2$-category $\mathrm{LZ}_{\mathcal D}$. Note that $\mathrm{LZ}_{\mathcal D}$ is equivalent to $\mathrm{LZ}_{\mathcal D}^{\mathrm{op}}$ (basically, as $\mathrm{Corr}(C)\cong \mathrm{Corr}(C)^{\mathrm{op}}$). Generalizing the idea of using adjunctions in $\mathrm{LZ}_{\mathcal D}$, we are led to consider the following notions.\footnote{The names ``suave'' and ``prim'' were introduced in \cite{HeyerMann}, and replace the (overloaded) terms ``smooth'' and ``proper'' originally used in these notes.}

\begin{definition}\label{def:suaveprimsheaves} Let $Y\in C$ be the final object, $f: X\to Y$ some morphism, and $A\in D(X)$.
\begin{enumerate}
\item The object $A\in D(X)$ is $f$-suave if $A\in \mathrm{Hom}_{\mathrm{LZ}_{\mathcal D}}(X,Y)$ is a left adjoint in $\mathrm{LZ}_{\mathcal D}$.
\item The object $A\in D(X)$ is $f$-prim if $A\in \mathrm{Hom}_{\mathrm{LZ}_{\mathcal D}}(Y,X)$ is a left adjoint in $\mathrm{LZ}_{\mathcal D}$.
\end{enumerate}
\end{definition}

\begin{remark} Both properties are stable under any base change, using that base change induces functors of $2$-categories and hence preserves adjunctions. Here, if $f': X'\to Y'$ is a morphism to some other object $Y'$, then we first replace $C$ by the slice $C_{/Y'}$ and then apply the previous definition to define $f'$-suave and $f'$-prim objects $A'\in D(X')$.
\end{remark}

\begin{remark}\label{rem:opposite3functors} If $\mathcal D: \mathrm{Corr}(C)\to \mathrm{Cat}_\infty$ is a $3$-functor formalism, then also $\mathcal D^{\mathrm{op}}: \mathrm{Corr}(C)\to \mathrm{Cat}_\infty$, sending any $X$ to the opposite $\infty$-category $\mathcal D(X)^{\mathrm{op}}$, is a $3$-functor formalism (as passing to opposite $\infty$-categories is a (covariant!) symmetric monoidal self-equivalence of $\mathrm{Cat}_\infty$). This abstract procedure exchanges $f$-suave and $f$-prim objects, and allows one to formally dualize some statements. Note, however, that this translation turns the right adjoints $f_\ast$, $f^!$, etc., into left adjoint functors; this explains some apparent asymmetries in the presentation below.
\end{remark}

\begin{remark} The second condition of ``$f$-primness'' was, to our knowledge, first suggested by Mann in \cite{MannNuclear}.
\end{remark}

Concretely, $A$ is $f$-suave if and only if there is some $B\in D(X)=\mathrm{Hom}_{\mathrm{LZ}_{\mathcal D}}(Y,X)$ together with maps
\[
\alpha: \Delta_!(1_X)\to p_1^\ast A\otimes p_2^\ast B, \beta: f_!(A\otimes B)\to 1_Y
\]
such that the composite
\[
B\cong p_{2!}(p_1^\ast B\otimes \Delta_!(1_X))\xrightarrow{\alpha} p_{2!}(p_1^\ast B\otimes p_1^\ast A\otimes p_2^\ast B)\cong p_{2!} p_1^\ast(B\otimes A)\otimes B\cong f^\ast f_!(B\otimes A)\otimes B\xrightarrow{\beta} f^\ast(1_Y)\otimes B\cong B
\]
as well as the similar one with $A$ and $B$ exchanged, is the identity. (As in the last lecture, to check that $A$ is $f$-suave, it suffices to check that these composites are isomorphisms, not necessarily equal to the identity.)

We note that this data is in fact symmetric in $A$ and $B$, which is a consequence of the self-duality of $\mathrm{LZ}_{\mathcal D}$ which ensures that the right adjoint of $A$ is itself $f$-suave, with right adjoint given again by $A$. We record some of these obversations below.

\begin{proposition}\label{prop:suavesheafproperties} Let $A\in D(X)$ be $f$-suave, with right adjoint $B\in D(X)$.
\begin{enumerate}
\item The object $B\in D(X)$ is $f$-suave, with right adjoint $A$.
\item There is a natural isomorphism of functors
\[
B\otimes f^\ast(-)\cong \mathscr{H}\mathrm{om}(A,f^!(-)): D(Y)\to D(X).
\]
In particular, $B\cong \mathscr{H}\mathrm{om}(A,f^!(1_Y)) = \mathbb D_f(A)$ is the (relative) Verdier dual of $A$.
\item The Verdier biduality map
\[
A\to \mathbb D_f(\mathbb D_f(A))
\]
is an isomorphism.
\item The formation of the Verdier dual $\mathbb D_f(A)$ commutes with any base change.
\end{enumerate}
\end{proposition}

In particular, (2) says that $f$-suave objects $A$ lead to isomorphisms between twisted versions of $f^\ast$ and $f^!$. Also note that by (1), one also has
\[
A\otimes f^\ast(-)\cong \mathscr{H}\mathrm{om}(B,f^!(-)): D(Y)\to D(X).
\]

\begin{proof} We already saw (1). For (2), note that applying the functor $\mathrm{LZ}_{\mathcal D}\to \mathrm{Cat}: X\mapsto D(X)$, we see that $B\otimes f^\ast(-)$ is the right adjoint of $f_!(A\otimes -)$; but that right adjoint is $\mathscr{H}\mathrm{om}(A,f^!(-))$. For (3), apply (2) twice; note that the maps $A\to \mathbb D_f(B)$ and $B\to \mathbb D_f(A)$ are both induced by the pairing $\beta: f_!(A\otimes B)\to 1_Y$ (equivalently, $A\otimes B\to f^!(1_Y)$). For (4), use that the formation of right adjoints (when they exist) commutes with any functor of $2$-categories, in particular with base change.
\end{proof}

We see that $B$ is necessarily given by the Verdier dual $\mathbb D_f(A)$ of $A$. If one takes this as the definition of $B$, one automatically gets the map $\beta: f_!(A\otimes B)\to 1_Y$. One can then wonder what it takes to supply the map $\alpha: \Delta_!(1_X)\to p_1^\ast A\otimes p_2^\ast B$. This is analyzed in the next proposition.

\begin{proposition}\label{prop:checksuavesheaf} Let $A\in D(X)$. Then $A$ is $f$-suave if and only if the natural map
\[
p_1^\ast A\otimes p_2^\ast\mathbb D_f(A)\to \mathscr{H}\mathrm{om}(p_2^\ast A,p_1^! A)
\]
is an isomorphism. In fact, it suffices to see that it induces an isomorphism on global sections after applying $\Delta^!$.
\end{proposition}

Here, the natural map is the one adjoint to the natural map
\[
p_1^\ast A\otimes p_2^\ast \mathbb D_f(A)\otimes p_2^\ast A\to p_1^\ast A\otimes p_2^\ast f^!(1_Y)\to p_1^\ast A\otimes p_1^!(1_X)\to p_1^! A.
\]

\begin{proof} If $A$ is $f$-suave, then for any base change $f': X'\to Y'$ along $g: Y'\to Y$ (with base change $g': X'\to X$) and any $B\in D(Y')$, the map
\[
f^{\prime\ast} B\otimes g^{\prime\ast} \mathbb D_f(A)\to \mathscr{H}\mathrm{om}(g^{\prime\ast} A,f^{\prime !} B)
\]
is an isomorphism. Applying this with $g: Y'\to Y$ given by $f: X\to Y$ and $B=A$ yields the desired isomorphism.

For the converse, we need to find the map $\alpha: \Delta_!(1_X)\to p_1^\ast A\otimes p_2^\ast\mathbb D_f(A)$. Conveniently, the hypothesis of the proposition shows that the target is naturally isomorphic to $\mathscr{H}\mathrm{om}(p_2^\ast A,p_1^! A)$, so it suffices to find a natural map
\[
\Delta_!(1_X)\to \mathscr{H}\mathrm{om}(p_2^\ast A,p_1^! A).
\]
By adjunction, this amounts to a map $\Delta_!(1_X)\otimes p_2^\ast A=\Delta_!(A)\to p_1^! A$, which by a further adjunction amounts to a morphism $p_{1!}\Delta_! A=A\to A$, where we take the identity map. It remains to see that two diagrams commute, which we leave as an exercise.

For the final sentence, note that the construction of the previous paragraph did not in fact require that the displayed map is an isomorphism, but only that it becomes one after applying global sections to $\Delta^!$ of this map.
\end{proof}

\begin{remark} This discussion may seem like formal nonsense, but when writing \cite{FarguesScholze}, I was for a long time unsuccessfully trying to prove that certain $A\in D(X)$ satisfying the condition of Proposition~\ref{prop:checksuavesheaf} also satisfy Verdier biduality. The work of Lu--Zheng \cite{LuZheng} then trivialized this problem!
\end{remark}

Let us now look at the ``dual'' case of $f$-prim $A\in D(X)$. In this case, the right adjoint $B\in D(X)$ comes with maps
\[
\alpha: p_1^\ast A\otimes p_2^\ast B\to \Delta_!(1_X), \beta: 1_Y\to f_!(A\otimes B),
\]
satisfying similar commutative diagrams. The analogue of Proposition~\ref{prop:suavesheafproperties} is the following.

\begin{proposition}\label{prop:primsheafproperties} Let $A\in D(X)$ be $f$-prim, with right adjoint $B\in D(X)$.
\begin{enumerate}
\item The object $B\in D(X)$ is $f$-prim, with right adjoint $A$.
\item There is a natural isomorphism of functors
\[
f_!(B\otimes -)\cong f_\ast\mathscr{H}\mathrm{om}(A,-): D(X)\to D(Y).
\]
\item The pairing
\[
\alpha: p_1^\ast A\otimes p_2^\ast B\to \Delta_!(1_X)
\]
induces isomorphisms
\[
B\cong p_{2\ast} \mathscr{H}\mathrm{om}(p_1^\ast A,\Delta_!(1_X)), A\cong p_{1\ast} \mathscr{H}\mathrm{om}(p_2^\ast B,\Delta_!(1_X)).
\]
\end{enumerate}
\end{proposition}

Of course, one could also formulate an analogue of Proposition~\ref{prop:suavesheafproperties}~(4), that the dual in the sense of (3) commutes with any base change.

\begin{proof} The proof is identical to the proof of Proposition~\ref{prop:suavesheafproperties}. To get the (first) isomorphism in (3), apply the result of (2) to the base change of $f: X\to Y$ along the map $g: Y'\to Y$ given by $f: X\to Y$, and the sheaf $\Delta_!(1_X)$.
\end{proof}

We note that the dual in part (3) simplifies considerably in case $\Delta_!(1_X)=\Delta_\ast(1_X)$, as is often the case (and conditions for which we will discuss momentarily). Indeed, in that case
\[
p_{2\ast} \mathscr{H}\mathrm{om}(p_1^\ast A,\Delta_!(1_X))=p_{2\ast} \mathscr{H}\mathrm{om}(p_1^\ast A,\Delta_\ast(1_X)) = p_{2\ast} \Delta_\ast \mathscr{H}\mathrm{om}(\Delta^\ast p_1^\ast A,1_X) = \mathscr{H}\mathrm{om}(A,1_X)
\]
is simply the naive dual of $A\in D(X)$. In particular, if $A=1_X$ is the unit, then this naive dual is $1_X$.

We also have an analogue of Proposition~\ref{prop:checksuavesheaf}.

\begin{proposition}\label{prop:checkprimsheaf} Let $A\in D(X)$. Then $A$ is $f$-prim if and only if the natural map
\[
f_!(A\otimes p_{2\ast}\mathscr{H}\mathrm{om}(p_1^\ast A,\Delta_!(1_X)))\to f_\ast\mathscr{H}\mathrm{om}(A,A)
\]
is an isomorphism on global sections.
\end{proposition}

In fact, one can construct a more general natural transformation
\[
f_!(-\otimes p_{2\ast}\mathscr{H}\mathrm{om}(p_1^\ast A,\Delta_!(1_X)))\to f_\ast\mathscr{H}\mathrm{om}(A,-)
\]
of functors $D(X)\to D(Y)$. This is adjoint to
\[\begin{aligned}
f^\ast f_!(-\otimes p_{2\ast}\mathscr{H}\mathrm{om}(p_1^\ast A,\Delta_!(1_X)))&\cong p_{1!}(p_2^\ast(-)\otimes p_2^\ast p_{2\ast} \mathscr{H}\mathrm{om}(p_1^\ast A,\Delta_!(1_X)))\\
&\to p_{1!}(p_2^\ast(-)\otimes \mathscr{H}\mathrm{om}(p_1^\ast A,\Delta_!(1_X)))\to \mathscr{H}\mathrm{om}(A,-)
\end{aligned}\]
which in turn is adjoint to
\[
p_2^\ast(-)\otimes \mathscr{H}\mathrm{om}(p_1^\ast A,\Delta_!(1_X))\to p_1^!\mathscr{H}\mathrm{om}(A,A)\cong \mathscr{H}\mathrm{om}(p_1^\ast A,p_1^!(-))
\]
which in turn is adjoint to
\[
p_2^\ast(-)\otimes \mathscr{H}\mathrm{om}(p_1^\ast A,\Delta_!(1_X))\otimes p_1^\ast A\to p_2^\ast(-)\otimes \Delta_!(1_X)\to p_1^!(-)
\]
where the latter map is adjoint to the isomorphism
\[
p_{1!}(p_2^\ast(-)\otimes \Delta_!(1_X))\cong (-).
\]

\begin{proof} The proof is similar to the proof of Proposition~\ref{prop:checksuavesheaf}.
\end{proof}

Let us now come back to the question of when $f_!$ and $f_\ast$ agree.

\begin{definition}\label{def:cohomproper} A map $f: X\to Y$ in $C$ is cohomologically proper if it is $n$-truncated for some $n$, the object $1_X\in D(X)$ is $f$-prim, and the diagonal $\Delta_f: X\to X\times_Y X$ is cohomologically proper (or an isomorphism).
\end{definition}

This definition may seem self-referential, but as $\Delta_f$ is $n-1$-truncated, it works by induction on $n$ (with $\Delta_f$ being an isomorphism as the base case).

\begin{proposition}\label{prop:checkcohomproper} Let $f: X\to Y$ be a morphism in $C$ such that $\Delta_f$ is cohomologically proper. Then there is a natural transformation $f_!\to f_\ast$ of functors $D(X)\to D(Y)$. It is an isomorphism if and only if $f$ is cohomologically proper. To check that $f$ is cohomologically proper, it suffices to check that $f_!(1_X)\to f_\ast(1_X)$ is an isomorphism on global sections.
\end{proposition}

\begin{proof} We argue by induction on $n$ such that $f$ is $n$-truncated. Thus, we can assume that $\Delta_{f!}\cong \Delta_{f\ast}$. In particular, the map of Proposition~\ref{prop:checkprimsheaf} applied to $1_X\in D(X)$ reduces to a map $f_!(1_X)\to f_\ast(1_X)$, which is an isomorphism on global sections if and only if $f$ is cohomologically proper; and the construction in fact gives a natural transformation $f_!\to f_\ast$ of functors $D(X)\to D(Y)$. If $f$ is indeed cohomologically proper, then it follows that $f_!\cong f_\ast$ as functors $D(X)\to D(Y)$.
\end{proof}

In particular, we see that the natural transformation $f_!\to f_\ast$ that often exists in practice (for separated maps), and that often is encoded as extra data in a $6$-functor formalism, notably in the approach of Gaitsgory--Rozenblyum, can in fact be constructed directly out of the data that we have fixed.

To restore the symmetry, let us also single out the corresponding class of cohomologically smooth morphisms, which we call cohomologically \'etale (as usually for \'etale morphisms, $f^\ast=f^!$ without twist).

\begin{definition}\label{def:cohometale} A map $f: X\to Y$ in $C$ is cohomologically \'etale if it is $n$-truncated for some $n$, the object $1_X\in D(X)$ is $f$-suave, and the diagonal $\Delta_f: X\to X\times_Y X$ is cohomologically \'etale (or an isomorphism).
\end{definition}

\begin{proposition}\label{prop:checkcohometale} Let $f: X\to Y$ be a morphism in $C$ such that $\Delta_f$ is cohomologically \'etale. Then there is a natural transformation $f^!\to f^\ast$ of functors $D(Y)\to D(X)$. It is an isomorphism if and only if $f$ is cohomologically \'etale. To check that $f$ is cohomologically \'etale, it suffices to check that $f^!(1_Y)\to 1_X$ is an isomorphism on global sections.
\end{proposition}

\begin{proof} This is the dual (in the sense of Remark~\ref{rem:opposite3functors}) of Proposition~\ref{prop:checkcohomproper}. Note that Proposition~\ref{prop:checkcohomproper} refers to the right adjoint $f_\ast$ of $f^\ast$, which translates badly. But what used to be a transformation $f_!\to f_\ast$ becomes a transformation from a hypothetical left adjoint of $f^\ast$ to $f_!$; but by passing to right adjoints, this manifests itself again as a transformation from $f^!$ to $f^\ast$.

Alternatively, follow the steps of the proof of Proposition~\ref{prop:checkcohomproper} and translate them accordingly.
\end{proof}

As a final topic, let us discuss how the conditions of being $f$-suave and $f$-prim relate to finiteness conditions one can put on $A\in D(X)$ as an object of the category $D(X)$. There are two notable conditions: As $D(X)$ is symmetric monoidal, there is the condition of being dualizable; and if $\mathcal D(X)$ has all colimits, one can ask that $A\in \mathcal D(X)$ is compact (i.e.~$\mathrm{Hom}(A,-)$ commutes with all filtered colimits).

\begin{proposition}\label{prop:suaveoverid} Let $f: X=Y\to Y$ be the identity. Then $A\in D(Y)$ is $f$-suave if and only if $A$ is $f$-prim if and only if $A$ is dualizable.
\end{proposition}

\begin{proof} This is immediate from the definitions.
\end{proof}

\begin{proposition}\label{prop:suavesheafcompact} Let $f: X\to Y$ be a morphism and let $A\in \mathcal D(X)$ be $f$-suave. Assume that $\Delta_!(1_X)\in \mathcal D(X\times_Y X)$ is compact. Then $A\in \mathcal D(X)$ is compact.
\end{proposition}

The condition that $\Delta_!(1_X)$ is compact is satisfied in many, but not all, situations, and there are important situations where $f$-suave objects are quite far from compact objects.

\begin{proof} For $B\in \mathcal D(X)$, we have an isomorphism
\[
p_1^\ast B\otimes p_2^\ast \mathbb D_f(A)\cong \mathscr{H}\mathrm{om}(p_2^\ast A,p_1^! B).
\]
We apply $\mathrm{Hom}(\Delta_!(1_X),-)$ to this, getting an isomorphism
\[\begin{aligned}
\mathrm{Hom}(\Delta_!(1_X),p_1^\ast B\otimes p_2^\ast \mathbb D_f(A))&\cong \mathrm{Hom}(\Delta_!(1_X),\mathscr{H}\mathrm{om}(p_2^\ast A,p_1^! B))\\
&\cong \mathrm{Hom}(\Delta_!(1_X)\otimes p_2^\ast A,p_1^! B)\\
&\cong \mathrm{Hom}(\Delta_!(A),p_1^! B)\cong \mathrm{Hom}(p_{1!} \Delta_!(A),B)\cong \mathrm{Hom}(A,B).
\end{aligned}\]
But by assumption, the functor $B\mapsto \mathrm{Hom}(\Delta_!(1_X),p_1^\ast B\otimes p_2^\ast \mathbb D_f(A))$ commutes with all filtered colimits.
\end{proof}

Again, there is an analogue for $f$-prim objects, this time under a much weaker assumption.

\begin{proposition}\label{prop:primsheafcompact} Let $f: X\to Y$ be a morphism and let $A\in \mathcal D(X)$ be $f$-prim. Assume that $1_Y\in \mathcal D(Y)$ is compact. Then $A\in \mathcal D(X)$ is compact.
\end{proposition}

\begin{proof} By $f$-primness, we have an isomorphism
\[
f_!(p_{2\ast}\mathscr{H}\mathrm{om}(p_1^\ast A,\Delta_!(1_X))\otimes B)\cong f_\ast \mathscr{H}\mathrm{om}(A,B).
\]
Applying $\mathrm{Hom}(1_Y,-)$, we get
\[
\mathrm{Hom}(1_Y,f_!(p_{2\ast}\mathscr{H}\mathrm{om}(p_1^\ast A,\Delta_!(1_X))\otimes B))\cong \mathrm{Hom}(A,B).
\]
But, as a functor of $B$, the left-hand side commutes with all filtered colimits.
\end{proof}

Finally, we note the following stability of these conditions under retracts.

\begin{proposition}\label{prop:fsuaveretract} Let $f: X\to Y$ and $f': X'\to Y$ be morphisms, and assume that $X'$ is a retract of $X$ over $Y$, i.e.~there are $i: X'\to X$ and $r: X\to X'$ over $Y$ with $ri=\mathrm{id}_{X'}$. Let $A'\in \mathcal D(X')$ and $A\in \mathcal D(X)$ be such that there are maps $i^\ast A\to A'$ and $r^\ast A'\to A$ such that the composite $A'=i^\ast r^\ast A'\to i^\ast A\to A'$ is the identity; or else that there are maps $i_! A'\to A$ and $r_! A\to A'$ such that the composite $A'=r_! i_! A'\to r_! A\to A'$ is the identity. Then if $A\in \mathcal D(X)$ is $f$-suave, also $A'$ is $f'$-suave.
\end{proposition}

Again, there is a similar result in the $f$-prim case; as usual, the two cases are swapped under passing to opposite categories, so one should have maps $A'\to i^\ast A$ and $A\to r^\ast A'$ etc.

\begin{proof} We consider the case first scenario. Let $e=ir: X\to X$ be the idempotent endomorphism of $X$. We note that $A$ comes equipped with a map $h: e^\ast A\to A$ given as the composite $r^\ast i^\ast A\to r^\ast A'\to A$; and this map is idempotent in the sense that the composite
\[
e^\ast A = e^\ast e^\ast A\xrightarrow{e^\ast h} e^\ast A\xrightarrow{h} A
\]
agrees with $h$. Conversely, given $A$ with such an ``idempotent'' $h: e^\ast A\to A$, we get an actual idempotent endomorphism $i^\ast h: i^\ast A = i^\ast e^\ast A\to i^\ast A$, and letting $A'$ be the corresponding retract of $i^\ast A$, we get a map $A'\to i^\ast A$ as well as a map $r^\ast A'\to r^\ast i^\ast A\to A$. This way, the datum of $A$ and $A'$ and the maps $i^\ast A\to A'$ and $r^\ast A'\to A$ is equivalently captured in the datum of $A$ and $h: e^\ast A\to A$.

In other words, in $\mathrm{LZ}_{\mathcal D}$ (where as usual we assume that $\mathcal C$ has final object $Y$), we have the $1$-morphism $A: X\to Y$ and $e: X\to X$ as well as the $2$-morphism $h: Ae\Rightarrow A$. If $B$ is the $f$-Verdier dual of $A$ (i.e., the right adjoint morphism $B: Y\to X$), then $h$ corresponds to a $2$-morphism $e\Rightarrow BA$. If we denote the self-anti-equivalence of $\mathrm{LZ}_{\mathcal D}$ by a transpose $-^t$, then this is also the same as a morphism $e^t\Rightarrow A^t B^t$ (recall that $-^t$ is contravariant for $1$-morphisms but covariant for $2$-morphisms). Now $A^t$ is right adjoint to $B^t$, so this can be further translated into $B^t e^t\Rightarrow B^t$. Concretely, this is a morphism $h': e_! B\to B$ in $\mathcal D(X)$. Again, this morphism is ``idempotent'', and can equivalently be described in terms of $B'\in \mathcal D(X')$ with maps $i_! B'\to B$ and $r_! B\to B'$ whose composite $B'=r_! i_! B\to r_! B\to B'$ is the identity. We note that this whole process could also be reversed, starting from the second type of data and producing the first on the Verdier dual side.

It is now some formal $2$-categorical nonsense to obtain the adjunction between $A': X'\to Y$ and $B': Y\to X'$ as a ``retract'' of the adjunction between $A: X\to Y$ and $B: Y\to X$.
\end{proof}
\newpage

\section*{Appendix to Lecture VI: Uniqueness of $f_!$}

In Lecture IV, we stated a construction principle for $3$-functor formalisms, or more precisely for constructing $f_!$ when given $\otimes$ and $f^\ast$. In fact, using the notions of cohomologically \'etale and proper morphisms introduced in this lecture, there is a uniqueness theorem for it. This was conjectured in the original version of these notes and has now been proved by Dauser--Kuijper.

\begin{theorem}[{\cite[Theorem 3.3]{DauserKuijper}}]\label{thm:dauserkuijper} Assume that $C$ is some $\infty$-category admitting all finite limits. Let $E$, $I$ and $P$ be classes of morphisms of $C$ satisfying hypothesis (1) from Lecture IV, and assume that all morphisms in $E$ are $n$-truncated for some $n$ (that may depend on the morphism).

The following two $\infty$-categories are equivalent via the forgetful functor from (2) to (1).

\begin{enumerate}
\item The $\infty$-category of functors
\[
\mathcal D_0: C^{\mathrm{op}}\to \mathrm{CMon}(\mathrm{Cat}_\infty)
\]
satisfying the properties from Lecture IV; and where morphisms $\mathcal D_0\to \mathcal D_0'$ are those natural transformations that also commute with the functors $j_!$ for $j\in I$ and $f_\ast$ for $f\in P$ (which is just a condition, as these are adjoints).
\item The $\infty$-category of lax symmetric monoidal functors
\[
\mathcal D: \mathrm{Corr}(C,E)\to \mathrm{Cat}_\infty
\]
(i.e., $3$-functor formalisms) with the property that all morphisms in $I$ are cohomologically \'etale, and all morphisms in $P$ are cohomologically proper; and morphisms $\mathcal D\to \mathcal D'$ are lax symmetric monoidal natural transformations.
\end{enumerate}
\end{theorem}

In other words, given the classes $I$ and $P$, the functors $\otimes$ and $f^\ast$ determine the functors $f_!$, and all of their compatibilities, uniquely up to all higher coherences.\newpage

\section*{Appendix to Lecture VI: Suave or prim descent}

Assume that $D(X)$ is a presentable $\infty$-category for all $X\in C$. Let $f: X\to Y$ be a map in $E$. If $1_X$ is $f$-suave or $f$-prim, one can prove strong descent properties of $X\mapsto D(X)$. (Usually, when discussing things like Artin stacks, one uses a smooth topology; but the results here show that one could also allow proper covers, at least as far as the dicussion of $6$ functors is concerned.)

\begin{proposition}\label{prop:Dsuavedescent} Let $f: X\to Y$ be an $E$-morphism such that $1_X$ is $f$-suave. Then
\[
f^\ast: D(Y)\to D(X)
\]
is conservative if and only if the natural map
\[
\mathrm{colim}_{[n]\in \Delta^{\mathrm{op}}} f_{n!} f_n^! (1_Y)\to 1_Y
\]
is an isomorphism (where $f_n: X^{n/Y}\to Y$ are the fibre products), and this condition passes to any base change. In that case, the pullback functors
\[
(f_n^\ast)_n: D(Y)\to \mathrm{lim}_{[n]\in \Delta} D(X^{n/Y}), (f_n^!)_n: D(Y)\to \mathrm{lim}_{[n]\in \Delta} D(X^{n/Y})
\]
are equivalences. In particular, $f$ is of universal $!$-descent.
\end{proposition}

Recall here that if $1_X$ is $f$-suave, then $f^! = f^\ast \otimes f^!(1_Y)$, and $f^\ast = \mathscr{H}\mathrm{om}(f^!(1_Y),f^!)$. In particular, $f^\ast$ admits a left adjoint $f_!(f^!(1_Y)\otimes -)$, and the resulting counit transformation $f_!(f^!(1_Y)\otimes f^\ast)\to \mathrm{id}$ agrees with the counit transformation $f_!f^!\to \mathrm{id}$. These will be used implicitly in the proof.

\begin{proof} Assume first that
\[
(f_i^\ast)_i: D(Y)\to \prod_i D(X_i)
\]
is conservative. If $f$ has a section, then the map
\[
\mathrm{colim}_{[n]\in \Delta^{\mathrm{op}}} f_{n!} f_n^! (1_Y)\to 1_Y
\]
is an isomorphism for formal reasons. But under the conservativity assumption, we can in general reduce to this situation, using that the formation of $f^!$ commutes with any pullback when $1_X$ is $f$-suave.

In the converse direction, assume that this map is an isomorphism. In that case, the projection formula implies that
\[
\mathrm{colim}_{[n]\in \Delta^{\mathrm{op}}} f_{n!} f_n^! A\to A
\]
is an isomorphism for all $A\in D(Y)$. In particular, if $f^\ast A=0$, then $f^! A = f^\ast A\otimes f^!(1_Y)=0$, and the isomorphism implies $A=0$.

The pullback functor
\[
(f_n^\ast)_n: D(Y)\to \mathrm{lim}_{[n]\in \Delta} D(X^{n/Y})
\]
admits a left adjoint, such that the counit of the adjunction is precisely the map
\[
\mathrm{colim}_{[n]\in \Delta^{\mathrm{op}}} f_{n!} f_n^! A\to A
\]
This is an equivalence, yielding fully faithfulness. On the other hand, both functors in this adjunction commute with any base change in $Y$, so to check that the unit transformation is also an equivalence, we can check after pullback along $f: X\to Y$. But there, the cover admits a splitting, and the claim is automatic. The similar argument applies to $!$-pullback.
\end{proof}

The $f$-prim analogue is the following result, related to the notion of descendability of Mathew \cite{MathewDescendable}. (Roughly speaking, the reason we have to impose stronger conditions here is that $D(X)$ is presentable, but $D(X)^{\mathrm{op}}$ is not, so we have to ensure that all limits and colimits that are implicit in the previous proof stay finite.)

\begin{proposition}\label{prop:Dprimdescent} Assume that for all $X\in C$, the presentable $\infty$-category $D(X)$ is stable. Let $f: X\to Y$ be an $E$-morphism such that $1_X$ is $f$-prim, with $f_n: X^{n/Y}\to Y$ the fibre product. Assume that the map
\[
1_Y\to \flim_{[n]\in \Delta} f_{n\ast} 1_{X^{n/Y}}
\]
is an isomorphism in $\mathrm{Pro}(D(Y))$; equivalently (by base change and the projection formula), $f_\ast 1_X\in \mathrm{CAlg}(D(Y))$ is descendable. (This condition passes to any base change of $f$.)

Then the pullback functors
\[
(f_n^\ast)_n: D(Y)\to \mathrm{lim}_{[n]\in \Delta} D(X^{n/Y}), (f_n^!)_n: D(Y)\to \mathrm{lim}_{[n]\in \Delta} D(X^{n/Y})
\]
are equivalences. In particular,  $f$ is of universal $!$-descent.
\end{proposition}

Recall that if $1_X$ is $f$-prim, with dual object $A\in D(X)$, then $f_\ast = f_!(A\otimes-)$ and $f_! = f_\ast \mathscr{H}\mathrm{om}(A,-)$. In particular, $f_\ast$ satisfies the projection formula and commutes with any $\ast$-base change. Similarly, $f_!$ commutes with any $!$-base change. Moreover, one has
\[
f_! f^! = f_\ast \mathscr{H}\mathrm{om}(A,f^!) = \mathscr{H}\mathrm{om}(f_! A,-) = \mathscr{H}\mathrm{om}(f_\ast 1_X,-).
\] 

\begin{proof} The functor $(f_n^\ast)_n$ has a right adjoint taking a collection $(A_n)_n\in \mathrm{lim}_{[n]\in \Delta} D(X^{n/Y})$ to $\mathrm{lim}_n f_{n\ast} A_n$. The unit of the adjunction is the natural map
\[
A\to \mathrm{lim}_n f_{n\ast} f_n^\ast A.
\]
We claim that this is an isomorphism, by showing that in fact the map
\[
A\to \flim_n f_{n\ast} f_n^\ast A
\]
in $\mathrm{Pro}(D(Y))$ is an isomorphism. But this latter map is obtained by tensoring the given isomorphism $1_Y\to \flim_{[n]\in \Delta} f_{n\ast} 1_{X^{n/Y}}$ with $A$ (by the projection formula).

In fact, more generally, the same argument applies to $A\in \mathrm{Pro}(D(Y))$, and shows that $f^\ast: \mathrm{Pro}(D(Y))\to \mathrm{Pro}(D(X))$ is conservative. We now claim that the counit map is an isomorphism, for which we have to see that for all $(A_n)_n\in \mathrm{lim}_{[n]\in \Delta} D(X^{n/Y})$, the map
\[
f^\ast \mathrm{lim}_n f_{n\ast} A_n\to A_0
\]
is an isomorphism. Again, we in fact claim that the map
\[
f^\ast \flim_n f_{n\ast} A_n\to A_0
\]
in $\mathrm{Pro}(D(X))$ is an isomorphism. But pullback under one projection map $\mathrm{Pro}(D(X))\to \mathrm{Pro}(D(X\times_Y X))$ is conversative (we proved this above for $f^\ast$, but the hypotheses pass to any base change of $f^\ast$), so we can check this after base changing everything along $f: X\to Y$. But now the cover is split, so descent is automatic.

Now consider the case of $!$-descent; here, $(f_n^!)_n$ has a left adjoint taking $(A_n)_n\in \mathrm{lim}_{[n]\in \Delta} D(X^{n/Y})$ to $\mathrm{colim}_n f_{n!} A_n$. The counit of the adjunction is now given by
\[
\mathrm{colim}_n f_{n!} f_n^! A\to A.
\]
This is an equivalence by taking the internal $\mathrm{Hom}$ from the Pro-isomorphism $1_Y\to \flim_n f_{n\ast} 1_{X^{n/Y}}$ towards $A$. In fact, this shows the slightly more precise claim that the map
\[
\fcolim_n f_{n!} f_n^! A\to A
\]
is an isomorphism in $\mathrm{Ind}(D(Y))$, and in fact this holds for all $A\in \mathrm{Ind}(D(Y))$. In particular, this implies that $f^!: \mathrm{Ind}(D(Y))\to \mathrm{Ind}(D(X))$ is conservative.

To show that the unit of the adjunction is an isomorphism, we need to show that for all $(A_n)_n\in \mathrm{lim}_{[n]\in \Delta} D(X^{n/Y})$, the natural map
\[
A_0\to f^!(\mathrm{colim}_n f_{n!} A_n)
\]
is an isomorphism. Again, we make the more precise claim that the map
\[
A_0\to f^!(\fcolim_n f_{n!} A_n) = \fcolim_n f^! f_{n!} A_n
\]
is an isomorphism in $\mathrm{Ind}(D(X))$. We conclude as in the case of $\ast$-descent: It suffices to prove this after applying $!$-pullback to $X\times_Y X$, where it reduces to the similar descent for the split cover $X\times_Y X\to X$.
\end{proof}
\newpage

\section{Lecture VII: Topological spaces}

Let us come back to the first example, of locally compact Hausdorff spaces. So $C$ is the category of locally compact Hausdorff spaces; equivalently, these are those topological spaces that can be written as open subsets of compact Hausdorff spaces (for example, their one-point compactification). We take $E$ to be the class of all morphisms.

In general, there are several slightly different possible definitions of the derived ($\infty$-)category of abelian sheaves $D(X,\mathbb Z)$ on $X$.

\begin{enumerate}
\item One can literally take the derived $\infty$-category $D(\mathrm{Ab}(X))$ of the abelian category of abelian sheaves on $X$. This is obtained from the category of chain complexes by inverting weak equivalences (in the $\infty$-categorical sense).
\item One can look at the $\infty$-category of sheaves $\mathrm{Shv}(X,D(\mathbb Z))$ (in the sense of Lurie \cite{LurieHTT}) on $X$ with values in $D(\mathbb Z)$. This is the $\infty$-category of contravariant functors $\mathcal F$ from open subsets of $X$ towards $D(\mathbb Z)$, such that $\mathcal F(\emptyset)=\ast$ and the following two conditions are satisfied. First, if $U=U_1\cup U_2$, then $\mathcal F(U)\to \mathcal F(U_1)\times_{\mathcal F(U_1\cap U_2)}\mathcal F(U_2)$ is an isomorphism. Second, if $U$ is a filtered union of subsets $U_i\subset U$, then $\mathcal F(U)\to \mathrm{lim}_i \mathcal F(U_i)$ is an isomorphism (where all limits are taken in the $\infty$-category $D(\mathbb Z)$).
\item The $\infty$-category $\mathrm{HypShv}(X,D(\mathbb Z))$ of hypersheaves on $X$ with values in $D(\mathbb Z)$, which is the localization of $\mathrm{Shv}(X,D(\mathbb Z))$ at the morphisms that induce isomorphisms on all stalks.
\item The versions bounded to the left $D^{\geq n}(\mathrm{Ab}(X))$ and $\mathrm{Shv}(X,D^{\geq n}(\mathbb Z))=\mathrm{HypShv}(X,D^{\geq n}(\mathbb Z))$ (as being bounded to the left implies being hypercomplete), and the left-completions
\[
\hat{D}(\mathrm{Ab}(X))=\mathrm{lim}_n D^{\geq n}(\mathrm{Ab}(X)), \widehat{\mathrm{Shv}}(X,D(\mathbb Z)) = \mathrm{lim}_n \mathrm{Shv}(X,D^{\geq n}(\mathbb Z)).
\]
\end{enumerate}

Let us compare the various possibilities.

\begin{proposition}\label{prop:comparederivedcategories} The composite functor $\mathrm{Ch}(\mathrm{Ab}(X))\to \mathrm{Shv}(X,D(\mathbb Z))\to \mathrm{HypShv}(X,D(\mathbb Z))$ factors over $D(\mathrm{Ab}(X))$ and induces a $t$-exact equivalence
\[
D(\mathrm{Ab}(X))\cong \mathrm{HypShv}(X,D(\mathbb Z)).
\]
In particular, one gets an equivalence of left-completions
\[
\hat{D}(\mathrm{Ab}(X))\cong \widehat{\mathrm{Shv}}(X,D(\mathbb Z)).
\]

If $X$ is paracompact and has finite covering dimension, then the functors
\[
\mathrm{Shv}(X,D(\mathbb Z))\to \mathrm{HypShv}(X,D(\mathbb Z))\to \widehat{\mathrm{Shv}}(X,D(\mathbb Z))
\]
are equivalences.
\end{proposition}

\begin{remark} In general, the three variants
\[
\mathrm{Shv}(X,D(\mathbb Z))\to \mathrm{HypShv}(X,D(\mathbb Z))\to \widehat{\mathrm{Shv}}(X,D(\mathbb Z))
\]
are different. The latter two agree when $X$ has ``locally finite cohomological dimension''. Hypersheaves always form a Bousfield localization of sheaves, but the left-completion in general does not. The functor $\mathrm{HypShv}\to \widehat{\mathrm{Shv}}$ (given by taking $\mathcal F$ to the system $(\tau_{\leq n} \mathcal F)_n$) has a right adjoint, taking a truncation-compatible sequence $\mathcal F_n$ to $\mathcal F=\mathrm{lim}_n \mathcal F_n$. But in general the map $\tau_{\leq n}\mathcal F\to \mathcal F_n$ is not an isomorphism. For example, if $\mathcal F_n=\prod_{i=1}^n \mathbb Z[i]$, then $\mathcal F=\prod_{i=1}^\infty \mathbb Z[i]$ and $\pi_0\mathcal F$ is the sheafification of $U\mapsto \prod_{i=1}^\infty H^i(U,\mathbb Z)$, which is in general nontrivial.
\end{remark}

\begin{proof} A map in $\mathrm{HypShv}(X,D(\mathbb Z))$ is an isomorphism if and only if it induces an isomorphism on all stalks. Indeed, passing to cones, it suffices to see that a hypercomplete sheaf $\mathcal F$ on $X$ is trivial as soon as all of its stalks vanish. But if all stalks vanish, then in particular all sheaves $\pi_i \mathcal F$ vanish, and hence $\mathcal F$ is $\infty$-connective, and thus trivial if hypercomplete.

Thus, to show that $\mathrm{Ch}(\mathrm{Ab}(X))\to \mathrm{HypShv}(X,D(\mathbb Z))$ factors over $D(\mathrm{Ab}(X))$, it suffices to see that quasi-isomorphisms induce isomorphisms on stalks (valued in $D(\mathbb Z)$). But this is clear, as the stalks are unchanged under the sheafification implicit in the functor $\mathrm{Ch}(\mathrm{Ab}(X))\to \mathrm{HypShv}(X,D(\mathbb Z))$: The functor $\mathrm{Ch}(\mathrm{Ab}(X))\to \mathrm{HypShv}(X,D(\mathbb Z))$ is the composite of the functor to presheaves, taking any open $U\subset X$ to the value at $U$ of the complex of sheaves, and hypersheafification.

When restricted to $K$-injective complexes, it turns out that the hypersheafification step is unnecessary. In other words, if $U\subset X$ has a hypercover $U_\bullet\to U$, and $C\in \mathrm{Ch}(\mathrm{Ab}(X))$ is $K$-injective, then the map
\[
C(U)\to \mathrm{lim} C(U^\bullet)
\]
in $D(\mathbb Z)$ is an isomorphism. This follows by noting that the hypercover induces a resolution of $j_! \mathbb Z$ (where $j: U\subset X$ is the open immersion), and taking the associated $R\mathrm{Hom}$ into $C$ (which vanishes by assumption that $C$ is $K$-injective).

This implies that for any $j: U\subset X$ and any $C\in D(\mathrm{Ab}(X))$, the map
\[
\mathrm{Hom}_{D(\mathrm{Ab}(X))}(j_! \mathbb Z,C)\to \mathrm{Hom}_{\mathrm{HypShv}(X,D(\mathbb Z))}(j_!\mathbb Z,\tilde{C})
\]
is an isomorphism (where $\tilde{C}\in \mathrm{HypShv}(X,D(\mathbb Z))$ is the associated hypersheaf). Indeed, we can take a $K$-injective representative for $C$, and use that morphisms in $D(\mathrm{Ab}(X))$ can be computed via $K$-injective resolutions, as well as the previous paragraph.

As the objects $j_! \mathbb Z$ generate $D(\mathrm{Ab}(X))$ as well as $\mathrm{HypShv}(X,D(\mathbb Z))$ under colimits (and shifts), this proves the desired equivalence. Moreover, for all $i$, the $i$-th homology functor on $D(\mathrm{Ab}(X))$ corresponds under this equivalence to $\pi_i$ on hypersheaves -- on $K$-injective representatives, this is even true on the level of presheaves of abelian groups, and thus after sheafification. In particular, the equivalence is $t$-exact.

The equivalence on left-completions is a corollary. For the final statement, see \cite[Theorem 7.2.3.6, Proposition 7.2.1.10, Corollary 7.2.1.12]{LurieHTT}.
\end{proof}

In \cite{EtCohDiamonds}, the author preferred left-completed versions. The reason are the very strong descent properties of $X\mapsto \widehat{\mathrm{Shv}}(X,D(\mathbb Z))$.

\begin{proposition}\label{prop:proetaledescentleftcomplete} Endow $C$ with the Grothendieck topology where a family of maps $f_i: X_i\to X$, $i\in I$, form a cover if for any compact subset $K\subset X$ there is some finite index set $J\subset I$ and compact subsets $K_j\subset X_j$ for $j\in J$ such that $K=\bigcup_j f_j(K_j)$.

Then $X\mapsto \widehat{\mathrm{Shv}}(X,D(\mathbb Z))$ defines a hypersheaf of $\infty$-categories on $C$, which agrees with the sheafification of $X\mapsto \mathrm{Shv}(X,D(\mathbb Z))$.
\end{proposition}

\begin{proof} As $\widehat{\mathrm{Shv}} = \mathrm{lim}_n \mathrm{Shv}(X,D^{\geq n}(\mathbb Z))$, it suffices to prove that $\mathrm{Shv}(X,D^{\geq n}(\mathbb Z))$ defines a hypersheaf of $\infty$-categories for all $n$; and by shifting we can assume $n=0$. Let us write $D^{\geq 0}(X,\mathbb Z)=\mathrm{Shv}(X,D^{\geq 0}(\mathbb Z))$, recalling that in this case all possible definitions agree.

Let $f_\bullet: X_\bullet\to X$ be a hypercover of $X$. One reduces easily to the case that $X$ and all $X_n$ are compact. We get the pullback functor
\[
f_\bullet^\ast: D^{\geq 0}(X,\mathbb Z)\to \mathrm{lim} D^{\geq 0}(X_\bullet,\mathbb Z)
\]
which has a right adjoint $f_{\bullet\ast}$ given by a totalization of the pushforwards along $f_n: X_n\to X$. The formation of these functors commutes with any base change in $X$, by proper base change (Theorem~\ref{thm:properbasechangetopology} below) and noting that the totalization is a finite limit in any bounded range of degrees (using critically that all sheaves lie in $D^{\geq 0}$). Thus, to check that the unit and counit transformations are equivalences, we can base change to points of $X$, where the hypercover splits and the desired descent is automatic.

For the last part, it suffices to show that the functor $\mathrm{Shv}(X,D(\mathbb Z))\to \widehat{\mathrm{Shv}}(X,D(\mathbb Z))$ is an equivalence locally in the Grothendieck topology. But any $X$ admits a cover by a disjoint union of profinite sets (which have covering dimension $0$), where it is an equivalence by Proposition~\ref{prop:comparederivedcategories}. (Note that to get an equivalence of sheafifications, one should check more precisely that this also gives an equivalence on all terms of the corresponding Cech cover; but the terms there are all closed inside such disjoint unions of profinite sets, and thus themselves disjoint unions of profinite sets.)
\end{proof}

Now we state the theorem on the existence of a $6$-functor formalism. Note that taking a topos to its $\infty$-category of $D(\mathbb Z)$-valued sheaves defines a functor to symmetric monoidal $\infty$-categories, so we have
\[
C^{\mathrm{op}}\to \mathrm{CAlg}(\mathrm{Cat}_\infty): X\mapsto \mathrm{Shv}(X,D(\mathbb Z)).
\]
For the classes $I$ and $P$, we take the open immersions and the proper maps, respectively.

\begin{theorem}\label{thm:topological6functors} In this situation, the conditions of Theorem~\ref{thm:construct6functors} are satisfied, and one gets a resulting $6$-functor formalism
\[
X\mapsto D(X):=\mathrm{Shv}(X,D(\mathbb Z)).
\]
\end{theorem}

\begin{remark} We work with $D(\mathbb Z)$-coefficients here, but everything works in a similar way with coefficients in the $\infty$-category of spectra (or other coefficients). See Volpe's thesis \cite{VolpeThesis} for a much more detailed and general discussion. Moreover, Kashiwara--Shapira have done a lot of work on topological $6$-functor formalisms, see for example \cite{KashiwaraShapiraSheavesOnManifolds}.
\end{remark}

\begin{proof} Condition (1) on the morphisms in $I$ and $P$ is easy to check. Note that if $f: X\to Y$ is any map in $C$, then by choosing a compactification $\overline{X}_0$ and and letting $\overline{X}\subset \overline{X}_0\times Y$ be the closure of the graph of $f$, one gets factorization of $f$ into an open immersion $j: X\to \overline{X}$ and a proper map $\overline{f}: \overline{X}\to Y$.

Condition (2) is essentially automatic. Roughly speaking, whenever $D(X)$ is some category of sheaves on $X$ with respect to a topology for which the morphisms $j\in I$ are open immersions, the pullback functors $j^\ast$ automatically acquire left adjoints $j_!$ having good properties.

Condition (3) is the hardest work and stated separately as Theorem~\ref{thm:properbasechangetopology} below. Condition (4) is then actually automatic by Remark~\ref{rem:lastconditionvacuousexcision}.
\end{proof}

\begin{theorem}[Proper Base Change]\label{thm:properbasechangetopology} Let $f: X\to Y$ be a proper map of locally compact Hausdorff spaces.
\begin{enumerate}
\item The functor $f_\ast: D(X)\to D(Y)$ commutes with all colimits.
\item For $A\in D(X)$ and $B\in D(Y)$, the projection formula map
\[
f_\ast A\otimes B\to f_\ast(A\otimes f^\ast B)
\]
is an isomorphism.
\item For any other map $g: Y'\to Y$ of locally compact Hausdorff spaces and pullback $f': X'=X\times_Y Y'\to Y'$, $g': X'\to X$, the base change map
\[
g^\ast f_\ast\to f'_\ast g^{\prime\ast}: D(X)\to D(Y')
\]
is an isomorphism.
\end{enumerate}
\end{theorem}

\begin{proof} Choosing compactifications, we can assume that $Y$ (and hence $X$) is compact Hausdorff. For part (1), it is convenient to give a different description of $D(X)=\mathrm{Shv}(X,D(\mathbb Z))$. Given $\mathcal F\in D(X)$, to any closed subset $K\subset X$ one can associate $\mathcal F(K)$ which are the global sections of the pullback to $K$; equivalently,
\[
\mathcal F(K) = \mathrm{colim}_{U\supset K} \mathcal F(U)
\]
is the colimit over all open neighborhoods $U$ of $K$. Then $D(X)$ is equivalent to the $\infty$-category of contravariant functors $\mathcal F$ taking closed subsets $K$ of $X$ to $D(\mathbb Z)$, with the properties that if $K=K_1\cup K_2$, then
\[
\mathcal F(K)\to \mathcal F(K_1)\times_{\mathcal F(K_1\cap K_2)} \mathcal F(K_2)
\]
is an isomorphism, and if $K=\bigcap_i K_i$ is a cofiltered intersection, then
\[
\mathcal F(K) = \mathrm{colim}_i \mathcal F(K_i).
\]
Indeed, one can recover $\mathcal F(U) = \mathrm{lim}_{K\subset U} \mathcal F(K)$, noting that relatively compact open subsets are cofinal with relatively compact closed subsets. See also \cite[Corollary 7.3.4.10]{LurieHTT}.

In this alternative description, the pushforward is given by $(f_\ast \mathcal F)(K)=\mathcal F(f^{-1}(K))$, noting that this does in fact satisfy all required conditions. Moreover, arbitrary colimits in $\mathcal F$ also respect the conditions, yielding (1).

Now for (2), as all functors commute with colimits, it suffices to check the result on generators, so we can assume $B=j_! \mathbb Z$ for an open immersion $j: Y'\to Y$. The result is clearly true over $Y'$ (as base change to open subsets always holds true), so we need to show that $f_\ast(A\otimes f^\ast j_! \mathbb Z)$ vanishes on the closed complement $Z\subset Y$ of $Y'$. But $j_! \mathbb Z$ can be written as a colimit of $j_{i!} \mathbb Z$ for open immersions $Y'_i\to Y$ such that $Z$ admits an open neighborhood $Z_i$ that does not meet $Y'_i$. This implies that $f_\ast(A\otimes f^\ast j_{i!}\mathbb Z)$ vanishes upon restriction to $Z_i$, and in particular to $Z$; and hence the colimit also vanishes on $Z$.

For (3), we can also assume that $Y'$ is compact. If $g$ is injective (i.e., a closed immersion), then the result is part of the description of pushforward given in (1). In general, to prove the result, it suffices to check that the sections over all compact subsets of $Y'$ agree. Replacing $Y'$ by such a compact subset, we see in particular that it suffices to prove the result after applying $g_\ast$. Now
\[
g_\ast g^\ast f_\ast\cong f_\ast\otimes g_\ast \mathbb Z
\]
by the projection formula (i.e., part (2)) for $g$; while
\[
g_\ast f'_\ast g^{\prime\ast}\cong f_\ast g'_\ast g^{\prime\ast}\cong f_\ast(-\otimes g'_\ast \mathbb Z)
\]
by the projection formula for $g'$. Finally, the projection formula for $f$ reduces us to showing
\[
g'_\ast \mathbb Z\cong f^\ast g_\ast \mathbb Z,
\]
i.e., up to switching the factors, the isomorphism of (3) restricted to the constant sheaf. But now $\mathbb Z\in D^{\geq 0}$, so to check this isomorphism, we can check on stalks. This means we can reduce to the case that $Y'$ is a point, which we have already handled.
\end{proof}

Finally, we discuss a few things regarding cohomological smoothness, and $f$-prim and $f$-suave objects. The following proposition was observed before:

\begin{proposition}\label{prop:manifoldcohomsmooth} Let $f: X\to Y$ be a topological manifold bundle. Then $f$ is cohomologically smooth.
\end{proposition}

\begin{proof} Cohomological smoothness can be checked locally on the source, and is stable under pullback and passage to open subsets. This reduces the question to $\mathbb R\to \ast$, which we handled directly in Proposition~\ref{prop:realssmooth}.
\end{proof}

Conversely, one can ask the following question. Say $X$ is some compact Hausdorff space. When is $f: X\to \ast$ cohomologically smooth? Recall that cohomological smoothness is equivalent to $\mathbb Z$ being $f$-suave, with invertible dualizing sheaf. The first condition is actually satisfied in large generality.

\begin{proposition}\label{prop:ENRULA} Let $X$ be a locally compact Hausdorff space that is a euclidean neighborhood retract, i.e.~a retract of an open subset of $\mathbb R^n$ for some $n$. Then $\mathbb Z$ is $f$-suave for $f: X\to \ast$.
\end{proposition}

\begin{proof} This is a consequence of the stability under retracts, see Proposition~\ref{prop:fsuaveretract}.
\end{proof}

In particular, if $X$ is a euclidean neighborhood retract, then $f: X\to \ast$ is cohomologically smooth if and only if the dualizing complex $f^! \mathbb Z$ is invertible. This can actually be tested pointwise. Note first that if $X$ is a euclidean neighborhood retract, then the homology of $X$ is given by $f_! f^! \mathbb Z$ (where $f: X\to \ast$ denotes the projection). Indeed, this is true for topological manifolds, and then passes to retracts. (Also note that the dual of this is given by cohomology $f_\ast \mathscr{H}\mathrm{om}(f^!\mathbb Z,f^!\mathbb Z)=f_\ast \mathbb Z$ (using Proposition~\ref{prop:suavesheafproperties}~(3)).) It follows that the stalk $(f^! \mathbb Z)_x$ at a point $x\in X$ is given by the relative homology of the pair $(X,X\setminus \{x\})$.

\begin{proposition}\label{prop:homologicalmanifold} Let $X$ be a euclidean neighborhood retract. Assume that there is some $d$ such that for all $x\in X$, the relative homology group $H_i(X,X\setminus \{x\};\mathbb Z)$ vanishes for $i\neq d$, and is given by $\mathbb Z$ for $i=d$. Then $f^!\mathbb Z$ is locally isomorphic to $\mathbb Z[d]$, and hence invertible.
\end{proposition}

\begin{proof} The assumption ensures that $f^!\mathbb Z = \mathcal F[d]$ for some sheaf of abelian groups $\mathcal F$ all of whose stalks are isomorphic to $\mathbb Z$. Moreover, Proposition~\ref{prop:suavesheafproperties}~(3) ensures that (the derived) $\mathscr{H}\mathrm{om}(\mathcal F,\mathcal F)=\mathbb Z$. We will show that these properties force $\mathcal F$ to be locally isomorphic to $\mathbb Z$. Take any $x\in X$ and a local section $s\in \mathcal F(U)$ that is a generator at the stalk at some $x\in U\subset X$. We can assume that $U$ is connected (as $X$ is locally connected, as a euclidean neighborhood retract); our goal is to show that $s$ induces an isomorphism $\mathbb Z\to \mathcal F|_U$. We can replace $X$ by $U$. Moreover, it suffices to see that it induces an isomorphism modulo $p$ for any prime $p$. The image of the map $\mathbb F_p\to \mathcal F/p$ is given by some quotient of the constant sheaf $\mathbb F_p$, and hence is of the form $i_\ast \mathbb F_p$ for some closed immersion $i: Z\to X$. We get an exact sequence
\[
0\to i_\ast \mathbb F_p\to \mathcal F/p\to j_! j^\ast\mathcal F/p\to 0;
\]
note that indeed the quotient is of the form $j_! \mathcal F'$ for some sheaf $\mathcal F'$, which must be the restriction of $\mathcal F/p$ to $U$. But the projection map splits naturally, so $i_\ast \mathbb F_p$ is a summand of $\mathcal F/p$. But the endomorphisms of $\mathcal F/p$ are just $\mathbb F_p$, so in particular it is irreducible, and thus $i_\ast \mathbb F_p\to \mathcal F/p$ is an isomorphism, so that $Z=X$ and $\mathcal F/p=\mathbb F_p$.
\end{proof}

Such $X$ are classically known as (ENR) homology manifolds, and they have been intensely studied. Let me highlight some of their fascinating theory.
\begin{enumerate}
\item If $X$ is any (connected) homology manifold, one can define an invariant $I(X)\in 1+8\mathbb Z$ known as the Quinn index, \cite{QuinnIndex}. (The strange value group $1+8\mathbb Z$ is justified by the multiplicativity property $I(X\times Y)=I(X)I(Y)$.)
\item If $X$ is a topological manifold, then $I(X)=1$.
\item Conversely, if $I(X)=1$ and satisfies the ``disjoint disks property'' (any two maps $D^2\to X$ admit small perturbations that are disjoint), then $X$ is a topological manifold.
\item A theorem of Bryant--Ferry--Mio--Weinberger \cite{BryantFerryMioWeinberger} states that (in dimension $\geq 6$) there are homology manifolds with any given Quinn index.
\end{enumerate}

Unfortunately, the construction of homology manifolds with $I(X)\neq 1$ is extremely indirect, relying on a lot of (high-dimensional) surgery theory. It is an open problem to give explicit constructions of homology manifolds, and whether there are local ``model spaces'' like euclidean space.

A consequence of Proposition~\ref{prop:ENRULA} is that ``constructible'' sheaves are $f$-suave.

\begin{proposition}\label{prop:topologicalconstructibleULA} Let $X$ be a euclidean neighborhood retract, and let $A\in D(X)$ be such that there is a locally finite stratification of $X$ into locally closed subsets $X_i\subset X$ whose closures $\overline{X}_i\subset X$ are neighborhood retracts, and such that $A|_{X_i}$ is a constant sheaf on a perfect complex of abelian groups. Then $A$ is $f$-suave for $f: X\to \ast$.
\end{proposition}

We assume here that a closure of a stratum is a union of strata.

\begin{proof} The condition of being $f$-suave is local on $X$, and stable under triangles. Now $A$ is locally in the stable subcategory generated by the pushforwards of perfect complexes under the closed immersion $p_i: \overline{X}_i\to X$. (This uses that the closure of a stratum is a union of strata.) Now $\overline{X}_i$ is a euclidean neighborhood retract, hence the constant sheaf and therefore any perfect complex is $f_i$-suave for $f_i: \overline{X}_i\to \ast$; and proper pushforwards preserve this property.
\end{proof}

It seems hard to give a complete characterization of the $f$-suave objects. In the $f$-prim case, a complete characterization can be given.

\begin{proposition}\label{prop:topologicalpropersheaves} Let $X$ be a locally compact Hausdorff space and $A\in D(X)$, and denote $f: X\to \ast$. Then $A$ is $f$-prim if and only if $A$ is compact if and only if $A$ is locally isomorphic to the constant sheaf on a perfect complex and supported on a compact subset of $X$.
\end{proposition}

Thus, this class of sheaves is very small, which also explains why that notion has not been much studied before (in contrast to $f$-suave objects). In fact, even the relative setting of a map $f: X\to Y$ can be analyzed completely, and $f$-prim objects are exactly those that are locally isomorphic to the constant sheaf on a perfect complex, and supported on a subset of $X$ that is compact over $Y$. Indeed, by base change compatibility, one can assume that $Y$ is compact. Then Proposition~\ref{prop:primsheafcompact} ensures that any $f$-prim $A$ must be compact, and Proposition~\ref{prop:topologicalpropersheaves} applies.

\begin{proof} By Proposition~\ref{prop:primsheafcompact}, if $A$ is $f$-prim, then $A$ is compact. The hard part is to show that $A$ being compact implies that $A$ is locally isomorphic to the constant sheaf on a perfect complex and supported on a compact subset of $X$. Indeed, such objects are easily seen to be $f$-prim.

Thus, let $A$ be compact. It is enough to show that $A$ is dualizable: Indeed, the dualizable objects are known to be locally isomorphism to the constant sheaf on a perfect complex, and the support must also be compact as $A$ is compact. (One way to show that any dualizable $A$ must be locally constant is to show this first on profinite sets, where sheaves are equivalent to modules over $C(X,\mathbb Z)$, and then descend.)

As the support of $A$ must be compact, we can assume that $X$ is compact. Then $D(X)$ is rigid dualizable as a presentable symmetric monoidal stable $\infty$-category, and in such categories dualizable and compact objects agree. Let us explain one direct way of arguing: One way to characterize ``rigid dualizable'' is that the unit is compact, it is $\omega_1$-compactly generated, and that any $\omega_1$-compact $B$ can be written as a sequential colimit of objects $B_n$ such that for all other objects $C$, the natural map
\[
\mathscr{H}\mathrm{om}(C,\mathbb Z)\otimes B\to \mathrm{colim}_n \mathscr{H}\mathrm{om}(C,B_n)
\]
is an isomorphism. This is in fact not hard to see in our case; the $\omega_1$-compact generators can be taken as $B=j_!\mathbb Z$ for open immersions $j: U\to X$ that are sequential unions of open subsets $j_n: U_n\to X$ along embeddings $\overline{U}_n\subset U_{n+1}$; then $B_n=j_{n!} \mathbb Z$ works. Applied to a compact $B$, the sequential colimit of $B_n$'s must necessarily be pro-constant, and then taking $C=B$ shows that $B$ is dualizable.
\end{proof}

\newpage

\section*{Appendix to Lecture VII: \'Etale Sheaves}

\'Etale sheaves on schemes behave in many ways quite similar to the case of abelian sheaves on locally compact Hausdorff spaces, at least with torsion coefficients (prime to the characteristic). Let us recall these results.

For the category $C$, we take qcqs schemes. Most naturally, one would take for $E$ the separated maps of finite type, but actually one can be a tiny bit more general, following Hamacher \cite{Hamacher}: we take for $E$ the class of separated morphisms $f: X\to Y$ of ``finite expansion'', i.e.~such that on open affines $\mathrm{Spec}(A)\subset X$, mapping to $\mathrm{Spec}(B)\subset Y$, there is a finite set of elements $X_1,\ldots,X_n\in A$ such that the map $B[X_1,\ldots,X_n]\to A$ is integral. (In particular, this includes separated maps of finite type, but it is slightly more general, and in particular it allows perfections of maps of finite type. It also includes things like pro-(finite \'etale) maps.) As $I$, we take the open immersions, while for $P$ we take the proper maps in $E$, i.e.~the ones that satisfy the valuative criterion of properness (equivalently, are universally closed).

The Nagata compactification theorem extends to this setting; this ensures that the condition on morphisms in Theorem~\ref{thm:construct6functors} is satisfied.

\begin{proposition}[{\cite[Proposition 1.8, Theorem 1.17]{Hamacher}}]\label{prop:nagatacompactification} Let $f: X\to Y$ be a morphism in $E$.
\begin{enumerate}
\item The morphism $f$ can be written as the composite of an integral map $X\to X'$ and a separated map of finite presentation $X'\to Y$.
\item If $f\in P$, then $f$ can be written as the composite of an integral map $X\to X'$ and a proper map of finite presentation $X'\to Y$.
\item The morphism $f$ can be written as the composite of an open immersion $X\hookrightarrow \overline{X}$ and a morphism $\overline{X}\to Y$ in $P$.
\end{enumerate}
\end{proposition}

Following the discussion for topological spaces, one can again consider several slightly different versions of the derived category of \'etale sheaves; effectively,
\begin{enumerate}
\item sheaves $\mathrm{Shv}(X_{\mathrm{\acute{e}t}},D(\mathbb Z))$ in the sense of Lurie;
\item hypersheaves $\mathrm{HypShv}(X_{\mathrm{\acute{e}t}},D(\mathbb Z))$; this agrees with the derived category of the abelian category of \'etale sheaves of $\Lambda$-modules;
\item the left-completion $\widehat{\mathrm{Shv}}(X_{\mathrm{\acute{e}t}},D(\mathbb Z))$, which agrees with the left-completion of hypersheaves.
\end{enumerate}

Under mild assumptions, all three notions agree by the following theorem of Clausen--Mathew.

\begin{theorem}[{\cite[Corollary 1.10, Corollary 4.40]{ClausenMathew}}]\label{thm:etalesheaveshypercomplete} Let $X$ be a qcqs scheme of finite Krull dimension that has a uniform bound on the virtual cohomological dimension of its residue fields; e.g., $X$ is of finite type over $\mathbb Z$ or an algebraically closed field. Then
\[
\mathrm{Shv}(X_{\mathrm{\acute{e}t}},D(\mathbb Z)) = \mathrm{HypShv}(X_{\mathrm{\acute{e}t}},D(\mathbb Z)) = \widehat{\mathrm{Shv}}(X_{\mathrm{\acute{e}t}},D(\mathbb Z)).
\]
\end{theorem}

Most sources in the literature either work with $D^+$, i.e.~sheaves bounded to the left (as in the original SGA setting), or with left completions (for example this is what we did in our own work \cite{EtCohDiamonds}). This has the advantage of guaranteeing extremely strong descent theorems, most generally v-descent or even descent for universal submersions, see \cite[Theorem 5.7]{HansenScholze}. On the other hand, considering all sheaves has the advantage that one can reduce to schemes of finite type:

\begin{proposition}\label{prop:sheavesfilteredcolimits} The contravariant functor $X\mapsto \mathrm{Shv}(X_{\mathrm{\acute{e}t}},D(\mathbb Z))$ from qcqs schemes to presentable stable $\infty$-categories takes cofiltered limits $X=\mathrm{lim}_i X_i$ along affine transition maps to filtered colimits. Equivalently, cf.~Lemma~\ref{lem:rightadjointcolimitPrL}~(1), the functor of $\infty$-categories
\[
\mathrm{Shv}(X_{\mathrm{\acute{e}t}},D(\mathbb Z))\to \mathrm{lim}_i \mathrm{Shv}(X_{i,\mathrm{\acute{e}t}},D(\mathbb Z)),
\]
induced by pushforward along the maps $X\to X_i$, is an equivalence.
\end{proposition}

\begin{proof} This follows easily from the presentation of $\mathrm{Shv}(X_{\mathrm{\acute{e}t}},D(\mathbb Z))$: It is generated by the free abelian sheaves on quasicompact separated \'etale maps to $X$, and for any \'etale cover between such, one gets a corresponding relation given by Cech descent. The generators and relations both satisfy noetherian approximation.
\end{proof}

In the case of schemes, some finiteness conditions on the morphisms have to be enforced in any case (as there is no scheme-theoretic compactification of infinite-dimensional affine space), and it turns out that proper base change holds for any choice. But as usual in the \'etale setting, we have to restrict to torsion sheaves, or at least profinitely complete sheaves. Thus, we take coefficients in the full subcategory
\[
\widehat{D}_{\mathrm{pf}}(\mathbb Z)\subset D(\mathbb Z)
\]
of profinitely complete complexes, i.e.~all $A$ such that $A\to \mathrm{lim}_n A/^L n$ is an isomorphism. This is also a Verdier quotient of $D(\mathbb Z)$ by $D(\mathbb Q)$, and acquires a symmetric monoidal structure. It is also equivalent to the full subcategory $D_{\mathrm{tor}}(\mathbb Z)\subset D(\mathbb Z)$ of torsion complexes.

\begin{theorem}\label{thm:etale6functors} Any of the functors
\[
X\mapsto \mathrm{Shv}(X_{\mathrm{\acute{e}t}},\widehat{D}_{\mathrm{pf}}(\mathbb Z)), X\mapsto \mathrm{HypShv}(X_{\mathrm{\acute{e}t}},\widehat{D}_{\mathrm{pf}}(\mathbb Z)), X\mapsto \widehat{\mathrm{Shv}}(X_{\mathrm{\acute{e}t}},\widehat{D}_{\mathrm{pf}}(\mathbb Z))
\]
satisfies the hypotheses of Theorem~\ref{thm:construct6functors} with respect to the classes $I$ and $P$, and hence defines a $6$-functor formalism. Moreover, the natural functors
\[
\mathrm{Shv}(X_{\mathrm{\acute{e}t}},\widehat{D}_{\mathrm{pf}}(\mathbb Z))\to \mathrm{HypShv}(X_{\mathrm{\acute{e}t}},\widehat{D}_{\mathrm{pf}}(\mathbb Z))\to \widehat{\mathrm{Shv}}(X_{\mathrm{\acute{e}t}},\widehat{D}_{\mathrm{pf}}(\mathbb Z))
\]
commute with the respective $\otimes$, $f^\ast$ and $f_!$ functors.
\end{theorem}

\begin{proof} We already discussed the conditions (1) on the classes of morphisms $I$ and $P$. For condition (2) regarding open immersions $j$, one has indeed the left adjoint $j_!$ of extension by zero, which satisfies base change and projection formula formally. Condition (3) is Theorem~\ref{thm:properbasechangeetale} below. Finally, for (4), we note that if $j: U\to X$ is an open immersion with closed complement $i: Z\to X$, then one gets an excision triangle
\[
j_! j^\ast A\to A\to i_\ast i^\ast A
\]
for any $A\in D(X)$ (for any choice of $D(X)$). Indeed, the projection formula and commutation with colimits reduce to the image of $\mathbb Z$ in profinitely complete sheaves (which is also the colimit of $1/n\mathbb Z/\mathbb Z[-1]$), where all objects are bounded to the left, and the exactness can be checked after pullback to geometric points, where it is evident. In particular Remark~\ref{rem:lastconditionvacuousexcision} applies to establish (4).

The functors between the different theories are by definition compatible with $\otimes$ and $f^\ast$, and it is easy to see that they commute with $j_!$ for open immersions, while the proof of Theorem~\ref{thm:properbasechangeetale} shows that proper pushforwards also commute with these functors.
\end{proof}

Again, the key input is the proper base change theorem.

\begin{theorem}[Proper Base Change]\label{thm:properbasechangeetale} Let $f: X\to Y$ be a universally closed separated map of finite expansion. Let $D(X)$ denote any of three options in Theorem~\ref{thm:etale6functors}.
\begin{enumerate}
\item The functor $f_\ast: D(X)\to D(Y)$ commutes with all colimits.
\item For $A\in D(X)$ and $B\in D(Y)$, the projection formula map
\[
f_\ast A\otimes B\to f_\ast(A\otimes f^\ast B)
\]
is an isomorphism.
\item For any other map $g: Y'\to Y$ of schemes and pullback $f': X'=X\times_Y Y'\to Y'$, $g': X'\to X$, the base change map
\[
g^\ast f_\ast\to f'_\ast g^{\prime\ast}: D(X)\to D(Y')
\]
is an isomorphism.
\end{enumerate}
\end{theorem}

\begin{proof} We first handle the case that $D(X)=\mathrm{Shv}(X_{\mathrm{\acute{e}t}},\widehat{D}_{\mathrm{pf}}(\mathbb Z))$. We can write $f$ as the composite of an integral map and a finitely presented proper map, and it suffices to handle both cases separately. By Lemma~\ref{lem:rightadjointcolimitPrL}~(1), one reduces to the case that $f$ is finitely presented proper. Then $f$ arises via base change from a similar map where $Y$ is of finite type over $\mathbb Z$, and Lemma~\ref{lem:rightadjointcolimitPrL}~(2) reduces us to the case that $Y$ is of finite type. Then all possible notions of $D(X)$ agree by Theorem~\ref{thm:etalesheaveshypercomplete}. Moreover, $f_\ast$ has finite cohomological dimension (bounded by twice the dimension of the geometric fibres of $f$), and one can reduce to $D^+(X)$ by Postnikov limits, or then to sheaves of abelian groups. Part (1) is then a general consequence of the \'etale sites of $X$ and $Y$ being coherent. Part (2) can be checked on geometric fibres, where it reduces to part (3) and part (1) (noting that $D(Y)$ is generated by the unit when $Y$ is a geometric point). Part (3) is then the usual proper base change theorem in \'etale cohomology.

In the case of hypercomplete sheaves, it suffices to prove all isomorphisms after pullbacks to geometric points. After base change to geometric points, sheaves and hypersheaves agree, and the assertions reduce to the case of sheaves. For example, for part (1), take any collection $A_i\in \mathrm{Shv}(X,D(\mathbb Z))$, $i\in I$, of sheaves, and assume that all $A_i$ are already hypersheaves. Let $\bigoplus_i A_i$ denote the direct sum as sheaves, and $\hat{\bigoplus}_i A_i$ its hypercompletion. We want to know whether the map
\[
\hat{\bigoplus}_i f_\ast A_i\to f_\ast(\hat{\bigoplus}_i A_i)
\]
is an isomorphism, where we note that $f_\ast$ preserves hypersheaves. To check this, take the base change to any geometric point $g: \overline{y}\to Y$. Note that base change taken in hypersheaves commutes with direct sums taken in hypersheaves. Also, pullback commutes with pushforwards as taken in sheaves (by the case of sheaves already established), but sheaves and hypersheaves agree after pullback. In the end, this reduces the displayed isomorphism to the case where $Y$ is a geometric point, where it was already established.

In the case of left-completed sheaves, one uses that $f_\ast$ has bounded cohomological dimension to show that it commutes with Postnikov limits, and then any question can be reduced to the case of bounded sheaves by passing to suitable truncations.
\end{proof}

\begin{lemma}\label{lem:rightadjointcolimitPrL} Let $I$ be an $\infty$-category.
\begin{enumerate}
\item Let $i\mapsto D_i$ be a covariant functor to presentable $\infty$-categories, so that all transition functors $f_{ij}: D_i\to D_j$ have right adjoints $g_{ij}: D_j\to D_i$. Then the colimit $D = \mathrm{colim}_i D_i$ in presentable $\infty$-categories is given by the limit $\mathrm{lim}_i D_i$ in $\infty$-categories, where the limit is taken along those right adjoint functors.

If $I$ is filtered and all $g_{ij}$ commute with filtered colimits, then the induced endofunctors $D_i\to D\to D_i$ are given by the colimit over $j\geq i$ of the composite functors $D_i\xrightarrow{f_{ij}} D_j\xrightarrow{g_{ij}} D_i$.
\item Let $i\mapsto (f_i: C_i\to D_i)$ be a covariant functor to the $\infty$-category of maps of presentable $\infty$-categories. Assume that $f_i$ admits a colimit-preserving right adjoint $g_i: D_i\to C_i$ that commutes with the transition functor $C_i\to C_j$, $D_i\to D_j$. Then the functor $f: C=\mathrm{colim}_i C_i\to \mathrm{colim}_i D_i$ given by the colimit of the $f_i$ admits the right adjoint $g: C=\mathrm{colim}_i C_i\to \mathrm{colim}_i D_i$ given as the colimit of the $g_i$.
\end{enumerate}
\end{lemma}

\begin{proof} The first assertion of (1) is the general description of colimits of presentable $\infty$-categories, see \cite[Theorem 5.5.3.18, Corollary 5.5.3.4]{LurieHTT}. The other assertion is also standard, but we do not know the correct reference, so we sketch the argument. Namely, one shows that the left adjoint to $D=\mathrm{lim}_j D_j\to D_i$ is the functor $D_i\to D=\mathrm{lim}_j D_j$ taking $X_i$ to $\mathrm{colim}_{j'\geq i,j} g_{j'j}(f_{ij}(X_i))$, noting that this indeed lands in $\mathrm{lim}_j D_j$, and that one can easily write down the unit and counit of the adjunction.

Part (2) again follows by noting that one can write down the unit and counit of the adjunction between $f$ and $g$, as the colimit of the ones for $f_i$ and $g_i$.
\end{proof}

Let us also prove Poincar\'e duality in this setting (modulo identifying the dualizing complex).

\begin{proposition}\label{prop:etalecohometaleDet} In any of the $6$-functor formalisms from Theorem~\ref{thm:etale6functors}, all \'etale morphisms of schemes are cohomologically \'etale.
\end{proposition}

\begin{proof} By construction, open immersions are cohomologically \'etale. If $f: X\to Y$ is \'etale, then its diagonal is an open immersion, so $\Delta_f$ is cohomologically \'etale. As all categories satisfy \'etale descent, checking whether $f$ is cohomologically \'etale is an \'etale local question. We can thus base change $f$ along itself, yielding $X\times_Y X\to X$. This decomposes into two components, one of which is an isomorphism (the diagonal $X\hookrightarrow X\times_Y X\to X$). By induction on the maximal cardinality of a geometric fibre of $f$, we can thus prove the desired result.
\end{proof}

We note that this implies that $f_!$ is canonically a left adjoint of $f^\ast$ when $f$ is \'etale, even if we did not make this part of the definition.

\begin{theorem}\label{thm:smoothcohomsmoothDet} Let $f: X\to Y$ be a smooth morphism of schemes. Assume that a prime $\ell$ is invertible on $Y$. Then $f$ is cohomologically smooth in the $\ell$-completions of the $6$-functor formalisms from Theorem~\ref{thm:etale6functors}.
\end{theorem}

In Zavyalov's paper \cite{Zavyalov}, it is explained how one can moreover canonically identify the dualizing complex as a Tate twist (using a deformation to the normal cone).

\begin{proof} The question is \'etale local on the source, so using Proposition~\ref{prop:etalecohometaleDet} one reduces to the case of affine space, and then to the affine line $f: \mathbb A^1_{\mathbb Z[\tfrac 1\ell]}\to \mathrm{Spec} \mathbb Z[\tfrac 1\ell]$. We take the sheaf $L=\mathbb Z_\ell(1)[2]$ on $\mathbb A^1$. One constructs a map
\[
\alpha: \Delta_\ast \mathbb Z_\ell\to p_2^\ast L
\]
using the first Chern class of the line bundle $\mathcal O(\Delta)$ on $\mathbb A^2$. (Indeed, $\mathcal O(\Delta)$ defines a $\mathbb G_m$-torsor trivialized outside $\Delta$, i.e.~a map $\Delta_\ast \mathbb Z\to \mathbb G_m[1]$. Passing to derived $\ell$-completions gives the desired map $\alpha$.) It is a standard computation that $f_! L\cong \mathbb Z_\ell$, where the isomorphism comes from the first Chern class of the Cartier divisor of a section. Combined, this suffices to construct the desired datum for Theorem~\ref{thm:critcohomsmooth}. See Zavyalov's paper \cite{Zavyalov} for a more detailed account.
\end{proof}

In \cite{HansenScholze} it is shown how the classical arguments of Deligne on Verdier duality and nearby cycles can be rephrased in terms of $f$-suave objects. In particular, $f$-suave objects are exactly the ``universally locally acyclic'' (ULA) objects in the sense of Deligne. Over a geometric point, this coincides with the constructible sheaves, i.e.~the compact objects. In particular, Verdier duality is a perfect duality on constructible sheaves on finite type schemes $X$ over an algebraically closed field $k$.

\newpage

\section{Lecture VIII: Coherent sheaves}

Let us now consider a rather different kind of sheaves, namely coherent sheaves. As our category of geometric objects, we would like to take schemes $X$, and as $D(X)$ the quasicoherent derived category $D_{\mathrm{qc}}(X)$. There are two possible definitions: Either this is the full subcategory of the derived category of sheaves of $\mathcal O_X$-modules such that all cohomology sheaves are quasicoherent; or, $\infty$-categorically, it is defined via descent from affine schemes, i.e.
\[
D_{\mathrm{qc}}(X) = \mathrm{lim}_{R, \mathrm{Spec}(R)\to X} D(R).
\]
There is of course a version of Poincar\'e duality in coherent cohomology, which is Grothendieck--Serre duality.

\begin{theorem}\label{thm:grothendieckserreduality} Let $f: X\to Y$ be a proper smooth map of schemes of relative dimension $d$. Then
\[
f_\ast: D_{\mathrm{qc}}(X)\to D_{\mathrm{qc}}(Y)
\]
has a right adjoint given by $\Omega^d_{X/Y}[d]\otimes f^\ast$.
\end{theorem}

This indicates that proper smooth maps should be cohomologically smooth in the formalism we seek. Moreover, the dualizing object has a description that is local on $X$; this would most naturally be explained if in fact all smooth morphisms are cohomologically smooth. We will see, however, that in the coherent setting things behave somewhat differently. In the next lecture, we will explain a way to modify $D_{\mathrm{qc}}$ so as to recover a $6$-functor formalism that feels much closer to the topological ones.

Before getting started, we must note that base change theorems for coherent sheaves tend to have some flatness or Tor-independence assumptions in them. The underlying reason is that when forming fibre products of schemes, on affine pieces one takes the corresponding tensor product of rings, but really this should be a derived tensor product. For this reason, any discussion of a coherent $6$-functor formalism has to work with derived schemes.

\subsection{Reminders on derived schemes}

As our model, we take here the ones modeled on animated commutative rings, but one could also work with (connective) $E_\infty$-rings. This is the model also studied by To\"en--Vezzosi \cite{ToenVezzosi}. Recall that the $\infty$-category of animated commutative rings is freely generated under sifted colimits by polynomial algebras $\mathbb Z[X_1,\ldots,X_n]$; equivalently, it is obtained from the category of simplicial commutative rings by inverting weak equivalences. (This procedure of animation is a classical operation, first introduced by Quillen as a non-abelian derived category (but see also the work of Illusie), and $\infty$-categorically the theory has been written up by Lurie \cite[Section 5.5.8]{LurieHTT}. The name animation is due to Clausen, and a general discussion of this procedure in this language is in \cite[Section 5.1.4]{CesnaviciusScholze}.)

\begin{definition}\label{def:derivedscheme} A derived scheme is a pair $(X,\mathcal O_X)$ consisting of a topological space $X$ and a sheaf of animated commutative rings $\mathcal O_X$ such that $(X,\pi_0 \mathcal O_X)$ is a scheme, and each $\pi_i \mathcal O_X$ is a quasicoherent $\pi_0 \mathcal O_X$-module. A morphism of derived schemes $(X,\mathcal O_X)\to (Y,\mathcal O_Y)$ is a map of topological spaces $f: X\to Y$ along with a map $f^\sharp: f^{-1} \mathcal O_Y\to \mathcal O_X$ of sheaves of animated commutative rings that induces a map of classical schemes upon passing to $\pi_0$ (i.e., induces local maps on local rings).

A derived scheme $(X,\mathcal O_X)$ is affine if the classical scheme $(X,\pi_0 \mathcal O_X)$ is affine.
\end{definition}

(To define this rigorously as an $\infty$-category, one first defines an $\infty$-category of topological spaces equipped with a sheaf of animated commutative rings, and then restricts the objects and the $1$-morphisms as above.) Then the $\infty$-category of affine derived schemes is equivalent to the $\infty$-category of animated commutative rings, and general derived schemes are glued from affine derived schemes along open covers. Also, there is an alternative definition of derived schemes in terms of their functor of points.

One can define flat maps of derived schemes (e.g., as those where any base change to a classical scheme becomes classical), and then \'etale resp.~smooth maps as those maps that are flat and induce \'etale resp.~smooth maps on the classical truncations.

Thus, we take for $C$ the $\infty$-category of derived schemes; for simplicity, we restrict to qcqs objects, i.e.~$|X|$ is a quasicompact and quasiseparated topological space. For any $X\in C$, one can define the quasicoherent $\infty$-category $D_{\mathrm{qc}}(X)$, for example via descent (along open covers) from the affine case. This defines a contravariant functor from $C$ to symmetric monoidal presentable stable $\infty$-categories. We will have occasion to consider some subcategories of $D_{\mathrm{qc}}(X)$.

\begin{theorem}[{Thomason--Trobaugh, \cite{ThomasonTrobaugh}}] The $\infty$-category $D_{\mathrm{qc}}(X)$ is compactly generated, and the compact objects agree with the dualizable objects, which are also the perfect complexes $\mathrm{Perf}(X)$, i.e.~on open affine subsets in the subcategory generated by $A$ under direct sums, cones, and retracts.
\end{theorem}

\begin{proof} Let us recall the argument briefly. First, if $X=\mathrm{Spec}(A)$ is affine, then $A$ is a compact generator of $D_{\mathrm{qc}}(X)=D(A)$, and thus all compact objects are perfect complexes, and thus dualizable. Conversely, as the unit is compact, all dualizable objects are compact. Now in general one can classify the dualizable objects as the perfect complexes, as dualizable objects satisfy Zariski descent. Moreover, the unit on $D_{\mathrm{qc}}(X)$ is compact, and thus all dualizable objects are compact. It remains to show that the dualizable (i.e., perfect) objects generate $D_{\mathrm{qc}}(X)$. This follows from the following general categorical lemma (applied first to separated $X$ with basis of open affines, and then to general $X$ with basis of separated open subsets), together with the observation that for $X=\mathrm{Spec}(A)$ affine and a constructible closed subset $Z$ of $X$, given as the vanishing locus of some $f_1,\ldots,f_n\in A$, the subcategory $D_{\mathrm{qc}}(X\ \mathrm{on}\ Z)\subset D_{\mathrm{qc}}(X)$ of those complexes that vanish outside $Z$ is compactly generated by the Koszul complex
\[
A/^L(f_1,\ldots,f_n) = A\otimes^L_{\mathbb Z[X_1,\ldots,X_n]} \mathbb Z\in \mathrm{Perf}(X\ \mathrm{on}\ Z).\qedhere
\]
\end{proof}

\begin{lemma}\label{lem:compactgenerationglobalizes} Let $X$ be a spectral space equipped with basis $B$ of quasicompact open subsets stable under finite intersections. Consider a sheaf of presentable stable $\infty$-categories $U\mapsto \mathcal C_U$ on $X$. Assume that for all $U\in B$, the $\infty$-category $\mathcal C_U$ is compactly generated, that for all $U'\subset U$ with $U'\in B$, the functor $\mathcal C_U\to \mathcal C_{U'}$ is a Bousfield localization that preserves compact objects, with compactly generated kernel.

Then for all quasicompact open $U\subset X$, the $\infty$-category $\mathcal C_U$ is compactly generated, and for all quasicompact open $U'\subset U$, the functor $\mathcal C_U\to \mathcal C_{U'}$ is a Bousfield localization that preserves compact objects, and its kernel is compactly generated.
\end{lemma}

\begin{proof} Note first that $U\mapsto \mathcal C_U^\omega\subset \mathcal C_U$ is also a sheaf of $\infty$-categories on $B$ (as an object that is locally compact is in fact compact, by computing Hom's via descent, which is finitary on a spectral space). It thus extends to a sheaf $U\mapsto \mathcal C^\omega(U)$ for all quasicompact open $U\subset X$, which comes with a functor
\[
\mathrm{Ind}(\mathcal C^\omega(U))\to \mathcal C_U
\]
which is in fact fully faithful.

One proves that this is an equivalence by induction on $B$, the starting point being the case that $U$ is a quasicompact open subset of $X$ that can be written as a union $U=U_1\cup U_2$ with $U_1,U_2\in B$. To show that the functor is an equivalence, it suffices to show that for any $A\in \mathcal C_{U_1}^\omega$ the object $A\oplus A[1]$ lifts to an object of $(\mathcal C^\omega)(U)$. This is equivalent to the similar lifting from $U_1\cap U_2$ to $U_2$. But by assumption $\mathcal C_{U_2}\to \mathcal C_{U_1\cap U_2}$ is a left Bousfield localization of compactly generated presentable stable $\infty$-categories whose kernel is also compactly generated; thus, the compact objects of $\mathcal C_{U_1\cap U_2}$ are the idempotent completion of the Verdier quotient
\[
\mathcal C_{U_2}^\omega/\mathrm{ker}(\mathcal C_{U_2}^\omega\to \mathcal C_{U_1\cap U_2}^\omega).
\]
(Indeed, taking the Verdier quotient on the level of compact objects first and passing to $\mathrm{Ind}$-categories produces the Verdier quotient on the level of the big categories. But the compact objects of the $\mathrm{Ind}$-category give the idempotent completion.) By Proposition~\ref{prop:thomasonstrick} below we see that the desired lifting is possible for any object with trivial $K_0$-class, such as $A\oplus A[1]$. The rest of the proof is some routine verification.
\end{proof}

\begin{proposition}[Thomason--Trobaugh]\label{prop:thomasonstrick} Let $\mathcal C$ be a small stable $\infty$-category and let $\mathcal C'$ be its idempotent completion. Then $K_0(\mathcal C)$ injects into $K_0(\mathcal C')$, and an object $X\in \mathcal C'$ lies in $\mathcal C\subset \mathcal C'$ if and only if its $K_0$-class $[X]\in K_0(\mathcal C')$ lies in $K_0(\mathcal C)\subset K_0(\mathcal C')$.
\end{proposition}

Here, $K_0(\mathcal C)$ is generated by elements $[X]$ for all $X\in \mathcal C$, subject to $[X]=[X']+[X'']$ for any distinguished triangle $X'\to X\to X''$ in $\mathcal C$.

\begin{proof} Let us first show what we really need, namely that for any $A'\in \mathcal C'$, the object $A'\oplus A'[1]$ lies in $\mathcal C$. By definition, $A'$ is determined by an object $A\in \mathcal C$ with an idempotent endomorphism $e: A\to A$. But then the cone of $A\xrightarrow{1-e} A$ is $A'\oplus A'[1]$, as desired.

This implies that the relations defining $K_0(\mathcal C')$ are generated by the relations $[X]+[X']=[X\oplus X']$, and the relations coming from $K_0(\mathcal C)$. Indeed, any distinguished triangle $X'\to X\to X''$ can be modified into one of the form $X'\oplus X'[1]\to X\oplus X'[1]\oplus X''[1]\to X''\oplus X''[1]$ making the first and last term in $\mathcal C$, and thus also the middle term. This easily implies that $K_0(\mathcal C)\to K_0(\mathcal C')$ is injective. Moreover, it shows that for $X\in\mathcal C'$, one has $[X]\in K_0(\mathcal C)$ if and only if there is some $X'\in \mathcal C$ such $X\oplus X'\in \mathcal C$. But then $X\in\mathcal C$ as the cone of $X'\to X\oplus X'$ is.
\end{proof}

\begin{definition}\label{def:pseudocoherent} Let $A$ be an animated commutative ring. An object $K\in D(A)$ is pseudocoherent if for any $n$ there is a perfect complex $K_n$ and a map $K_n\to K$ whose cone sits in homological degrees $\geq n$. An object $K\in D(A)$ is coherent if it is pseudocoherent and bounded.
\end{definition}

These notions satisfy Zariski descent and thus globalize to conditions on $K\in D_{\mathrm{qc}}(X)$. On general derived schemes, one always has many perfect and pseudocoherent complexes, but not always coherent ones. One would like to obtain them by truncating pseudocoherent complexes, but this may not preserve pseudocoherent complexes. It is true, however, under mild conditions on the derived scheme $(X,\mathcal O_X)$.

\begin{proposition}\label{prop:truncatepseudocoherent} Assume that $(X,\pi_0\mathcal O_X)$ is a coherent scheme, i.e.~locally the spectrum of a coherent ring, and $\pi_i \mathcal O_X$ is a coherent $\pi_0 \mathcal O_X$-module for all $i$. Then $A\in D_{\mathrm{qc}}(X)$ is pseudocoherent if and only if all $\pi_i A$ are coherent $\pi_0 \mathcal O_X$-modules. In particular, any truncation of $A$ is again pseudocoherent.
\end{proposition}

Finally, we have the following finiteness condition on maps $A\to B$.

\begin{definition}\label{def:pseudocoherentrings} A map $f: A\to B$ of animated commutative rings is almost of finite presentation if there is some factorization $A\to A[X_1,\ldots,X_n]\to B$ such that $B$ is pseudocoherent as $A[X_1,\ldots,X_n]$-module.
\end{definition}

One can check that this condition is independent of the chosen map $A[X_1,\ldots,X_n]\to B$ as long as it induces a surjection on $\pi_0$. Moreover, one can show that it globalizes to a notion for maps $f: X\to Y$ of derived schemes. We will use the following terminology.

\begin{definition} A map $f: X\to Y$ of derived schemes is proper if it is almost of finite presentation and the map of underlying classical schemes is proper.
\end{definition}

This definition is in particular relevant for the finiteness results in coherent cohomology:

\begin{theorem}\label{thm:finitenesscoherentcohom} Let $f: X\to Y$ be a proper map of derived schemes. Then $f_\ast$ takes pseudocoherent objects of $D_{\mathrm{qc}}(X)$ to pseudocoherent objects of $D_{\mathrm{qc}}(Y)$, and coherent objects to coherent objects. If $f$ is of finite Tor-dimension, then $f_\ast$ takes $\mathrm{Perf}(X)$ to $\mathrm{Perf}(Y)$.
\end{theorem}

We will give a direct proof in the next lecture, using the formalism of solid modules. It can also be deduced from the usual finiteness results in coherent cohomology.

\subsection{Right and left adjoints to pullback}

We would like to extend $X\mapsto D_{\mathrm{qc}}(X)$ to a $6$-functor formalism. To do so, we need to define classes of morphisms $I$ and $P$, and these classes are restrained by the existence of suitable left and right adjoints to pullback. First, right adjoints actually always satisfy base change:

\begin{proposition}\label{prop:quasicoherentbasechange} Let $f: X\to Y$ be any map in $C$, i.e.~a map of qcqs derived schemes. Then $f^\ast$ admits a colimit-preserving right adjoint $f_\ast$ satisfying projection formula and base change.
\end{proposition}

\begin{proof} Once one has proved the result for affine $Y$, it follows in general (as the then locally defined right adjoints commute with base change and thus globalize to the desired adjoint). Thus, assume that $Y$ is affine. One can then moreover reduce to the case that $X$ is affine, by first deducing the case of separated $X$, and then general $X$, using finite covers by open affine subsets. If $X=\mathrm{Spec}(A)$ and $Y=\mathrm{Spec}(B)$ are affine, then $f^\ast$ corresponds to the tensor product functor $-\otimes_B A: D(B)\to D(A)$, whose right adjoint $D(A)\to D(B)$ is the forgetful functor, and commutes with all colimits, and satisfies the projection formula. Base change is also immediate, noting that everything is suitably derived.
\end{proof}

Thus, in principle we could take for $P$ the class of all maps. Before turning to the discussion of such possibilities, let us analyze a case when $f^\ast$ admits a left adjoint.

\begin{proposition}\label{prop:leftadjointtoproperpullback} Let $f: X\to Y$ be a proper map of finite Tor-dimension. Then $f^\ast$ admits a left adjoint $f_\sharp$ satisfying projection formula and base change. Concretely, it is the functor taking a filtered colimit $\mathrm{colim}_i P_i$ of perfect complexes $P_i\in D_{\mathrm{qc}}(X)$ to $\mathrm{colim}_i (f_\ast P_i^\vee)^\vee$ where $P_i^\vee=\mathscr{H}\mathrm{om}(P_i,\mathcal O_X)$ denotes the naive dual (and the second $-^\vee$ denotes the naive dual on $Y$).
\end{proposition}

\begin{proof} We have the adjunction between $f^\ast$ and $f_\ast$ on $D_{\mathrm{qc}}=\mathrm{Ind}(\mathrm{Perf})$. But both functors preserve $\mathrm{Perf}$, so we get an adjunction there. On the other hand, $\mathrm{Perf}$ is selfdual (with the naive duality) and so one gets an adjunction the other way between the naive duals of $f^\ast$ and $f_\ast$ on $\mathrm{Perf}$. But $f^\ast$ commutes with the duality, so one gets a left adjoint of $f^\ast: \mathrm{Perf}(Y)\to \mathrm{Perf}(X)$. Passing to $\mathrm{Ind}$-categories again, we get the desired left adjoint, with the given formula. This shows that it commutes with base change and satisfies the projection formula.
\end{proof}

Something weird happens here: In the other formalisms, it were the open immersions where $f^\ast$ admits a left adjoint; here is it is the proper maps (of finite Tor-dimension). And indeed, for open immersions like $\mathrm{Spec}(A[\tfrac 1g])\hookrightarrow \mathrm{Spec}(A)$, the pullback functor
\[
-\otimes_A A[\tfrac 1g]: D(A)\to D(A[\tfrac 1g])
\]
does not admit a left adjoint, as it does not commute with products. We will see in the next lecture a way to change the setting so that it does, by working with some kind of topological modules, and remembering the product topology.

Today, however, we will stick with abstract modules, and see what we can do.

\subsection{Option 1: All maps are proper.}

The first option is to simply allow all maps to lie in $P$ (and hence in $E$) and for $I$ just take the isomorphisms. By Proposition~\ref{prop:quasicoherentbasechange} and Theorem~\ref{thm:construct6functors}, this defines a $6$-functor formalism on $C$ with values in presentable stable $\infty$-categories.

\begin{theorem}\label{thm:option1ULA} In this $6$-functor formalism, for any proper map $f: X\to Y$ of finite Tor-dimension the sheaf $\mathcal O_X$ is $f$-suave.
\end{theorem}

In particular, we get a dualizing object $\omega_{X/Y}\in D_{\mathrm{qc}}(X)$ as the Verdier dual of $\mathcal O_X$, i.e.~$f^!\mathcal O_Y$ with respect to the $!$-functor of this formalism. (We are reluctant to denote this by $f^!$ in general, as $f^!$ has a standard meaning in coherent cohomology, but for proper maps it is the correct functor.)

\begin{proof} Using Proposition~\ref{prop:checksuavesheaf}, it suffices to see that the formation of $f^! \mathcal O_Y$ commutes with any base change. But $f^!$ commutes with all colimits by Theorem~\ref{thm:finitenesscoherentcohom} and is thus $D_{\mathrm{qc}}(Y)$-linear, and the resulting $D_{\mathrm{qc}}(Y)$-linear adjunction between $f_\ast$ and $f^!$ base changes to a similar adjunction after any base change $Y'\to Y$, noting that $D_{\mathrm{qc}}(X\times_Y Y')=D_{\mathrm{qc}}(X)\otimes_{D_{\mathrm{qc}}(Y)} D_{\mathrm{qc}}(Y')$.
\end{proof}

If $f$ is smooth, we can say more:

\begin{proposition}\label{prop:option1smooth} Let $f: X\to Y$ be a proper smooth map of derived schemes. Then $f$ and its diagonal $\Delta$ are cohomologically smooth. The dualizing sheaf $\omega_{X/Y} = f^! \mathcal O_Y$ is inverse to the sheaf $\Delta^! \mathcal O_{X\times_Y X}$ and thus local on $X$.
\end{proposition}

This argument goes back to Verdier \cite{VerdierGrothendieckSerre}.

\begin{proof} Let $p_1,p_2: X\times_Y X\to X$ be the two projections. We compute, using the base change compatibility of $f^! \mathcal O_X$ and Theorem~\ref{thm:option1ULA} applied to $\Delta$ (which is still of finite Tor-dimension):
\[
\mathcal O_X = \Delta^! p_1^! \mathcal O_X\cong \Delta^! p_2^\ast f^! \mathcal O_Y\cong \Delta^\ast p_2^\ast f^! \mathcal O_Y\otimes \Delta^! \mathcal O_{X\times_Y X}\cong f^! \mathcal O_Y\otimes \Delta^! \mathcal O_{X\times_Y X}.\qedhere
\]
\end{proof}

If $X$ and $Y$ are classical, it is not hard to define a canonical isomorphism $\Delta^! \mathcal O_{X\times_Y X}\cong (\Omega_{X/Y}^d[d])^\vee$ by working with local coordinates, and checking that the isomorphism does not depend on the choice of coordinates. If $Y$ is derived, this does not work, but instead one can use a deformation to the normal cone as in \cite[Lecture 13]{ClausenScholzeComplex} to define this isomorphism.

\subsection{Option 2: $I$ consists of the proper maps of finite Tor-dimension, $P$ consists of the open immersions.}

If we want to allow some morphisms in $I$, then it has to be the proper maps of finite Tor-dimension. We restrict $P$ to consist of the open immersions. The class $E$ should then consist of the separated maps that are almost of finite presentation and of finite Tor-dimension. (We note that the discussion of Lecture 4 also applies if instead of asking that all maps in $E$ factor as $P\circ I$, they factor as $I\circ P$. In fact, one can formally deduce it by replacing all $D(X)$ by $D(X)^{\mathrm{op}}$.)

For this discussion, one needs a derived version of Nagata compactifications, see also \cite[Chapter 5, Proposition 2.1.6]{GaitsgoryRozenblyum}.

\begin{theorem}\label{thm:derivednagata} Let $f: X\to Y$ be a separated map of qcqs derived schemes that is almost of finite presentation. Then $f$ can be factored as the composite of an open immersion $j: X\hookrightarrow \overline{X}$ and a proper map $\overline{f}: \overline{X}\to Y$.

More precisely, any such Nagata compactification on the level of classical schemes can be lifted to the derived level, up to universal homeomorphism.
\end{theorem}

\begin{proof} Let $f_0: X_0\to Y_0$ be the map of underlying classical schemes, and fix a compactification $X_0\hookrightarrow \overline{X}_0$. We will lift the given compactification to the derived level, up to universal homeomorphism. This problem is local on $\overline{X}_0$. (More precisely, we make $X_0$ inductively larger, and the resulting extension is local on $\overline{X}_0$.) We can thus assume that $Y$ and $\overline{X}_0$ are affine. Replacing $Y$ by $\mathbb A^n_Y$, we can moreover arrange that the map $\overline{X}_0\to Y_0$ is finite. Now we apply the following lemma.
\end{proof}

\begin{lemma}\label{lem:zariskimain} Let $A$ be an animated commutative ring, and let $U$ be a derived scheme over $\mathrm{Spec}(A)$ of almost finite presentation together with an open embedding of the corresponding classical scheme $U_0$ into $\mathrm{Spec}(B_0)$ for some classical $\pi_0 A$-algebra $B_0$ that is finitely presented as $\pi_0 A$-module. Then there is a pseudocoherent $A$-algebra $B$ with a map $\pi_0 B\to B_0$ that is a universal homeomorphism, and an open embedding of $U$ into $\mathrm{Spec}(B)$ lifting the given open embedding of classical schemes.
\end{lemma}

\begin{proof} One can argue by inductively modifying $A$ by first adjoining further elements, so as to generate $\pi_0 \mathcal O_U$ and make $\pi_0 A\to B_0$ surjective up to universal homeomorphism, and then quotienting by elements in degree $0$ (so that $U_0$ becomes an open subscheme of $\mathrm{Spec}(\pi_0 A)$), and then by elements in higher homotopy groups, giving the desired $B$ in the limit.
\end{proof}

However, in the current setting, we actually need to preserve the property of finite Tor-dimension. We do not know whether this property can be preserved under Nagata compactifications, so for the moment we do not know how to construct a $6$-functor formalism of the desired type. This would actually have the further problem that $P$ does not satisfy the 2-out-of-3 property, and neither does $E$. This is not actually really required for the machinery of Liu--Zheng \cite{LiuZhengGluing}, but it is required for some of our machinery like the notions of cohomologically \'etale (or proper) maps.

Still, let us check that the maps $I$ and $P$ would satisfy axiom (4) from Lecture IV:

\begin{proposition}\label{prop:option2axiom4} Consider a cartesian diagram of derived schemes
\[\xymatrix{
X'\ar[r]^{f'}\ar[d]_{g'} &X\ar[d]^g\\
Y'\ar[r]^f &Y
}\]
where $f$ is an open immersion, and $g$ is a proper map of finite Tor-dimension. Then the natural map
\[
g_\sharp f'_\ast\to f_\ast g'_\sharp: D_{\mathrm{qc}}(X')\to D_{\mathrm{qc}}(Y)
\]
is an isomorphism.
\end{proposition}

\begin{proof} The claim is local on $Y$, so we can assume that $Y=\mathrm{Spec}(B)$ is affine. Moreover, we can use descent along open immersions in $Y'$ to reduce to the case that $Y'=\mathrm{Spec}(B[\tfrac 1h])$ is a standard open affine. All functors commute with colimits, so it suffices to check the isomorphism on perfect complexes on $X'$; moreover, we can assume that they arise via pullback from $X$. In this case, it follows from the base change compatibility of $g_\sharp$ upon inverting $h$.
\end{proof}

Thus, if we would have the required Nagata compactifications of finite Tor-dimension (and suitably generalize the construction principle of Lecture IV to allow situations where $P$ does not satisfy the 2-out-of-3 property), we would get a $6$-functor formalism
\[
X\mapsto D_{\mathrm{qc}}(X)
\]
on all derived schemes, with values in presentable stable $\infty$-categories. The class $E$ would conists of separated maps of finite Tor-dimension.

This time, Grothendieck--Serre duality is encoded not in terms of $f$-suaveness, but in terms of $f$-primness.

\begin{theorem}\label{thm:option2lciproper} In this putative $6$-functor formalism, any map $f: X\to Y$ in $E$ that is a local complete intersection has the property that $\mathcal O_X$ is $f$-prim, with $f$-prim dual being invertible, and the inverse of $\omega_{X/Y}$.
\end{theorem}

One could wonder whether in fact $\mathcal O_X$ is $f$-prim for all $f\in E$, but this turns out to be wrong. In Option 5, we will later correct this.

There is something very mind-bending about this situation. There is a way to make things more intuitive by passing to opposite categories, so we will prove the theorem after a reinterpretation.

\subsection{Option 3: Replace $\mathrm{Ind}(\mathrm{Perf})$ by $\mathrm{Pro}(\mathrm{Perf})$}

Recall that replacing $D(X)$ by $D(X)^{\mathrm{op}}$ replaces cohomologically \'etale maps by cohomologically proper maps. As in Option 2, there is an apparent mismatch between geometric notions and the abstract notions defined in terms of $D$, it may be psychologically useful to switch from
\[
D_{\mathrm{qc}}(X) = \mathrm{Ind}(\mathrm{Perf}(X))
\]
to its opposite, which is $\mathrm{Pro}(\mathrm{Perf}(X))$ (noting that $\mathrm{Perf}(X)$ is self-dual, via the naive duality functor). A formalism of this type (or maybe rather the next option below) was first investigated by Deligne in the appendix of \cite{DeligneAppendix}. This has the disadvantage that $\mathrm{Pro}(\mathrm{Perf}(X))$ is no longer presentable, so we will generally have to be content with a $3$-functor formalism (with occasionally defined right adjoints).

The putative formalism from Option 2 then dualizes to a $3$-functor formalism on $\mathrm{Pro}(\mathrm{Perf}(X))$. Let us use the standard notation $f_!$ for the functors for $f\in E$. Then $f_! = \overline{f}_\ast j_!$ as usual, where $j_!$ is a left adjoint of $j^\ast$ and $\overline{f}_\ast$ is a right adjoint of $\overline{f}^\ast$ (and $j$ is an open immersion and $\overline{f}\in E$ is proper).

\begin{example} Assume that $Y=\mathrm{Spec}(k)$ is the spectrum of a field. Then the category of Pro-objects of finite-dimensional $k$-vector spaces is the category of linearly compact $k$-vector spaces. For any separated scheme of finite type $X$ over $k$ and any $E\in \mathrm{Perf}(X)$, the object
\[
R\Gamma_c(X,E):=f_! E\in \mathrm{Pro}(\mathrm{Perf}(k))
\]
defines a bounded complex of linearly compact $k$-vector spaces. Concretely, to compute it take any compactification $\overline{X}\supset X$ whose boundary is a Cartier divisor $D\subset \overline{X}$. Assume that $E$ extends to $\overline{E}$ on $\overline{X}$ (which always happens up to retracts). Then
\[
R\Gamma_c(X,E) = \mathrm{lim}_n R\Gamma(\overline{X},\overline{E}(-nD))\in \mathrm{Pro}(\mathrm{Perf}(k)).
\]
\end{example}

In this language, Theorem~\ref{thm:option2lciproper} becomes the following result.

\begin{theorem}\label{thm:option3lcismooth} Let $f: X\to Y$ be any map in $E$ that is a local complete intersection. Then $f$ is cohomologically smooth.
\end{theorem}

We can then define $\omega_{X/Y} = f^! \mathcal O_Y\in \mathrm{Perf}(X)$ (which is then defined). We note that this is automatically local on $X$.

\begin{proof} In this formalism, open immersions are cohomologically \'etale, so we can localize (on $X$ and $Y$), and hence assume that $X$ and $Y$ are affine. By the assumption that $f$ is a local complete intersection, and using base change compatibility, one can in fact reduce further to the case of $\mathbb A^1_{\mathbb Z}\to \mathrm{Spec}(\mathbb Z)$ or its zero section. The case of $\mathbb A^1$ can be reduced to $\mathbb P^1$. In those cases, the right adjoint $f^!$ to $f_!=f_\ast$ actually exists, preserves perfect complexes, and commutes with any base change (using Proposition~\ref{prop:option1smooth}, upon restricting from $\mathrm{Ind}(\mathrm{Perf})$ to $\mathrm{Perf}$, and extending back to $\mathrm{Pro}(\mathrm{Perf})$). Thus, $\mathcal O_X$ is $f$-suave, but by Proposition~\ref{prop:option1smooth} the object $f^! \mathcal O_X$ is invertible.
\end{proof}

As we observed above, it is very annoying that the maps in $E$ have to be restricted to be of finite Tor-dimension; this precludes in particular the 2-out-of-3 property. The underlying reason is that in Theorem~\ref{thm:finitenesscoherentcohom} this restriction appears in relation to perfect complexes.

\subsection{Option 4: Replace $\mathrm{Pro}(\mathrm{Perf})$ by $\mathrm{Pro}(\mathrm{Coh})$}

However, we can artificially replace $\mathrm{Perf}$ by $\mathrm{Coh}$, the coherent complexes, at least for noetherian schemes $X$ (i.e.~the underlying classical scheme is noetherian, and $\pi_i \mathcal O_X$ is a coherent $\pi_0 \mathcal O_X$-module for all $i$).

Restricting to noetherian derived schemes, we send any $X$ to the stable $\infty$-category $\mathrm{Pro}(\mathrm{Coh}(X))$. We note that by restricting the symmetric monoidal $\infty$-category $D_{\mathrm{qc}}(X)$ to $\mathrm{Coh}(X)$, one still has an $\infty$-operad, which after passage to $\mathrm{Pro}(\mathrm{Coh}(X))$ is again a symmetric monoidal $\infty$-category. Concretely, the tensor product of two objects of $\mathrm{Coh}(X)$ need not itself lie in $\mathrm{Coh}(X)$ (in general, it just lies in $\mathrm{PCoh}(X)$), but it still defines a pro-object of $\mathrm{Coh}(X)$ (the Pro-system of its truncations $\tau_{\leq n}$). Similar remarks apply to pullback functoriality. We note that the resulting pullback and tensor product functors on $\mathrm{Pro}(\mathrm{Coh})$ automatically commute with all limits.

We let $I$ be the class of open immersions, and $P$ the class of proper maps; then $E$ is the class of separated maps of almost finite type (of noetherian derived schemes).

\begin{theorem}\label{thm:option4works} The functor $X\mapsto \mathrm{Pro}(\mathrm{Coh}(X))$ satisfies the hypotheses of Theorem~\ref{thm:construct6functors} with respect to $I$ and $P$, and hence defines a $3$-functor formalism on noetherian derived schemes.
\end{theorem}

The proof is a variation on the same observations as before. Compared to the previous option, there are many more $f$-suave objects. In particular, we get a stronger result on biduality.

\begin{proposition}\label{prop:option4ula} Let $f: X\to Y$ be a separated map of noetherian derived schemes of almost finite type and assume that $K\in \mathrm{Coh}(X)\subset \mathrm{Pro}(\mathrm{Coh}(X))$ is of finite Tor-dimension over $Y$. Then $K$ is $f$-suave.

In particular, if $Y$ is a regular classical scheme, then all objects $K\in \mathrm{Coh}(X)$ are $f$-suave, and $\mathrm{Coh}(X)$ is self-dual with respect to the Verdier duality $\mathbb D_f(K) = \mathscr{H}\mathrm{om}(K,f^! \mathcal O_Y)$.
\end{proposition}

We note in particular that if $f$ is itself of finite Tor-dimension, we get the dualizing complex $\omega_{X/Y}=f^!\mathcal O_Y\in \mathrm{Coh}(X)$, in a way that localizes on $X$, fulfilling a promise made in Option 1 to show that the dualizing complex is of local nature in general.

\begin{proof} The assertion is local on $X$ and $Y$, so we can assume that they are affine. Replacing $Y$ by an affine space over $Y$, which is cohomologically smooth, one can then assume that $f$ is a closed immersion. In that case $K$ becomes perfect as an object of $\mathrm{Coh}(Y)$, and its naive $Y$-dual (with induced $\mathcal O_X$-module structure) defines the desired $f$-suave dual.
\end{proof}

\subsection{Option 5: Replace $\mathrm{Pro}(\mathrm{Coh})$ by $\mathrm{Ind}(\mathrm{Coh})$}

Finally, we can dualize again, and replace $\mathrm{Pro}(\mathrm{Coh}(X))$ by its opposite $\infty$-category $\mathrm{Ind}(\mathrm{Coh}(X)^{\mathrm{op}})$. If one restricts to schemes almost of finite type over a classical regular base scheme $Y$ (e.g., a field $k$), one can use the self-duality from Proposition~\ref{prop:option4ula} to identify $\mathrm{Ind}(\mathrm{Coh}(X)^{\mathrm{op}})$ with $\mathrm{Ind}(\mathrm{Coh}(X))$. Now the $\infty$-categories are again presentable (and all $3$ functors commute with colimits), so one gets a full $6$-functor formalism. This gives exactly the $6$-functor formalism constructed by Gaitsgory--Rozenblyum \cite{GaitsgoryRozenblyum}.

Working over a classical regular base scheme $Y$, we note that under this identification
\[
\mathrm{Ind}(\mathrm{Coh}(X)^{\mathrm{op}})\cong \mathrm{Ind}(\mathrm{Coh}(X)),
\]
the pullback, tensor product, and exceptional functors on $\mathrm{Ind}(\mathrm{Coh}(X))$ are Verdier dual to the usual pullback, tensor product, and lower shriek functors; they are thus denoted by $f^!$, $\otimes^!$, and $f_\ast$ in the work of Gaitsgory--Rozenblyum, even while they play the role of $f^\ast$, $\otimes$ and $f_!$ in this $6$-functor formalism.

Let us end with a remark on the relation between $D_{\mathrm{qc}}$ and $\mathrm{Ind}(\mathrm{Coh})$. As in \cite{GaitsgoryRozenblyum}, take as our category $C$ the affine derived schemes almost of finite type over a field $k$ (of any characteristic, for now). We have two functors to symmetric monoidal presentable stable $\infty$-categories
\[
X\mapsto D_{\mathrm{qc}}(X), X\mapsto \mathrm{Ind}(\mathrm{Coh}(X))
\]
from Option 2 and Option 5, respectively. We note that both theories satisfy Zariski descent, so in the limit above it suffices to restrict to affine $X$. In \cite{GaitsgoryRozenblyum}, sometimes a restriction to truncated $X$ is made. This does not change $\mathrm{Ind}(\mathrm{Coh}(X))$, in the following sense: If $X=\mathrm{Spec}(A)$ is affine, then the functor
\[
\mathrm{Ind}(\mathrm{Coh}(A))\to \mathrm{lim}_n \mathrm{Ind}(\mathrm{Coh}(\tau_{\leq n} A))
\]
is an equivalence. (To see this, note that the pullback functors here are concretely realized as $!$-functors, and have the forgetful functors $\mathrm{Coh}(\tau_{\leq n} A)\to \mathrm{Coh}(A)$ as left adjoints. Thus, this limit can also be computed as the colimit of presentable stable $\infty$-categories along these left adjoint functors that preserve compact objects. Then it follows from
\[
\mathrm{colim}_n \mathrm{Coh}(\tau_{\leq n} A)\to \mathrm{Coh}(A)
\]
being an equivalence of small stable $\infty$-categories, i.e.~on a bounded complex the $A$-action always factors (uniquely) over $\tau_{\leq n} A$ for large enough $n$.) This means that in the limit defining $\mathrm{IndCoh}(\tilde{X})$, one can even restrict to affine $X=\mathrm{Spec}(A)$ for which $A$ is $n$-truncated for some $n$.

We note that there is a symmetric monoidal natural transformation
\[
D_{\mathrm{qc}}(X)\to \mathrm{Ind}(\mathrm{Coh}(X))
\]
given by
\[
D_{\mathrm{qc}}(X)=\mathrm{Ind}(\mathrm{Perf}(X))\cong \mathrm{Ind}(\mathrm{Perf}(X)^{\mathrm{op}})\to \mathrm{Ind}(\mathrm{Coh}(X)^{\mathrm{op}})\cong \mathrm{Ind}(\mathrm{Coh}(X))=\mathrm{IndCoh}(X),
\]
where the first equivalence is naive duality, and the last equivalence is Serre duality (over $k$). (The composite functor is actually just tensoring with the dualizing complex of $X$, which is the tensor unit of $\mathrm{IndCoh}(X)$.) If $X$ is truncated (i.e.~$\pi_i X=0$ for $i\gg 0$) then this comparison functor is fully faithful.

\begin{remark} We note that the arguments of this lecture give in particular a construction of the $6$-functor formalism $\mathrm{IndCoh}$ of Gaitsgory--Rozenblyum that does not rely on difficult properties of the Gray tensor product of $(\infty,2)$-categories; in fact, it makes use of no $(\infty,2)$-categorical machinery at all.
\end{remark}

\newpage

\section*{Appendix to Lecture VIII: $D$-modules}

As in the work of Gaitsgory--Rozenblyum \cite{GaitsgoryRozenblyum}, one can recover the $6$-functor formalism of $D$-modules (on algebraic varieties over a field $k$ of characteristic $0$), originally developed by Sato, Kashiwara, Bernstein, ...

In fact, we will give a presentation that avoids the use of the $\mathrm{IndCoh}$-formalism. Indeed, on de Rham stacks, the $\mathrm{IndCoh}$- and $D_{\mathrm{qc}}$-formalisms agree, and we will work directly with the latter, equipped with the six-functor formalism from Option 1. (This formalism seems awkward as all maps are considered proper, but it will become more natural from the point of view of analytic stacks developed later.)

Before starting to use the formalism from the lecture, let us announce directly a way to introduce the formalism. Fix a field $k$ of characteristic $0$,\footnote{We actually do not use the assumption of characteristic $0$ for a long time. It is only used in the identification with modules over the ring of differential operators, and in the discussion of Poincar\'e duality -- in characteristic $p$ the resulting theory is a $6$-functor formalism but does not satisfy Poincar\'e duality.} and consider the category $C$ of separated schemes of finite type over $k$, with all morphisms allowed in $E$. To any $X\in C$, we can associate the category of crystals on the infinitesimal site; concretely,
\[
\mathrm{Crys}(X) = \mathrm{lim}_{R,\mathrm{Spec}(R_{\mathrm{red}})\to X} D(R)
\]
as $R$ runs over $k$-algebras of finite type.\footnote{One can also allow animated $k$-algebras almost of finite type; it is not hard to see that the two results are equivalent, via a cofinality argument.} This gives a contravariant functor from $C$ to symmetric monoidal presentable stable $\infty$-categories.

\begin{theorem}\label{thm:dmodule6functors0} The functor $X\mapsto \mathrm{Crys}(X)$ satisfies the hypotheses of Theorem~\ref{thm:construct6functors} with respect to the classes $I$ of proper maps and $P$ of open immersions, and hence extends to a $6$-functor formalism.
\end{theorem}

Our approach in this appendix is actually to construct a $6$-functor formalism directly and verify that proper maps are cohomologically smooth while open immersions are cohomologically proper; this proves Theorem~\ref{thm:dmodule6functors0}. To do so, we use the extension of the coherent $6$-functor formalism to stacks. By Theorem~\ref{thm:stacky6functors}, the formalism $X\mapsto D_{\mathrm{qc}}(X)$ (with all maps in $P$) extends to stacks on $C$ (i.e.~sheaves of anima), sending any stack $\tilde{X}$ to
\[
D_{\mathrm{qc}}(\tilde{X}) = \mathrm{lim}_{X\in C,X\to \tilde{X}} D_{\mathrm{qc}}(X).
\]

One way to motivate the upcoming definition of $D$-modules is to observe that one respect in which the coherent $6$-functor formalisms differ from the previous ones is that they are not nil-invariant: If $\mathrm{Spec}(A)\to \mathrm{Spec}(B)$ is a nilpotent immersion, i.e.~$\pi_0 B\to \pi_0 A$ is surjection with nilpotent kernel, then $D(B)\to D(A)$ is not usually an equivalence.

There is a universal way to make a theory nil-invariant. Namely, for any derived scheme $X$, let $X_{\mathrm{dR}}$ denote the prestack taking an animated commutative ring $A$ to
\[
\mathrm{colim}_{I\subset \pi_0 A} X(A_{\mathrm{red}})
\]
where $A_{\mathrm{red}}$ is the (static) reduced quotient of $A$. For a nil-invariant theory, the pullback functor $D(X_{\mathrm{dR}})\to D(X)$ is an equivalence; and in general $X\mapsto D(X_{\mathrm{dR}})$ is nil-invariant.

\begin{definition}\label{def:Dmodules} Let $k$ be a field of characteristic $0$ and let $X$ be a derived scheme almost of finite type over $k$. The symmetric monoidal presentable stable $\infty$-category of $D$-modules on $X$ is
\[
\mathrm{Dmod}(X) := D_{\mathrm{qc}}(X_{\mathrm{dR}}).
\]
\end{definition}

Thus, $\mathrm{Crys}(X) = \mathrm{Dmod}(X)$ for the above definition of $\mathrm{Crys}(X)$.

\begin{remark} From the definition it follows that if $X$ is a derived scheme almost of finite type with underlying reduced classical scheme $X_{\mathrm{red}}$, the map $X_{\mathrm{red},\mathrm{dR}}\to X_{\mathrm{dR}}$ is an isomorphism, hence $\mathrm{Dmod}(X) = \mathrm{Dmod}(X_{\mathrm{red}})$. Thus, in the following we can often restrict attention to (reduced) classical schemes.
\end{remark}

Our goal is to prove the following theorem, which implies Theorem~\ref{thm:dmodule6functors0}.

\begin{theorem}\label{thm:Dmodule6functors} For any map $f: X\to Y$ of schemes of finite type over $k$, the $!$-functors on $D_{\mathrm{qc}}$ are defined for $f_{\mathrm{dR}}: X_{\mathrm{dR}}\to Y_{\mathrm{dR}}$. Moreover, if $f$ is \'etale, then $f_{\mathrm{dR}}$ is cohomologically proper, while if $f$ is proper, then $f_{\mathrm{dR}}$ is cohomologically smooth.
\end{theorem}

To get started, we have the following proposition regarding ``formal completions''.

\begin{proposition}\label{prop:formalcompletioncohometale} Let $X$ be a derived scheme almost of finite type over $k$ and let $Z\subset X$ be a closed subscheme. Let $X^\wedge_Z\subset X$ be the subfunctor of those maps $\mathrm{Spec}(A)\to X$ that set-theoretically factor over $Z$. Then the map $j: X^\wedge_Z\hookrightarrow X$ admits $!$-functors in the $D_{\mathrm{qc}}$-formalism, and is cohomologically \'etale. It is the $!$-image of $Z\to X$.
\end{proposition}

\begin{proof} The question is local on $X$, so we can assume that $X$ is affine, and $Z$ is the vanishing locus of some functions $f_1,\ldots,f_n$; and by induction on $n$ and stability under pullback, we can assume that $n=1$. In that case, the situation comes via pullback from $X=\mathrm{Spec}(k[T])$ and $Z=\mathrm{Spec}(k)$ embedded as $T=0$. The map $Z\to X$ is cohomologically smooth as a proper local complete intersection, and hence its image (in the sheaf-theoretic sense) admits a $!$-functor and is a cohomologically \'etale monomorphism. Indeed, this image admits a cohomologically smooth surjection from $Z$ (as any base change agrees with a base change of $Z\to X$) and $Z\to X$ is cohomologically smooth.

We claim that the image of $Z\to X$ is given by $X^\wedge_Z\subset X$. It is clearly contained in $X^\wedge_Z\subset X$. Conversely, if $Y=\mathrm{Spec}(B)$ is some affine derived scheme mapping towards $X^\wedge_Z$, then it factors over $\mathrm{Spec}(k[T]/T^n)$ for some $n$, so it suffices to see that $\mathrm{Spec}(k[T]/T^n)\to X=\mathrm{Spec}(k[T])$ is contained in the image of $Z\to X$, i.e.~$\mathrm{Spec}(k)\to \mathrm{Spec}(k[T]/T^n)$ is a cover in the $D$-topology. But this follows from Proposition~\ref{prop:Dprimdescent} as $k[T]/T^n\to k$ -- like any nilpotent thickening -- is a descendable map in the sense of Mathew \cite[Proposition 3.33]{MathewDescendable}.
\end{proof}

It may be useful to describe what happens in terms of the derived categories.

\begin{proposition}\label{prop:DXonZ} In the situation of Proposition~\ref{prop:formalcompletioncohometale}, the functor
\[
j_!: D_{\mathrm{qc}}(X^\wedge_Z)\to D_{\mathrm{qc}}(X)
\]
is fully faithful and its essential image is given by $D_{\mathrm{qc}}(X\ \mathrm{on}\ Z)$, i.e.~those objects that vanish outside $Z$.
\end{proposition}

\begin{proof} For any cohomologically \'etale monomorphism $j$, the functor $j_!$ is fully faithful as $j^\ast j_!=\mathrm{id}$ by base change. As $Z\to X^\wedge_Z$ is surjective, the kernel of $j^\ast$ agrees with the kernel of $D_{\mathrm{qc}}(X)\to D_{\mathrm{qc}}(Z)$; and $D_{\mathrm{qc}}(X\ \mathrm{on}\ Z)\to D_{\mathrm{qc}}(Z)$ is known to be conservative. (The point is that if the base change of a complex to the vanishing locus of some function $f$ is zero, then $f$ acts invertibly on that complex, and hence lives on the open subset where $f$ is invertible.)
\end{proof}

\begin{proposition}\label{prop:XtoXdR} For a smooth $k$-scheme $X$, the map $\pi: X\to X_{\mathrm{dR}}$ admits $!$-functors in the $D_{\mathrm{qc}}$-formalism and is a cohomologically smooth $!$-cover.
\end{proposition}

In particular, by Barr--Beck--Lurie, $\pi^! \pi_!$ defines a monad on $D_{\mathrm{qc}}(X)$ and modules over it identify with $D_{\mathrm{qc}}(X_{\mathrm{dR}})$. Unraveling the description of the monad yields the description of $D_{\mathrm{qc}}(X_{\mathrm{dR}})$ as $D_X$-modules (with the pullback to $D_{\mathrm{qc}}(X)$ corresponding to passing to the underlying quasicoherent $\mathcal O_X$-module).

\begin{proof} The question is local on $X$, so we can restrict first to affine $X$, and assume that $X$ has an \'etale map to $\mathbb A^n$. Then $X\to X_{\mathrm{dR}}$ is pulled back from $\mathbb A^n\to \mathbb A^n_{\mathrm{dR}}$, so we can reduce to $X=\mathbb A^n$. Taking products, we can assume $n=1$. We may also consider $X=\mathbb P^1$ instead. Thus, assume from now on that $X$ is proper (and smooth).

We have to see that for any derived scheme $Y$ almost of finite type over $k$ with a map $Y\to \mathbb P^1_{\mathrm{dR}}$, the fibre product $X\times_{X_{\mathrm{dR}}} Y\to Y$ admits $!$-functors, and is cohomologically smooth and surjective. We note that as we work everywhere with sheaves of anima for the $D$-topology, the map $Y\to X_{\mathrm{dR}}$ is a map to the sheafified version of $X_{\mathrm{dR}}$; but because the condition of admitting $!$-functors is local on source and target, we can assume that $Y=\mathrm{Spec}(B)$ is affine and the map $Y\to X_{\mathrm{dR}}$ arises from a map $Y_{\mathrm{red}}\to X$.

Now we first claim that the monomorphism $X\times_{X_{\mathrm{dR}}} Y\to X\times_k Y$ admits $!$-functors and is cohomologically \'etale. In fact, the graph of the map $Y_{\mathrm{red}}\to X$ gives a closed subset $Z\subset X\times_k Y$ and then
\[
X\times_{X_{\mathrm{dR}}} Y = (X\times_k Y)^\wedge_Z,
\]
so the result follows from Proposition~\ref{prop:formalcompletioncohometale}. But the projection $X\times_k Y\to Y$ is also cohomologically smooth (as $X$ is proper and smooth). It is easy to see that the map $X\to X_{\mathrm{dR}}$ is surjective.
\end{proof}

At this point, we can prove Theorem~\ref{thm:Dmodule6functors}.

\begin{proof}[Proof of Theorem~\ref{thm:Dmodule6functors}] First, $f_{\mathrm{dR}}: X_{\mathrm{dR}}\to Y_{\mathrm{dR}}$ admits $!$-functors. As the class of maps admitting $!$-functors is stable under diagonals and composites, it suffices to prove this for the projection maps to the point. We can also work locally, so assume that $X$ is affine. Then $X$ is a fibre product $X=\mathbb A^n\times_{\mathbb A^m} \ast$, and by stability under pullbacks, we can also reduce to the case $X=\mathbb A^n$. Now Proposition~\ref{prop:XtoXdR} gives a cohomologically smooth $!$-able cover $X\to X_{\mathrm{dR}}$, giving the claim.

If $f$ is \'etale, then $f_{\mathrm{dR}}$ is cohomologically proper. Namely, $f_{\mathrm{dR}}$ pulls back to $f$ under the cover $Y\to Y_{\mathrm{dR}}$, and $f$ is cohomologically proper (as all maps of schemes are in $P$ in the $D_{\mathrm{qc}}$-formalism).

If $f$ is proper, we want to show that $f_{\mathrm{dR}}$ is cohomologically smooth. By Chow's theorem, we can assume that $f$ is projective. This reduces us to proving the assertion for closed immersions, and for projections $\mathbb P^n_{\mathrm{dR}}\to \ast$. For a closed immersion, the result follows from Proposition~\ref{prop:formalcompletioncohometale}. For $\mathbb P^n_{\mathrm{dR}}$, it follows from Proposition~\ref{prop:XtoXdR} and $\mathbb P^n$ itself being cohomologically smooth.
\end{proof}

\begin{remark} We warn the reader that the $6$-functor formalism from Theorem~\ref{thm:dmodule6functors0} is not quite the specialization of the $6$-functor formalism $D_{\mathrm{qc}}$ when applied to de Rham stacks: Namely, for proper maps, we defined $f_{\mathrm{Dmod},!}$ as the left adjoint of the pullback functor, while $f_{D_{\mathrm{qc}},!}$ differs from it by a shift (by twice the dimension).
\end{remark}

The $D$-module $6$-functor formalism actually has a property stronger than nil-invariance: Namely, it satisfies excision. This result is classically known as ``Kashiwara's lemma''.

\begin{proposition}\label{prop:Dmoduleexcision} Let $X$ be a finite type $k$-scheme with a closed subscheme $j: Z\subset X$ and open complement $i: U\subset X$.\footnote{We apologize for the nonstandard notation; it is however dictated by the $6$-functor formalism, where the closed immersion $j$ is cohomologically \'etale and the open immersion $i$ is cohomologically proper.} For any $A\in \mathrm{Dmod}(X)$, the triangle
\[
j_! j^\ast A\to A\to i_\ast i^\ast A
\]
is exact, yielding a semi-orthogonal decomposition of $\mathrm{Dmod}(X)$ into $\mathrm{Dmod}(Z)$ and $\mathrm{Dmod}(U)$.
\end{proposition}

\begin{proof} It suffices to check this after pullback along the surjection $X\to X_{\mathrm{dR}}$, and then the strata become $X\times_{X_{\mathrm{dR}}} U_{\mathrm{dR}} = U$ and $X\times_{X_{\mathrm{dR}}} Z_{\mathrm{dR}} = X^\wedge_Z$. The result follows from Proposition~\ref{prop:DXonZ}.
\end{proof}

Moreover, we have the following important result on duality.

\begin{theorem}\label{thm:coherentDmodules} For all finite type $k$-schemes $X$, the presentable stable $\infty$-category $\mathrm{Dmod}(X)$ is compactly generated, and the compact objects agree with the $f_{\mathrm{Dmod}}$-prim objects for $f: X\to \mathrm{Spec}(k)$. In particular, duality for $f_{\mathrm{Dmod}}$-prim objects gives a selfduality on the compact objects of $\mathrm{Dmod}(X)$.

Moreover:
\begin{enumerate}
\item If $X$ is smooth of dimension $d$, then $1_{\mathrm{Dmod(X)}}$ is prim, and the $f_{\mathrm{Dmod}}$-prim dual of $1_{\mathrm{Dmod(X)}}$ is isomorphic to $1_{\mathrm{Dmod}(X)}[-2d]$.
\item If $f: X\to Y$ is a proper map, then $f_{!,\mathrm{Dmod}}$ preserves compact objects, and commutes with prim duality.
\end{enumerate}
\end{theorem}

Under the identification of $\mathrm{Dmod}(X)$ with the derived category of $D_X$-modules (for $X$ smooth), the compact objects are of course the $D_X$-coherent objects. Assertion (1) is encoding Poincar\'e duality.

\begin{proof} Choose a resolution of singularities $\tilde{X}\to X$ (an alteration would be enough). The map $\tilde{X}\to X_{\mathrm{dR}}$ is a cohomologically smooth cover (as the composite of the cohomologically smooth covers $\tilde{X}\to \tilde{X}_{\mathrm{dR}}\to X_{\mathrm{dR}}$), and hence the (colimit-preserving) pullback $D_{\mathrm{qc}}(X_{\mathrm{dR}})\to D_{\mathrm{qc}}(\tilde{X})$ has a left adjoint, which thus preserves compact objects. As $D_{\mathrm{qc}}(\tilde{X})$ is compactly generated (and pullback is conservative), so is $\mathrm{Dmod}(X)$. Moreover, any $K\in \mathrm{Perf}(\tilde{X})$ is $\tilde{\pi}$-prim in the $D_{\mathrm{qc}}$-formalism for $\tilde{\pi}: X\to \ast$. As cohomologically smooth pushforwards (like $g: \tilde{X}\to X_{\mathrm{dR}}$) preserve prim objects, this implies that $g_{!,D_{\mathrm{qc}}} K\in D_{\mathrm{qc}}(X_{\mathrm{dR}})$ is $f_{\mathrm{dR}}$-prim, i.e.~as an object of $\mathrm{Dmod}(X)$ it is $f_{\mathrm{Dmod}}$-prim (as primness does not see the different twists in the lower-$!$-functors). These generate all compact objects, so all compact objects are $f_{\mathrm{Dmod}}$-prim. The converse follows from Proposition~\ref{prop:primsheafcompact}.

In part (1), one first shows that it is coherent and its prim dual is locally isomorphic to $1_{\mathrm{Dmod}(X)}[-2d]$; this reduces to a computation on $\mathbb A^1$ (and is the only place we use characteristic $0$). Then there are various ways to conclude a global isomorphism, for example using a deformation to the normal cone. Part (2) follows from the preservation of prim objects under cohomologically \'etale pushforward (and the preservation of prim duality).
\end{proof}

\begin{remark} There is yet another interpretation of $\mathrm{Dmod}(X)$, or rather of its opposite $\mathrm{Dmod}(X)^{\mathrm{op}}$. Namely,
\[\begin{aligned}
\mathrm{Dmod}(X)^{\mathrm{op}} &= \mathrm{ProPerf}(X_{\mathrm{dR}}) = \mathrm{lim}_{R,\mathrm{Spec}(R_{\mathrm{red}})\to X} \mathrm{Pro}(\mathrm{Perf}(R))
\end{aligned}\]
using the formalism from Option 3 in the lecture. In this interpretation, the \'etale maps are cohomologically \'etale and the proper maps are cohomologically proper. Moreover, the excision triangle takes the usual form: If $X$ has an open subscheme $j: U\subset X$ with closed complement $i: Z\subset X$, then for all $A\in \mathrm{ProPerf}(X_{\mathrm{dR}})$, one has an exact triangle
\[
j_! j^\ast A\to A\to i_\ast i^\ast A.
\]
Moreover, in this realization, Theorem~\ref{thm:coherentDmodules} says that all co-compact objects of $\mathrm{Dmod}^{\mathrm{op}}$ are $f$-suave over the point, the duality becomes actual Verdier duality, and part (1) says that the map $X\to \ast$ is cohomologically smooth if $X$ is smooth (with dualizing complex $1_{\mathrm{Dmod}^{\mathrm{op}}(X)}[2d]$).
\end{remark}

\begin{example} Let us give some examples of $D$-modules, and discuss their different realizations. Let us take the affine line $X=\mathbb A^1_k=\mathrm{Spec}(k[T])$ with its closed point $Z=\mathrm{Spec}(k)=\{0\}\subset \mathbb A^1_k=X$; let $U=\mathrm{Spec}(k[T^{\pm 1}])$ be its open complement. Following our previous terminology, we write $j: Z\to X$ and $i: U\to X$ for the morphisms, and we want to understand the different realizations of the exact triangle
\[
j_{!,\mathrm{Dmod}}(1_{\mathrm{Dmod}(Z)})\to 1_{\mathrm{Dmod}(X)}\to i_{\ast,\mathrm{Dmod}}(1_{\mathrm{Dmod(U)}}).
\]
Here $i_{\ast,\mathrm{Dmod}}$ is, as in any $6$-functor formalism, the right adjoint to $i^\ast_{\mathrm{Dmod}}$; but because $i$ is open and thus cohomologically proper, it also agrees with $i_{!,\mathrm{Dmod}}$. The functor $j_{!,\mathrm{Dmod}}$ is the left adjoint of $j^\ast_{\mathrm{Dmod}}$, as $j$ is proper and thus cohomologically \'etale.

Let us first realize under the functor $\mathrm{Dmod}(X)\to D_{\mathrm{qc}}(X)$; this is the usual realization as left $D$-modules. This functor is symmetric monoidal for the usual tensor product. Also, $i_{\ast,\mathrm{Dmod}}$ pulls back to the the functor $i_{\ast,D_{\mathrm{qc}}}$ as $U=X\times_{X_{\mathrm{dR}}} U_{\mathrm{dR}}$. Thus, we can identify the last two terms of the sequence above as $k[T]\to k[T^{\pm 1}]$, and thus the whole exact triangle must be
\[
(k[T^{\pm 1}]/k[T])[-1]\to k[T]\to k[T^{\pm 1}].
\]
In particular,
\[
j_{!,\mathrm{Dmod}}(1_{\mathrm{Dmod}(Z)}) = (k[T^{\pm 1}]/k[T])[-1].
\]
We note that this module can also be written as the local cohomology of the structure sheaf.

Now let us analyze the realization via $\mathrm{ProPerf}$. In this realization, maps within the category go the other way, so the triangle rather becomes a triangle
\[
j_{!,\mathrm{Dmod}^{\mathrm{op}}}(1_{\mathrm{Dmod}^{\mathrm{op}}(U)})\to 1_{\mathrm{Dmod}^{\mathrm{op}}(X)}\to i_{\ast,\mathrm{Dmod}^{\mathrm{op}}}(1_{\mathrm{Dmod}^{\mathrm{op}}(Z)}),
\]
where also now we switched to $j: U\subset X$ and $i: Z\subset X$. Under $\mathrm{Dmod}^{\mathrm{op}}(X) = \mathrm{ProPerf}(X_{\mathrm{dR}})\to \mathrm{ProPerf}(X)$, this realizes to
\[
(k[[T]]/k[T])[-1]\to k[T]\to k[[T]],
\]
where $k[[T]]$ denotes the pro-object $(k[T]/T^n)_n$.
\end{example}

\begin{example}\label{ex:dualrealization} Continuing our previous example, we note that all objects involved are $D$-coherent, so we also get a triangle of their duals. Let us write $\mathbb D_X$ for the duality on the compact objects of $\mathrm{Dmod}(X)$. We have $\mathbb D(1_{\mathrm{Dmod}(X)}) = 1_{\mathrm{Dmod}(X)}[-2]$ and
\[
\mathbb D(j_{!,\mathrm{Dmod}}(1_{\mathrm{Dmod}(Z)})) = j_{!,\mathrm{Dmod}}(1_{\mathrm{Dmod}(Z)}).
\]
(This equation may look confusing, but recall that $\mathbb D$ is prim duality, and this commutes with $j_!$ for cohomologically \'etale maps. Alternatively, think in terms of the $\mathrm{ProPerf}$-picture, where $j_!$ would be denoted $i_!=i_\ast$, and Verdier duality commutes with (cohomologically) proper pushforward.) Thus, the triangle dualizes to a triangle
\[
\mathbb D(i_{\ast,\mathrm{Dmod}}(1_{\mathrm{Dmod}(U)}))\to 1_{\mathrm{Dmod}(X)}[-2]\to j_{!,\mathrm{Dmod}}(1_{\mathrm{Dmod}(Z)}).
\]
In the first realization, this becomes a triangle
\[
M\to k[T][-2]\to (k[T^{\pm 1}]/k[T])[-1]
\]
for some $M\in D(k[T])$. In fact, any such triangle is split, so $M=(k[T]\oplus k[T^{\pm 1}]/k[T])[-2]$, but the $D$-module structure is more subtle. More precisely, on any basis element $(T^n,0)$ or $(0,T^{-n})$ with $n>0$ the derivative $\partial_T$ does the same as the direct sum of the two $D$-modules, but instead of killing $(T^0,1)$, it sends it to $(0,T^{-1})$.

Note that all terms here are concentrated in one cohomological degree; this is related to the exact sequence $0\to i_\ast \mathbb Z[-1]\to j_! \mathbb Z\to \mathbb Z\to 0$ in perverse sheaves on $\mathbb C$ (with $i: \{0\}\hookrightarrow \mathbb C$, $j:\mathbb C^\times\hookrightarrow \mathbb C$).

The second realization, via the $\mathrm{ProPerf}$-formalism, becomes a triangle
\[
k[[T]]\to k[T][2]\to N.
\]
\end{example}

\begin{remark} We should also translate between the functors $f^\ast_{\mathrm{Dmod}}$, $\otimes_{\mathrm{Dmod}}$ and $f_{!,\mathrm{Dmod}}$ from the $6$-functor formalism $\mathrm{Dmod}$ constructed above, and the usual functors from the literature. We take Bernstein's notes \cite{BernsteinDmodules} as our reference. Bernstein takes the perspective of left $D$-modules as the primary one, and we follow him. Our pullback functor $f^\ast_{\mathrm{Dmod}}$ matches Bernstein's naive pullback functor $f^{\triangle}_{\mathrm{Be}}$ (Be for Bernstein); he defines the actual pullback functor as a shift $f^!_{\mathrm{Be}} = f^\triangle_{\mathrm{Be}}$ by $\mathrm{dim}(Y)-\mathrm{dim}(X)$. Similarly, $f_{!,\mathrm{Dmod}}$ corresponds to a shift $f_{\ast,\mathrm{Be}}[\mathrm{dim}(Y)-\mathrm{dim}(X)]$. The tensor products agree.
\end{remark}

\begin{remark} Continuing the previous remark, the usual discussion of $D$-module $6$ functors proceeds in the following way. First, one has the big category of all $D$-modules, on which functors $f^!$ and $f_\ast$ are introduced. Then one restricts to the coherent $D$-modules, which have a selfduality $\mathbb D$, and then tries to dualize $f^!$ and $f_\ast$ to get functors $f^\ast$ and $f_!$. These are in general only well-defined as functors to the $\mathrm{Pro}$-category, yielding $\mathrm{Pro}(\mathrm{Dmod}(X)^\omega)$ with $\otimes$, $f^\ast$ and $f_!$-functors. We note that this exactly agrees with the $\otimes$, $f^\ast$ and $f_!$ we defined on $\mathrm{Dmod}(X)^{\mathrm{op}}\cong \mathrm{Pro}(\mathrm{Dmod}(X)^\omega)$.

Then one defines the class of holonomic $D$-modules and shows that this is stable under $\otimes$, $f^\ast$, $f_!$ and their right adjoints (which, for abstract reasons, are also Verdier duals of $f_!$ and $f^\ast$). The resulting $6$-functor formalism of holonomic $D$-modules will then satisfy excision, \'etale maps are cohomologically \'etale, proper maps are cohomologically proper, and smooth maps are cohomologically smooth, so is a ``usual'' $6$-functor formalism.
\end{remark}

\newpage

\section{Lecture IX: Solid Modules}

As we have seen in the last lecture, the theory of coherent sheaves does not seem to fit into the framework of a $6$-functor formalism easily. The essential reason is that for an open immersion $j: U\subset X$, we would want $j^\ast$ to admit a left adjoint $j_!$, but it does not exist. And the reason for this is that $j^\ast$ does not commute with limits; in fact, not even with countable products. Concretely: If
\[
j: U=\mathrm{Spec}(A[\tfrac 1f])\subset X=\mathrm{Spec}(A)
\]
is a standard open subset $j^\ast$ is given by $-\otimes_A A[\tfrac 1f]$, and the natural map
\[
(\prod_{\mathbb N} A)\otimes_A A[\tfrac 1f]\to \prod_{\mathbb N}(A[\tfrac 1f])
\]
is not an isomorphism. Indeed, the image consists of the subspace of sequences $(a_i/f^{n_i})_i$ where the sequence $n_i$ stays bounded.

One way one can hope to rectify the problem is by thinking of $\prod_{\mathbb N} A$ not as an abstract $A$-module, but as some kind of topological $A$-module, equipped with the product topology; and by thinking of the base change $-\otimes_A A[\tfrac 1f]$ as being some kind of ``completed tensor product'', where the completion produces the corresponding product over $A[\tfrac 1f]$. One essentially tautological way of achieving this is by treating $\prod_{\mathbb N} A$ as a Pro-$A$-module, and this is related to Options 3 and 4 from the last lecture. However, ideally we would also like to retain the good categorical properties like having a presentable ($\infty$-)category of modules, which fails for the $\mathrm{Pro}$-category.

In our lectures on condensed mathematics, \cite{Condensed}, we outlined precisely such a theory of ``complete topological $A$-modules'' and sketched that it defines a $6$-functor formalism. With the current machinery in place for constructing $6$-functor formalisms, this can be turned into a theorem.

\subsection{Reminder on solid modules}

First, we need to recall the definition of ``complete topological $A$-modules'' that we use, which is the theory of solid $A$-modules. Their theory was developed in \cite{Condensed}; let us review the definition briefly.

One problem with the category of topological abelian groups is that it is not an abelian category. In fact, you can have morphisms that are both injective and surjective, but not an isomorphism, as for example $\mathbb R_{\mathrm{disc}}\to \mathbb R$, or really any map relating two different topologies on the same underlying group. One key observation of condensed mathematics is that by replacing topological spaces by the closely related notion of condensed sets, such foundational issues disappear.

\begin{definition}\label{def:condensedset} Consider the category $\mathrm{ProFin}$ of profinite sets, equipped with the Grothendieck topology generated by finite families of jointly surjective maps. A condensed set is a sheaf on $\mathrm{ProFin}$ for this Grothendieck topology.
\end{definition}

\begin{remark} As $\mathrm{ProFin}$ is a large category, some care is required here to avoid set-theoretic problems. For example, one can pick some cutoff cardinal $\kappa$ and work with $\kappa$-small profinite sets. The resulting category of sheaves is called $\kappa$-condensed sets. As you increase $\kappa$, the corresponding categories admit compatible fully faithful embeddings. The category of all condensed sets is by definition the union of the category of $\kappa$-condensed sets over all $\kappa$.
\end{remark}

Any topological space $X$ defines a ($\kappa$-)condensed set $\underline{X}$ by sending a profinite set $S$ to the continuous maps from $S$ to $X$. Restricted to ($\kappa$-)compactly generated $X$, this functor is fully faithful. Thus, in practice, the category of condensed sets serves as an enlargement of the category of topological spaces. Inside this larger category, one can meaningfully take ``wild'' quotients such as $\mathbb R/\mathbb Q$ or even $\mathbb R/\mathbb R_{\mathrm{disc}}$.

One can similarly define condensed abelian groups, and there is a conservative forgetful functor $\mathrm{CondAb}\to \mathrm{CondSet}$ which admits a left adjoint, denoted $X\mapsto \mathbb Z[X]$. Concretely, $\mathbb Z[X]$ is the sheafification of $S\mapsto\mathbb Z[X(S)]$. Condensed abelian groups admit a tensor product, satisfying $\mathbb Z[X]\otimes \mathbb Z[X']\cong \mathbb Z[X\times X']$.

\begin{proposition}\label{prop:CondAb} The category $\mathrm{CondAb}$ is abelian. It satisfies Grothendieck's axioms (AB3)--(AB6) and (AB*3)--(AB*4). In particular, filtered colimits and infinite products exact. Moreover, $\mathrm{CondAb}$ admits projective generators, given by $\mathbb Z[S]$ for extremally disconnected profinite sets $S$.
\end{proposition}

Thus $\mathrm{CondAb}$ is a good place for homological algebra. One can in particular pass to its derived category $D(\mathrm{CondAb})$. (Reminding ourselves of the discussion between derived categories of abelian sheaves versus (hyper)sheaves, this means we are looking at hypersheaves on $\mathrm{ProFin}$. Indeed, in condensed mathematics, one is usually working with hypersheaves on $\mathrm{ProFin}$.)\footnote{It is an open question whether all sheaves of anima on $\mathrm{ProFin}$ are hypersheaves, but it is generally believed that this is not the case.}

At this point, we have a nice abelian category of condensed abelian groups, serving as a replacement for topological abelian groups. However, for our goal, we need to restrict to ``complete'' condensed abelian groups. The theory of ``completeness'' turns out to be surprisingly rich, and somewhat different from the setting of topological abelian groups where it is an ``absolute'' notion. Here, the notion of ``completeness'' is allowed to depend on the base ring, and we even want to allow the possibility to deal with different notions of ``completeness''.

This is encapsulated in the following definition. We state it here immediately in its ``derived'' version for wider applicability. For any condensed animated commutative ring $A^\triangleright$, we have the (connective part of) the derived $\infty$-category $D_{\geq 0}(A^\triangleright)$ of condensed $A^\triangleright$-modules. If $A^\triangleright$ is a usual ring, this is literally the expected subcategory of the derived $\infty$-category of condensed $A^\triangleright$-modules.

\begin{definition}[{\cite[Proposition 12.20]{Analytic}}]\label{def:analyticring} An analytic ring is a condensed animated commutative ring $A^\triangleright$ together with a subcategory $D_{\geq 0}(A)\subset D_{\geq 0}(A^\triangleright)$ of ``complete condensed $A^\triangleright$-modules'', with the following properties:
\begin{enumerate}
\item The subcategory contains $A^\triangleright$, is stable under all limits and colimits, as well internal $\mathrm{Hom}$, and restriction of scalars along Frobenius $A^\triangleright\to A^\triangleright/p: x\mapsto x^p$ (on the subcategory of $p$-torsion objects).
\item The inclusion admits a left adjoint.
\end{enumerate}
\end{definition}

It turns out that these conditions ensure that the left adjoint $D_{\geq 0}(A^\triangleright)\to D_{\geq 0}(A)$ is a Verdier localization whose kernel is a $\otimes$-ideal, and accordingly $D_{\geq 0}(A)$ acquires its own ``completed'' tensor product (as well as $\mathrm{Sym}^n$ operations). By stabilization, everything passes also to unbounded derived categories $D(A)$.

For any profinite set $S$, the completion of $A^\triangleright[S]$ is denoted $A[S]\in D_{\geq 0}(A)$. For $S$ extremally disconnected, these form compact projective generators of $D_{\geq 0}(A)$, which is thus freely generated under sifted colimits by these objects. One can equivalently describe analytic rings in terms of the condensed animated commutative ring $A^\triangleright$ and a functor $S\mapsto A[S]$ from extremally disconnected profinite sets towards $D_{\geq 0}(A^\triangleright)$, with properties that ensure that the $\infty$-category $D_{\geq 0}(A)\subset D_{\geq 0}(A^\triangleright)$ generated by the $A[S]$ satisfies the hypotheses of Definition~\ref{def:analyticring}.

\begin{remark} By \cite[Proposition 12.21]{Analytic}, analytic ring structures depend only on the condensed commutative ring $\pi_0 A^\triangleright$. Moreover, they are also determined by the subcategory $D_{\geq 0}(A)^\heartsuit\subset D_{\geq 0}(A^\triangleright)^\heartsuit$, as an object is complete if and only if all its cohomologies are complete.
\end{remark}

As $A^\triangleright$ is complete, it follows that all condensed $A^\triangleright$-modules that can be built from $A^\triangleright$ using limits and colimits have to be complete, too. In some cases, this minimal choice of $D_{\geq 0}(A)$ does define an analytic ring structure.

\begin{theorem} Let $R$ be a discrete animated commutative ring that is an almost finitely presented $\mathbb Z$-algebra. Then the subcategory $D_{\geq 0}(R_\square)\subset D_{\geq 0}(R)$ generated under colimits by $\prod_I R$ for varying sets $I$ defines an analytic ring structure $A=R_\square$ on $A^\triangleright=R$.

The completion of $R[S]$ for $S=\varprojlim_i S_i$ is given by the static module
\[
R_\square[S] = \varprojlim_i R[S_i].
\]
This is isomorphic to $\prod_I R$ for some set $I$, and these objects form compact projective generators of $\mathcal D_{\geq 0}(R_\square)$ stable under tensor products
\[
\prod_I R\otimes_{R_\square} \prod_J R\cong \prod_{I\times J} R.
\]
\end{theorem}

\begin{proof} For $R=\mathbb Z$, this is \cite[Theorem 5.8, Corollary 6.1]{Condensed}, and for $R=\mathbb Z[T_1,\ldots,T_n]$, it is \cite[Theorem 8.1]{Condensed}. If $\mathbb Z[T_1,\ldots,T_n]\to R$ yields a surjection on $\pi_0$, then $R$ is pseudocoherent as $\mathbb Z[T_1,\ldots,T_n]$-module and hence the map
\[
\prod_I \mathbb Z[T_1,\ldots,T_n]\otimes_{\mathbb Z[T_1,\ldots,T_n]} R\to \prod_I R
\]
is an isomorphism, and so this is just the induced analytic ring structure.
\end{proof}

Note that $D(R_\square)$ is compactly generated, and the subcategory of compact objects $D(R_\square)^\omega$ is generated under finite colimits, shifts and retracts by $\prod_I R$ for varying sets $I$. It follows that under naive duality, $D(R_\square)^\omega$ is anti-equivalent to the subcategory of the usual derived $\infty$-category $D(R)$ of (abstract) $R$-modules generated under finite colimits, shift and retracts by free modules $\bigoplus_I R$. Thus, $D(R_\square)$ is closely related to $\mathrm{Ind}(\mathrm{Pro}(\mathrm{Perf}(R)))$, and the solid formalism is closely related to the formalism from Option 2 of the previous lecture, but passing to $\mathrm{Ind}$-categories again.

\begin{theorem}\label{thm:solid6functorsschemes} Let $C$ be the category of derived schemes almost of finite type over $\mathbb Z$.

By descent, the functor $R\mapsto D(R_\square)$ globalizes to a functor $X\mapsto D_\square(X)$ of quasicoherent sheaves of solid modules, on any $X\in C$. With respect to the class $P$ of proper maps and $I$ of open immersions (and $E$ of separated maps), the functor $X\mapsto D_\square(X)$ defines a $6$-functor formalism.
\end{theorem}

It turns out that both for the proof and for applications, it is better to generalize the result to ``discrete adic spaces''.\footnote{We keep discreteness here for simplicity; it's not required.} Namely, for any map $R\to A$ of animated commutative rings, with $R$ almost of finite type over $\mathbb Z$, we can form the analytic ring $(A,R)_\square$ whose underlying condensed animated commutative ring is $A$ (with discrete condensed structure), and the analytic ring structure is induced from $R$. In other words,
\[
(A,R)_\square[S] = R_\square[S]\otimes_R A
\]
for any profinite set $S$.

\begin{theorem} The analytic ring $(A,R)_\square$ depends only on $A$ and the integral closure $A^+\subset A$ of the image of $R$ in $\pi_0 A$.
\end{theorem}

\begin{proof} By definition of maps of analytic rings, there is at most one map $(A,R)_\square\to (A,R')_\square$ over the identity of $A$. We need to see that this exists if the image of $R$ in $\pi_0 A$ is contained in the integral closure of the image of $R'$ in $\pi_0 A^+$. Note first that by choosing a surjection $\mathbb Z[X_1,\ldots,X_n]\to R$, one has $R_\square = (R,\mathbb Z[X_1,\ldots,X_n])_\square$ as
\[
\prod_I R = \prod_I \mathbb Z[X_1,\ldots,X_n]\otimes_{\mathbb Z[X_1,\ldots,X_n]} R
\]
in this case (as $R$ is pseudocoherent, even perfect, over $\mathbb Z[X_1,\ldots,X_n]$); and thus also $(A,R)_\square = (A,\mathbb Z[X_1,\ldots,X_n])_\square$. If the image of $R$ in $A$ is integral over $R'$, then the map $R'[X_1,\ldots,X_n]\to A$ factors over a finite $R'$-algebra
\[
R'' = R'[X_1,\ldots,X_n] / (g_1,\ldots,g_n)\to A
\]
where $g_i$ are monic polynomials in $X_i$ with coefficients in $R'$ (expressing that the image of $X_i$ is integral over $R'$). As there is an evident map $(A,R)_\square = (A,\mathbb Z[X_1,\ldots,X_n])_\square\to (A,R'')_\square$, it suffices to see that the map $(A,R')_\square\to (A,R'')_\square$ is an isomorphism, for which it suffices that $(R'',R')_\square\to R''_\square$ is an isomorphism. But this follows from
\[
\prod_I R'\otimes_{R'} R''\to \prod_I R''
\]
being an isomorphism, which holds true as $R''$ is a finitely generated (thus, pseudocoherent) $R'$-module.
\end{proof}

\begin{definition} An (animated) discrete Huber pair is a pair $(A,A^+)$ where $A$ is an animated commutative ring and $A^+\subset \pi_0 A$ is an integrally closed subring.
\end{definition}

For a discrete Huber pair $(A,A^+)$, define
\[
(A,A^+)_\square = \mathrm{colim}_{R\subset \pi_0 A^+} (A,R)_\square
\]
where the filtered colimit runs over finitely generated subalgebras of $A^+$. (A priori, this is not well-defined as there is no map $R\to A$ to define $(A,R)_\square$; but by the previous theorem, we can define it to be $(A,\tilde{R})_\square$ for any free $\mathbb Z$-algebra $\tilde{R}$ mapping to $A$ with image $R\subset \pi_0 A$.) This still defines an analytic ring.

\begin{theorem} The functor
\[
(A,A^+)\mapsto (A,A^+)_\square
\]
from (animated) discrete Huber pairs to analytic rings is fully faithful.
\end{theorem}

\begin{warning} Unless $A$ is finitely generated over $\mathbb Z$, the free modules $(A,A)_\square[S]$ of the analytic ring $(A,A)_\square$ are not given by
\[
\varprojlim_i A[S_i];
\]
rather, they are the filtered colimit
\[
\mathrm{colim}_{R\subset A} \varprojlim_i R[S_i]
\]
as $R$ runs over finitely generated subalgebras of $A$. This may seem undesired, but actually the fully faithfulness uses this structure. In some situations, for example for $A=\mathbb Q$, there is also the ``ultrasolid'' analytic ring structure $\mathbb Q_{u\square}$ with $\mathbb Q_{u\square}[S] = \varprojlim_i \mathbb Q[S_i]$.
\end{warning}

\begin{proof} On both source and target, the maps are a subanima of the maps between $A$'s. Thus, it suffices to see that if there is a map of analytic rings $(A,A^+)_\square\to (B,B^+)_\square$, then $A^+$ maps into $B^+$. This can be checked elementwise, so we can assume $(A,A^+)=(\mathbb Z[T],\mathbb Z[T])$. Now $(\mathbb Z[T],\mathbb Z[T])_\square$ is obtained from $(\mathbb Z[T],\mathbb Z)$ by killing the compact idempotent algebra $\mathbb Z((T^{-1}))$, cf.~\cite[Lecture VIII]{Condensed}. By compactness, if this idempotent algebra becomes zero in $(B,B^+)_\square$, it already becomes zero in some $(B,R)_\square$ for some finitely generated $\mathbb Z$-algebra $R$ mapping to $\pi_0 B^+$. Now the idempotent algebra $\mathbb Z((T^{-1}))$ has the projective resolution
\[
0\to \prod_{\mathbb N} \mathbb Z\otimes \mathbb Z[T]\xrightarrow{1-\mathrm{shift}\otimes T} \prod_{\mathbb N} \mathbb Z\otimes \mathbb Z[T]\to \mathbb Z((T^{-1}))\to 0
\]
in $(\mathbb Z[T],\mathbb Z)_\square$-modules. Taking the base change to $(B,R)_\square$, and using that $\mathbb Z((T^{-1}))$ becomes zero after the base change, this means that the map
\[
\prod_{\mathbb N} R\otimes_R B\xrightarrow{1-\mathrm{shift}\otimes g} \prod_{\mathbb N} R\otimes_R B
\]
is an isomorphism, where $g\in B$ is the image of $T$. We want to see that this implies that $g$ is integral over $R$ -- then $g$ does in fact lie in $B^+$. We can assume that $B$ is static.

Consider the element
\[
(1,0,\ldots,0,\ldots)\in \prod_{\mathbb N} R\otimes_R B.
\]
By assumption, this is in the image of $1-\mathrm{shift}\otimes g$. But
\[
\prod_{\mathbb N} R\otimes_R B = \mathrm{colim}_{M\subset B} \prod_{\mathbb N} R\otimes_R M = \mathrm{colim}_{M\subset B} \prod_{\mathbb N} M
\]
as $M$ runs over finitely generated $R$-submodules of $B$. Thus, there is a finitely generated submodule $M\subset B$ and a sequence $(b_0,b_1,\ldots)$ in $\prod_{\mathbb N} M$ such that
\[
b_0 = 1,\ b_1 - gb_0 = 0,\ b_2 - gb_1 = 0,\ \ldots\ .
\]
Thus, $b_i = g^i$, all of which are contained in the finitely generated $R$-submodule $M$ of $B$; therefore $g$ is integral over $R$, as desired.
\end{proof}

Huber associates to the pair $(A,A^+)$ (or rather its static truncation $(\pi_0 A,A^+)$) the adic space
\[
\mathrm{Spa}(A,A^+)
\]
consisting of equivalence classes of valuations $|\cdot|: A\to \Gamma\cup \{0\}$ with $|f|\leq 1$ for all $f\in A^+$.\footnote{Again, this definition makes sense a priori for general subsets of $A$, but only depends on the integrally closed subring generated the subset.} We usually write $x\in \mathrm{Spa}(A,A^+)$ and denote the corresponding valuation by $f\mapsto |f(x)|$. Its topology is generated by rational subsets, which are given by elements $f,g_1,\ldots,g_n\in A$ as
\[
U\left(\frac{g_1,\ldots,g_n}{f}\right) = \{x\in \mathrm{Spa}(A,A^+)\mid \forall i: |g_i(x)|\leq |f(x)|\neq 0\}.
\]
This is homeomorphic to
\[
\mathrm{Spa}(A[\tfrac 1f],A^+[\tfrac{g_1}f,\ldots,\tfrac{g_n}f]).
\]
This lets one define a sheaf of (animated) Huber pairs over $\mathrm{Spa}(A,A^+)$. Moreover, each point induces a valuation on the stalk of the structure sheaf.

\begin{definition}\label{def:derivedadicspace} A derived discrete adic space is a triple $(X,\mathcal O_X,(|\cdot(x)|)_x)$ consisting of a topological space $X$, a sheaf of animated commutative rings $\mathcal O_X$, and a valuation $|\cdot(x)|$ on $\mathcal O_{X,x}$ for each $x\in X$, such that $(X,\mathcal O_X,(|\cdot(x)|)_x)$ admits an open cover by subspaces isomorphic to affine adic spaces
\[
\mathrm{Spa}(A,A^+):=(\mathrm{Spa}(A,A^+),\mathcal O_{\mathrm{Spa}(A,A^+)},(|\cdot(x)|)_{x\in \mathrm{Spa}(A,A^+)}).
\]
\end{definition}

Fix some base ring $k$. There are now two functors from schemes of finite type over $k$ towards adic spaces: On the one hand, we can send $\mathrm{Spec}(A)$ to $\mathrm{Spa}(A,A)$; on the other, we can send it to $\mathrm{Spa}(A,k)$. Both extend to general schemes by gluing. Let us denote the first functor by $X\mapsto X^{\mathrm{ad}}$, and the second by $X\mapsto X^{\mathrm{ad}/k}$. There is a natural comparison map
\[
X^{\mathrm{ad}}\to X^{\mathrm{ad}/k}.
\]

\begin{proposition}\label{prop:describeadicspaces}\leavevmode
\begin{enumerate}
\item[{\rm (1)}] The underlying set of $X^{\mathrm{ad}}$ is the set of equivalence classes of valuation rings $V$ over $k$ with a map $\mathrm{Spec}(V)\to X$, where $\mathrm{Spec}(V)\to X$ and $\mathrm{Spec}(V')\to X$ are considered equivalent if there is some valuation ring $V''$ over $k$ and a commutative diagram
\[\xymatrix{
& \mathrm{Spec}(V)\ar[dr]\\
\mathrm{Spec}(V'')\ar[ur]\ar[dr]\ar[rr] && X.\\
& \mathrm{Spec}(V')\ar[ur]
}\]
\item[{\rm (2)}] The underlying set of $X^{\mathrm{ad}/k}$ is the space of equivalence classes of valuation rings $V$ over $k$ with a map $\mathrm{Spec}(\mathrm{Frac}(V))\to X$; the equivalence relation is analogous to (1). The map $X^{\mathrm{ad}}\to X^{\mathrm{ad}/k}$ sends $\mathrm{Spec}(V)\to X$ to the restriction $(V,\mathrm{Spec}(\mathrm{Frac}(V))\to X)$.
\item[{\rm (3)}] The map $X^{\mathrm{ad}}\to X^{\mathrm{ad}/k}$ is open, and the structure sheaf $\mathcal O_{X^{\mathrm{ad}}}$ with its valuations is pulled back from $X^{\mathrm{ad}/k}$.
\end{enumerate}

In particular, $X$ is separated if and only if $X^{\mathrm{ad}}\to X^{\mathrm{ad}/k}$ is an open immersion of adic spaces, and $X$ is proper if and only if $X^{\mathrm{ad}}\to X^{\mathrm{ad}/k}$ is an isomorphism.
\end{proposition}

\begin{remark} Any proof of finiteness of coherent cohomology for proper maps has to use properness in some way. We use it here in the form of the isomorphism $X^{\mathrm{ad}}\to X^{\mathrm{ad}/k}$, which is a direct application of the valuative criterion of properness.
\end{remark}

\begin{proof} All assertions readily reduce to the affine case, where they are immediate from the definitions.
\end{proof}

Thus, the following theorem refines Theorem~\ref{thm:solid6functorsschemes}.

\begin{theorem}\label{thm:solid6functorsadicspaces} The functor $(A,A^+)\mapsto D((A,A^+)_\square)$ satisfies descent on $\mathrm{Spa}(A,A^+)$, globalizing to a functor $X\mapsto D_\square(X)$ of solid quasicoherent sheaves on discrete adic spaces.

Moreover, $X\mapsto D_\square(X)$ upgrades to a six-functor formalism for some class $E$ of morphisms stable under base change, diagonals, and compositions, satisfying the following properties.
\begin{enumerate}
\item[{\rm (i)}] Any map $j: \mathrm{Spa}(A,A'^+)\hookrightarrow \mathrm{Spa}(A,A^+)$, where $A'^+$ is the integral closure of a finitely generated $A^+$-algebra in $A$, lies in $E$ and is cohomologically \'etale; in particular, $j_!$ is given by the left adjoint of pullback $j^\ast$.
\item[{\rm (ii)}] Any map $f: \mathrm{Spa}(B,B^+)\to \mathrm{Spa}(A,A^+)$ where $B^+$ is the integral closure of the image of $A^+$ in $B$ lies in $E$, and $f_!$ is given by the right adjoint $f_\ast$ of pullback $f^\ast$.
\item[{\rm (iii)}] If $f: X\to Y$ is an open immersion of schemes, then $X^{\mathrm{ad}}\to Y^{\mathrm{ad}}$ is cohomologically \'etale.
\item[{\rm (iv)}] If $k$ is some base ring and $f: X\to Y$ is a map of schemes of finite type over $k$, then $f^{\mathrm{ad}/k}: X^{\mathrm{ad}/k}\to Y^{\mathrm{ad}/k}$ lies in $E$ and $f^{\mathrm{ad}/k}_!$ is the right adjoint of pullback. In particular, if $f$ is proper, this applies to $X^{\mathrm{ad}}\to Y^{\mathrm{ad}}$ (which is the pullback of $X^{\mathrm{ad}/k}\to Y^{\mathrm{ad}/k}$ in this case).
\end{enumerate}
\end{theorem}

There are two approaches to this theorem. The first is to develop a sufficient amount of geometry on adic spaces, including the theory of Huber's compactification, to apply the construction principle from Lecture IV. This is slightly confusing, as the class $I$ is not the class of all open immersions of adic spaces, but only of some of them -- algebraically inverting an element is not an open immersion, only adjoining it to $A^+$ is.

The other approach is to construct the six-functor formalism directly on analytic rings and analytic stacks, and import it. As the six-functor formalism on analytic stacks will be useful later, we take this approach.

Thus, let $C$ be the opposite of the category of analytic rings.

\begin{definition}\leavevmode
\begin{enumerate}
\item[{\rm (1)}] Let $I$ be the class of maps of analytic rings $j: A\to B$ for which $j^\ast: D(A)\to D(B)$ has a fully faithful $D(A)$-linear left adjoint $j_!$; we also call these open immersions.
\item[{\rm (2)}] Let $P$ be the class of maps of analytic rings $f: A\to B$ for which $B$ has the induced analytic ring structure; equivalently, $f^\ast: D(A)\to D(B)$ has a $D(A)$-linear right adjoint $f_\ast$; we also call these proper.
\item[{\rm (3)}] Let $E$ be the class of maps of analytic rings $f: A\to B$ for which the induced map $B^\triangleright\otimes_{A^\triangleright} A\to B$ lies in $I$; we also call these $!$-able.
\end{enumerate}
\end{definition}

The prototypical example for (3) is $\mathbb Z_\square\to \mathbb Z[T]_\square = (\mathbb Z[T],\mathbb Z[T])_\square$: This factors as
\[
\mathbb Z_\square\to (\mathbb Z[T],\mathbb Z)_\square\to (\mathbb Z[T]_\square,\mathbb Z[T]_\square)
\]
where the first map lies in $P$ and the second in $I$.

It is easy to see that these classes of morphisms satisfy all the assumptions of Lecture IV. In fact, the ``compactification'' of morphisms in $E$ that is usually the hardest comes for free in this formalism. Thus, we get a six-functor formalism on $(C,E)$, taking $A$ to $D(A)$. Using Theorem~\ref{thm:stacky6functors}, we can then extend the formalism to analytic stacks $\mathrm{AnStack}$, for some (much) larger class of maps $\tilde{E}$.\footnote{A minor tweak to this story was introduced in \cite{AnStack}: Namely, one can allow certain $D$-hypercovers as well. More precisely, we invert all $\infty$-connective maps $f: Y\to X=\mathrm{AnSpec}(A)$ that can be written as a countable colimit of $!$-able $\mathrm{AnSpec}(B_i)\to \mathrm{AnSpec}(A)$, and such that $f$ and all of its diagonals satisfy universal $!$-descent (i.e.~the map $\mathrm{colim}_{i,!} D(B_i)\to D(A)$ is an equivalence, and similarly for all diagonals). In principle, a similar tweak can be done for general $6$-functor formalisms.}

The functor $(A,A^+)\mapsto (A,A^+)_\square$ takes open covers of adic spectra to covers in the $D$-topology on analytic rings. Indeed, one can reduce to standard rational covers by \cite[Lemma 2.6]{HuberGeneralization}. Then all terms in the cover go to open immersions of analytic rings, so by Proposition~\ref{prop:Dsuavedescent}, it suffices to see that pullback is conservative. This is \cite[Lemma 10.3]{Condensed}. Thus, by gluing we get the desired functor from discrete adic spaces to analytic stacks, and we can pull back the six-functor formalism.

\begin{proof}[Proof of Theorem~\ref{thm:solid6functorsadicspaces}] Properties (i) and (ii) are obvious from the construction, while (iii) follows from (i) by gluing, and similarly (iv) follows from (ii) by gluing.
\end{proof}

This finishes the construction of the solid $6$-functor formalism. Finally, this gives us a formalism where smooth maps are cohomologically smooth:

\begin{theorem} Let $f: X\to Y$ be a smooth map of derived schemes. Then $f$ is cohomologically smooth in the $D_\square(X)=D_\square(X^{\mathrm{ad}})$-formalism, with dualizing complex $\Omega^d_{X/Y}[d]$.

More generally, if $f$ has finite Tor-dimension, then $f$ is suave. If $f$ is a local complete intersection, then $f^! 1$ is invertible (thus, $f$ is cohomologically smooth), and isomorphic to $\mathrm{det}(L_{X/Y})$.
\end{theorem}

\begin{proof} First, one checks that $\mathbb A^1_{\mathbb Z}\to \mathrm{Spec}(\mathbb Z)$ is cohomologically smooth. This is done in detail in \cite[Lecture VIII]{Condensed}. Using Theorem~\ref{thm:critcohomsmooth}, it is in fact enough to construct a trace map $f_! 1[1]\to 1$ which, together with the a map $\Delta_! 1 = \Delta_\ast 1\to 1[1]$ satisfies the assumptions there. But $f_! 1$ can be computed using the compactification
\[
\mathrm{Spa}(\mathbb Z[T],\mathbb Z[T])\to \mathrm{Spa}(\mathbb Z[T],\mathbb Z)\to \mathrm{Spa}(\mathbb Z,\mathbb Z)
\]
giving
\[
\mathrm{fib}(\mathbb Z[T]\to \mathbb Z((T^{-1}))).
\]
In particular $f_! 1$ is given by $T^{-1} \mathbb Z[[T^{-1}]] [-1]$, and the trace map $f_! 1[1]\to 1$ can be defined by projection to the coefficient of $T^{-1}$.

Now in general, we can factor $f$ has a closed immersion into affine space, and projection from affine space. The latter case has been handled, and for closed immersions, one has $f_! = f_\ast$ and $f^!$ is given by $\mathrm{Hom}_{\mathcal O_Y}(\mathcal O_X,-)$. If $f$ has finite Tor-dimension, this is a linear functor, as desired, yielding suaveness. If moreover $f$ is a local complete intersection, then $f^!(1)$ can also be computed to be invertible, by the Koszul resolution. A deformation to the normal cone identifies the dualizing complex, cf.~\cite[Lecture XIII]{ClausenScholzeComplex}.
\end{proof}

Let us use this formalism to prove finiteness of coherent cohomology; more precisely, Theorem~\ref{thm:finitenesscoherentcohom}, which we restate here for convenience:

\begin{theorem} Let $f: X\to Y$ be a proper map of derived schemes. Then $f_\ast$ takes pseudocoherent objects of $D_{\mathrm{qc}}(X)$ to pseudocoherent objects of $D_{\mathrm{qc}}(Y)$, and coherent objects to coherent objects. If $f$ is of finite Tor-dimension, then $f_\ast$ takes $\mathrm{Perf}(X)$ to $\mathrm{Perf}(Y)$.
\end{theorem}

\begin{proof} On discrete modules, the functor $f_\ast$ agrees with the same functor on $D_\square$; this can be checked by reduction to the affine case. By properness, $f_\ast = f_!$. It thus suffices to prove the more general statement that for any separated map of derived schemes, embedded into discrete adic spaces via $\mathrm{Spec}(A)\mapsto \mathrm{Spa}(A,A)$, the functor $f_!$ takes (pseudo)coherent objects of $D_\square(X)$ to (pseudo)coherent objects of $D_\square(Y)$, and when $f$ is of finite Tor-dimension, also compact objects to compact objects.

Via gluing, this statement immediately reduces to the affine case $X=\mathrm{Spec}(A)\to Y=\mathrm{Spec}(B)$. We can then factor over a closed immersion into affine space. It is clear that closed immersions have all the desired properties (where finite Tor-dimension is required to get a finite resolution of the ideal sheaf). Thus, we can reduce to affine space, and inductively to $\mathbb A^1_Y\to Y$. By approximation, it suffices to see that $f_!$ preserves compact objects, for which it suffices to see that $f^!$ preserves colimits, or better that $f$ is cohomologically smooth. By base change, this follows from the case of $\mathbb A^1_{\mathbb Z}\to \mathrm{Spec}(\mathbb Z)$, where it was verified in detail in \cite[Lecture VIII]{Condensed}.
\end{proof}

\newpage

\section{Lecture X: Ring Stacks}

In the rest of these lectures, we will be interested in six-functor formalisms defined for schemes (separated of finite type over a field, say). As mentioned in the introductory lecture, these can very often be factored as a composite
\[
\mathrm{Sch}\to \mathrm{AnStack}\xrightarrow{D_{\mathrm{qc}}} \mathrm{Pr}^L
\]
via some functor $X\mapsto X^?$ from schemes to analytic stacks (called a ``transmutation'', following Bhatt \cite[Remark 2.3.8]{BhattTransmutation}). This approach has the virtue of black-boxing all the category theory to the theory of analytic stacks, and reducing the construction of six-functor formalisms to a geometric construction $X\mapsto X^?$. In fact, even better, we will see that it suffices to treat the case $X=\mathbb A^1$, reducing us to the construction of the ring stack $(\mathbb A^1)^?$.

We have already seen two examples. In both of these examples, we actually do not need analysis: The functor takes values in usual stacks, i.e.~stacks on (the opposite of) the category of commutative rings, endowed with the six-functor formalism where all maps are proper, and the resulting $D$-topology. This is actually the descendable topology of Mathew \cite{MathewDescendable}.

\begin{example}[The Betti stack] For a finite-dimensional locally compact Hausdorff topological space $X$ (for example, the complex points of a scheme over $k=\mathbb C$) one can define its Betti stack
\[
X_{\mathrm{Betti}}: S\mapsto \mathrm{Cont}(|S|,X).
\]
This is actually an algebraic stack. If $X$ is profinite, then this is an affine scheme
\[
X_{\mathrm{Betti}} = \mathrm{Spec}(\mathrm{Cont}(X,\mathbb Z)).
\]
In general, we can find a surjective map $\tilde{X}=\bigsqcup_i \tilde{X}_i\to X$ where all $\tilde{X}_i$ are profinite. The induced equivalence relation $R = \tilde{X}\times_X \tilde{X} = \bigsqcup_{i,j} \tilde{X}_i\times_X \tilde{X}_j$ is then similarly a disjoint union of profinite sets, with a pro-\'etale map to $\tilde{X}$, and
\[
X_{\mathrm{Betti}} = \tilde{X}_{\mathrm{Betti}} / R_{\mathrm{Betti}}
\]
is accordingly a ``pro-\'etale algebraic space'', the quotient of a scheme by a pro-\'etale equivalence relation.

For profinite $X$, we have
\[
D_{\mathrm{qc}}(X_{\mathrm{Betti}}) = D(\mathrm{Cont}(X,\mathbb Z))=D(X,\mathbb Z),
\]
the derived category of abelian sheaves on $X$ (where the equivalence with $\mathrm{Cont}(X,\mathbb Z)$ is realized by taking global sections). In general, descent implies
\[
D_{\mathrm{qc}}(X_{\mathrm{Betti}})\cong D(X,\mathbb Z),
\]
giving the desired factorization. Moreover, this identification is compatible with the six functors.
\end{example}

\begin{example}[The de Rham stack] For $k$ a field of characteristic $0$, we have the functor $X\mapsto X_{\mathrm{dR}}$ studied in the appendix to Lecture VIII, and
\[
\mathrm{Dmod}(X) = D_{\mathrm{qc}}(X_{\mathrm{dR}}).
\]
Different authors use slightly different conventions regarding the six functors on $\mathrm{Dmod}(X)$; up to these shift conventions, the six-functor formalism imported from $D_{\mathrm{qc}}(X_{\mathrm{dR}})$ agrees with the one on $\mathrm{Dmod}(X)$. (This is true even though \'etale maps are cohomologically proper and proper maps are cohomologically smooth in this formalism: These properties are rather reflected in slightly strange geometric properties of $X_{\mathrm{dR}}$.)
\end{example}

In these examples, and in fact all examples known to the author, the functor $X\mapsto X^?$ commutes with all finite limits, and is compatible with Zariski gluing. This has the consequence that $X^?$ is determined by its values for affine $X$; and writing some affine $X$ of finite type over $k$ as a fibre product
\[\xymatrix{
X\ar[r]\ar[d] & \mathbb A^n\ar[d]\\
0\ar[r] & \mathbb A^m
}\]
for some polynomial map $\mathbb A^n\to \mathbb A^m$, we see that $X^?$ is given by the fibre product
\[\xymatrix{
X^?\ar[r]\ar[d] & [(\mathbb A^1)^?]^n\ar[d]\\
0\ar[r] & [(\mathbb A^1)^?]^m.
}\]
Suitably formalized, this means the whole theory is determined by the stack $(\mathbb A^1)^?$. However, one must remember not the abstract stack: Rather, $\mathbb A^1$ has the structure of a $k$-algebra object in schemes over $k$, and by transmutation this induces the structure of a $k$-algebra stack on $(\mathbb A^1)^?$. These are the ``ring stacks'' of this lecture's title.

The idea of ring stacks emerged in the theory of prismatic cohomology through the works of Drinfeld \cite{DrinfeldRingStack}, \cite{BhattLurieRingStack} on the stacky approach to prismatic cohomology. It makes it possible to give remarkably slick definitions of cohomology theories that otherwise take quite a bit of setup.

\begin{example}[Crystalline cohomology via ring stacks] Let $k=\mathbb F_p$. Consider the functor of $p$-typical Witt vectors $W$ defined on rings $R$ killed by some power of $p$ (i.e., over $\mathrm{Spf}(\mathbb Z_p)$). As abstractly $W(R)\cong \prod_{\mathbb N} R$ via Witt coordinates, this is representable by an affine scheme, an infinite-dimensional affine space; but it also has a ring structure, i.e.~$W$ is a ring scheme. Now consider
\[
(\mathbb A^1_{\mathbb F_p})^{\mathrm{crys}} := W/^{\mathbb L} p,
\]
the (derived) reduction of $W$ modulo $p$. This is an $\mathbb F_p$-algebra stack over $\mathrm{Spf}(\mathbb Z_p)$. The $p$-th power of the Teichm\"uller map $R\to W(R): x\mapsto [x]^p$ yields a map
\[
\mathbb A^1_{\mathrm{Spf} \mathbb Z_p}\to W
\]
and one can show that the induced map
\[
\mathbb A^1_{\mathrm{Spf} \mathbb Z_p}\to W/^{\mathbb L} p = (\mathbb A^1_{\mathbb F_p})^{\mathrm{crys}}
\]
is surjective. Moreover, this is a map of rings: It is clearly multiplicative, but in fact it is also additive as for $x,y\in R$ we have
\[
[x+y]^p - [x]^p - [y]^p = F([x+y]-[x]-[y]) = FV(z) = pz
\]
for some $z\in W(R)$ (functorially determined by $x$ and $y$), as the first Witt component of $[x+y]-[x]-[y]$ vanishes. Drinfeld shows that it induces an isomorphism of stacks
\[
(\mathbb A^1_{\mathbb F_p})^{\mathrm{crys}}\cong \mathbb A^1_{\mathrm{Spf} \mathbb Z_p} / \mathbb G_a^\sharp
\]
where $\mathbb G_a^\sharp = \mathrm{Spf} \mathbb Z_p[x]_{\mathrm{PD}}$ is the spectrum of the divided power algebra on a variable $x$ (so one adjoins also $\frac{x^n}{n!}$ for all $n\geq 1$).

Note that quasicoherent sheaves on
\[
\mathbb A^1_{\mathrm{Spf} \mathbb Z_p} / \mathbb G_a^\sharp
\]
precisely recover crystals on $\mathbb A^1_{\mathbb F_p}$ on the crystalline site; the divided powers in $\mathbb G_a^\sharp$ account for the divided power thickenings of diagonals.
\end{example}

\begin{example}[Prismatization] We have seen that taking the quotient of $W$ by $p$, we recover crystalline cohomology. It turns out that its generalization, prismatic cohomology, arises simply by replacing the element $p$ by some deformation of it. More precisely, if $A$ is some ring (still, say, killed by some power of $p$) and $\xi\in W(R)$ is some element that is primitive of degree $1$ (meaning that the zeroth Witt component is nilpotent and the first one is a unit), then we can consider the ring stack
\[
(\mathbb A^1)^{\Delta_A}\to \mathrm{Spec}(A)
\]
given by $R\mapsto W(R)/^{\mathbb L} \xi$. This has the structure of a $W(A)/^{\mathbb L} \xi$-algebra stack; in particular, a $\mathbb Z_p$-algebra stack, and induces a transmutation $X\mapsto X^{\Delta_A}$ on (formal) schemes $X$ over $\mathrm{Spf}(\mathbb Z_p)$. In fact, transmutation has an explicit functor of points:
\[
X^{\Delta_A}: R\mapsto X(W(R)/^{\mathbb L}\xi).
\]
Indeed, this is clear for $X=\mathbb A^1$, and is compatible with finite limits.

As a specific example, let $A=\mathbb Z_p[[q-1]]$ (and work over $\mathrm{Spf}(A)$). Then $A$ is a $\delta$-ring, i.e.~is equipped with a Frobenius lift, sending $q$ to $q^p$; and this induces a map $A\to W(A)$. Let $\xi$ be the image of the $q$-deformation $[p]_q = 1+q+\ldots+q^{p-1}$ in $W(A)$. This yields the ``$q$-de Rham prism''; the associated stack computes $q$-de Rham cohomology. In this case, one can compute the stack
\[
(\mathbb A^1)^{\Delta_A}\to \mathrm{Spf}(\mathbb Z_p[[q-1]])
\]
as the quotient of $\mathbb A^1_A$ by the equivalence relation corresponding to the ring
\[
\mathbb Z_p[[q-1]][X,Y,\left(\frac{(X-Y)(X-qY)\cdots(X-q^{n-1}Y)}{[n]_q!}\right)_n],
\]
a $q$-divided power thickening of the diagonal. Here $[n]_q! = \prod_{i=1}^n [i]_q$ is the $q$-factorial.
\end{example}

Let us now introduce the first genuinely analytic example.

\begin{example}[The solid de Rham stack]\footnote{This theory will also appear as part of work in progress by Aoki--Rodr\'iguez Camargo--Zavyalov.} We have seen that the six-functor formalism for quasicoherent cohomology becomes more intuitive when passing to solid modules; in that situation, smooth maps are cohomologically smooth. This idea can also be transported to the theory of $D$-modules and de Rham stacks. The idea is to replace
\[
\mathbb A^1_{\mathrm{dR}} = \mathbb A^1/\hat{\mathbb G}_a,
\]
the quotient of $\mathbb A^1$ by the formal additive group $\hat{\mathbb G}_a = \mathrm{Spf} k[[x]]$, with a similar quotient, but this time by the affine analytic stack $\mathrm{AnSpec}(k[[x]])$, where now $k[[x]]$ is considered as a solid module.

Let us execute this strategy when $k=\mathbb Q$, and we endow it with the ultrasolid analytic ring structure $\mathbb Q_{u\square}$ where the completion $\mathbb Q_{u\square}[S]$ of $\mathbb Q[S]$ is $\varprojlim_i \mathbb Q[S_i]$ (when $S=\varprojlim_i S_i$ is a profinite set). For $\mathbb Q$, and more generally finitely generated fields, this defines an analytic ring structure. For $\mathbb Q$, one kills the compact idempotent algebra $\hat{\mathbb Z}$ in solid $\mathbb Z$-modules to get the ultrasolid structure on $\mathbb Q$.

Now consider
\[
\mathbb A^{1,\square}_{\mathrm{dR}} := \mathrm{AnSpec}(\mathbb Q[T]_{u\square})/\mathrm{AnSpec}(\mathbb Q_{u\square}[[T]]).
\]
It is easy to see that $\mathrm{AnSpec}(\mathbb Q_{u\square}[[T]])\subset \mathrm{AnSpec}(\mathbb Q[T]_{u\square})$ is an ideal, and hence this quotient defines a ring stack. This globalizes to a functor $X\mapsto X^\square_{\mathrm{dR}}$ on separated schemes $X$ of finite type over $\mathbb Q$. Moreover, \'etale maps are mapped to cohomologically \'etale maps of analytic stacks; smooth maps are mapped to cohomologically smooth maps; and proper maps are mapped to cohomologically proper maps. To check this, we use Theorem~\ref{thm:conditionsonringstack} below.

For example, the diagram
\[
\ast\xrightarrow{i} \mathbb A^{1,\square}_{\mathrm{dR}}\xleftarrow{j} \mathbb G_{m,\mathrm{dR}}^\square,
\]
yields an excision sequence
\[
j_! 1\to 1\to i_\ast 1
\]
which, after pullback to $\mathrm{AnSpec}(\mathbb Q[T]_{u\square})$, realizes to the expected sequence
\[
(\mathbb Q[[T]]/\mathbb Q[T])[-1]\to \mathbb Q[T]\to \mathbb Q[[T]].
\]
\end{example}

The following theorem shows that whether a ring stack gives rise to a transmutation functor with good properties is easily checkable; it is closely related to Theorem~\ref{thm:aokithesis} from Aoki's thesis \cite{AokiThesis}.

\begin{theorem}\label{thm:conditionsonringstack} Work in the $\infty$-category $C$ of analytic stacks over some base analytic stack $S$. Let $k$ be some ring and let $R$ be a $k$-algebra stack in $C$. Assume the following conditions.
\begin{enumerate}
\item[{\rm (1)}] The map $f: R\to \ast$ is $!$-able and cohomologically smooth, and the map $1\to f_! f^! 1$ is an equivalence.
\item[{\rm (2)}] The map $i: \ast\xrightarrow{0} R$ is a closed immersion, the units $j: R^\times\hookrightarrow R$ yield an open immersion, and
\[
j_! 1\to 1\to i_\ast 1
\]
is a fibre sequence in $D_{\mathrm{qc}}(R)$.\footnote{Implicit here is that the composite is zero, which means that $R^\times$ and $0$ have empty intersection.}
\end{enumerate}
This induces a transmutation functor
\[
X\mapsto X_R
\]
taking any separated scheme $X$ of finite type over $k$ to the $0$-truncated analytic stack over $S$ given by the sheafification of the functor taking an analytic ring $A$ over $S$ to $X(R(A))$. The functor $X\mapsto X_R$ commutes with finite limits and Zariski gluing.

Moreover, the functor $X\mapsto X_R$ sends \'etale maps to cohomologically \'etale maps, proper maps to cohomologically proper maps, and smooth maps to cohomologically smooth maps, and thus $X\mapsto D_{\mathrm{qc}}(X_R)$ yields a six-functor formalism with these properties.
\end{theorem}

\begin{remark} The Betti stack $\mathbb C_{\mathrm{Betti}}$ (or also $\mathbb R_{\mathrm{Betti}}$) satisfies these conditions. However, the other examples of algebraic nature -- the algebraic de Rham stack, or the crystalline or prismatic ring stacks -- do not satisfy these conditions. But there are many interesting examples of genuinely analytic nature, listed below.
\end{remark}

\begin{remark} It is somewhat surprising how few conditions must be imposed; a priori one might imagine that the conditions that \'etale maps go to cohomologically \'etale maps, and proper maps go to cohomologically proper maps, must be imposed in addition. As these are conditions ranging over arbitrary maps of schemes, this would be a subtle condition to check. Fortunately, it turns out that both conditions are automatic, by slightly nontrivial diagram chases with six functors.
\end{remark}

\begin{proof} Principal open immersions $\mathrm{Spec}(B[\tfrac 1f])\subset \mathrm{Spec}(B)$ go to open immersions of analytic stacks, via base change from the case of $\mathbb G_m\subset \mathbb A^1$. Moreover, a cover by principal opens goes to a cover in analytic stacks. Indeed, by Proposition~\ref{prop:Dsuavedescent}, it suffices to see that pullback is conservative, and by (2), this can be checked on a stratification. But a stratification lifts to an open cover.

From here, it formally follows that $X\mapsto X_R$ is in general compatible with gluing, and that open immersions go to open immersions.

Next, we show that \'etale maps go to cohomologically \'etale maps. Working Zariski locally, we can assume the \'etale map is of the standard \'etale form $e: \mathrm{Spec}(B[T]/g(T)[\tfrac 1J])\to \mathrm{Spec}(B)$, where $g$ is some monic polynomial, and $J$ is the discriminant. By base change, we can assume $k=B$. The map $e$ is the base change of an \'etale map
\[
g|_U: U\subset \mathbb A^1\xrightarrow{g} \mathbb A^1
\]
from an open subset of $U$ of $\mathbb A^1$ towards $\mathbb A^1$. The diagonal of $g|_U$ is an open immersion and hence cohomologically \'etale. By Proposition~\ref{prop:checkcohometale}, it suffices to see that a natural map $(g|_U)_R^! 1\to 1$ is an isomorphism in $D_{\mathrm{qc}}(U_R)$. Note that as both source and target of $g|_U$ are cohomologically smooth, it follows that $(g|_U)_R^! 1$ (or even $g_R^! 1$) is invertible. In fact, we can see that it is even trivial:

Note first that the dualizing complex of $R\to \ast$ is in the essential image of the fully faithful $\ast$-pullback functor $D_{\mathrm{qc}}(S)\to D_{\mathrm{qc}}(R)$. Indeed, it is the pullback of the dualizing complex of $\ast\to \ast/R$, the map to the classifying space of the additive group $R$. Let $L\in D_{\mathrm{qc}}(S)$ denote the resulting object (i.e., the Tate twist). Then $g_R^! 1$ is isomorphic to $L|_U^{\otimes -1}\otimes L|_U\cong 1$.

By smooth base change along $U\to \ast$, we can arrange that there is a section $s: \ast\to U$ (covering the whole original $U$), and it suffices to see that the map $(g|_U)_R^! 1\to 1$ becomes an isomorphism after applying $s^\ast$.

Thus, it suffices to see that for any map $g: \mathbb A^1\to \mathbb A^1$ which is \'etale over $U\subset \mathbb A^1$, and any section $s: \ast\to U$, the map $(g|_U)_R^! 1\to 1$ becomes an isomorphism after applying $s^\ast$. By shifting, we can assume $s$ is the zero section, and maps under $g$ to the zero section. Let
\[
g(T)=a_1T+a_2T^2+\ldots+a_nT^n.
\]
We can put this in the family of polynomials
\[
g_b(T) = a_1T + b(a_2T^2+\ldots+a_nT^n)
\]
with a parameter $b\in \mathbb A^1$. This family gives a map
\[
\mathbb A^2\to \mathbb A^2: (T,b)\mapsto (g_b(T),b)
\]
which is \'etale over some open subset $U_b\subset \mathbb A^2$ containing $0\times \mathbb A^1$. We can then do the whole construction in the family over $\mathbb A^1$ with parameter $b$, giving a map $1\to 1$ over $\mathbb A^1 = 0\times \mathbb A^1\subset \mathbb A^2$. By $\mathbb A^1$-invariance, this map is constant. At $b=1$, it is the map we are trying to identify; and at $b=0$ the map $g_0$ is multiplication by the invertible number $a_1$, and the claim is clear.

Now that we know that \'etale maps are cohomologically smooth, it follows that smooth maps are cohomologically smooth, by writing them locally as \'etale over affine space. It remains to see that proper maps are cohomologically proper. For this, we make use of some results of Ayoub. More specifically, we note $X\mapsto D_{\mathrm{qc}}(X_R)$ defines a stable homotopical $2$-functor in the sense of Ayoub \cite[Definition 1.4.1]{AyoubSixFunctors}; and he shows that this implies proper base change \cite[Th\'eor\`eme 1.7.9]{AyoubSixFunctors}, at least for projective spaces. But Chow's theorem lets us reduce to this case.
\end{proof}

\begin{example}[The analytic de Rham stack] Let $K$ be a Banach field of characteristic $0$. If $K$ is nonarchimedean, we endow it with the solid ring structure (say, induced from $\mathbb Z$); if $K$ is archimedean, we endow it with the liquid analytic ring structure (for some choice of $0<\alpha\leq 1$).\footnote{Other choices are possible -- we could also always use a liquid analytic ring structure, or even the gaseous one.} We can then define the overconvergent neighborhood of $0$ in the affine line
\[
\mathbb G_a^\dagger = \mathrm{AnSpec} \{\sum_{n\geq 0} a_n T^n\mid (a_n)_n\ \mathrm{has\ at\ most\ exponential\ growth}\}.
\]
This is the ring of functions converging on small disk, and its spectrum is the intersection of all disks in $\mathbb A^1$ containing $0$. It turns out that
\[
\mathbb G_a^\dagger\subset \mathbb A^1_K
\]
is a closed subspace, but it is not an ideal. This is rectified by passing to the analytic $\mathbb A^1$, given as the union of all disks (of finite radius)
\[
\mathbb A^{1,\mathrm{an}}_K = \bigcup_{R<\infty} \mathbb D(0,R)_K
\]
where $\mathbb D(0,R)_K$ is the affine analytic stack corresponding to the overconvergent functions on the disk of radius $R$; explicitly, these are those sums $\sum_{n\geq 0} a_n T^n$ such that for some $R'>R$, $|a_n| R'^n\to 0$.

Then the analytic de Rham ring stack is
\[
\mathbb A^{1,\mathrm{an}}_{\mathrm{dR},K} = \mathbb A^{1,\mathrm{an}}_K / \mathbb G_a^\dagger.
\]
One can show that its quasicoherent sheaves are some type of analytic $D$-modules. Moreover, one again checks that open immersions go to open immersions, smooth maps to cohomologically smooth maps, proper maps to cohomologically proper maps, etc.

If $K=\mathbb C$, it turns out that
\[
\mathbb A^{1,\mathrm{an}}_{\mathrm{dR},\mathbb C} = \mathbb C_{\mathrm{Betti}}\times_{\mathrm{Spec}(\mathbb Z)} \mathrm{AnSpec}(\mathbb C)
\]
is just a base change of the Betti stack from the first example! This is the analytic Riemann--Hilbert correspondence of \cite{ScholzeRealLLC}. Indeed, $\mathbb A^{1,\mathrm{an}}_{\mathbb C}$ is the usual complex-analytic affine line, and this maps to its underlying topological space, yielding a map to $\mathbb C_{\mathrm{Betti}}$ of ring stacks. The fibre is precisely the subspace that maps to $0$, which is just $\mathbb G_a^\dagger$. Again, via transmutation, a similar statement holds for all schemes of finite type over $\mathbb C$ (or, by generalizing the argument directly, for all complex manifolds, see \cite{ScholzeRealLLC}).
\end{example}

\begin{example}[Analytic de Rham stacks of (un)tilts]\label{ex:deRhamofuntilt} Generalizing the previous example, assume that $K$ is a perfectoid Banach field, so in particular $|p|<1$ for a (unique) prime $p$; and let $K'$ be a second perfectoid field, of characteristic $0$, with an isomorphism $K'^\flat\cong K^\flat$. The rigid-analytic affine line $\mathbb A^{1,\mathrm{an}}_K$ can then be tilted, and untilted back to $K'$, as a diamond, yielding a diamond $(\mathbb A^{1,\mathrm{an}}_K)_{K'}$ over $\mathrm{Spd}(K')$. By the theory of \cite{AnalyticdeRhamFF}, any diamond over $\mathrm{Spd}(\mathbb Q_p)$ admits an analytic de Rham stack, yielding a ring stack
\[
(\mathbb A^{1,\mathrm{an}}_K)_{K',\mathrm{dR}}\to \mathrm{AnSpec}(K').
\]
This is a $K$-algebra stack, having the same properties as before, in particular yielding a transmutation $X\mapsto X_{K',\mathrm{dR}}$ from schemes of finite type $X$ over $K$ towards analytic stacks over $K'$. (This can also be defined by directly generalizing the given construction.) Note that this can be applied in particular in the case $K=\mathbb F_p((t^{1/p^\infty}))$ to yield a theory for varieties over $\mathbb F_p$; this is closely related to the theory of overconvergent isocrystals.

One can also fix $K$ and vary $K'$; this yields an interpolation of these theories over the Fargues--Fontaine curve $\mathrm{FF}_K$, or even over its analytic de Rham stack; cf.~\cite{AnalyticdeRhamFF}. In particular, this yields a six-functor formalism surrounding the de Rham--Fargues--Fontaine cohomology theory conjectured to exist in \cite[Conjecture 6.4]{ScholzeICM} and constructed by le Bras--Vezzani \cite{leBrasVezzani}.
\end{example}

\newpage

\section{Lecture XI: Motivic sheaves}

The previous lecture dealt with various explicit cohomology theories, all constructed using the stacky approach. At the other extreme, there is the motivic six-functor formalism, which can in fact be characterized as a universal example of a six-functor formalism on schemes satisfying certain simple axioms. Let us recall this theorem of Drew--Gallauer \cite{DrewGallauer}.

Let $C$ be the category of separated schemes of finite type over a noetherian base ring $k$ with $\mathrm{Spec}(k)$ of finite Krull dimension, with the usual classes $I$ of open immersions and $P$ of proper maps (so $E$ consists of all maps). We will be interested in presentable six-functor formalisms
\[
\mathrm{Corr}(C)\to \mathrm{Pr}^L
\]
for which the maps in $I$ are cohomologically \'etale and the maps in $P$ cohomologically proper. By the theorem of Dauser--Kuijper, Theorem~\ref{thm:dauserkuijper}, these are equivalent to functors
\[
C^{\mathrm{op}}\to \mathrm{CAlg}(\mathrm{Pr}^L)
\]
with the appropriate hypotheses on $I$ and $P$. Also, we will always assume that all $D(X)$ are stable.

\begin{theorem}[Drew--Gallauer, \cite{DrewGallauer}] Consider the $\infty$-category of functors
\[
D: C^{\mathrm{op}}\to \mathrm{CAlg}(\mathrm{Pr}^L_{\mathrm{st}})
\]
satisfying the following properties.
\begin{enumerate}
\item[{\rm (0)}] Pullbacks along open immersions admit left adjoints and pullbacks along proper maps admit right adjoints satisfying projection formula and base change; in particular, $D$ yields a six-functor formalism.
\item[{\rm (1)}] All smooth maps are cohomologically smooth.
\item[{\rm (2)}] Excision holds: If $X$ has a closed subset $i: Z\subset X$ with open complement $j: U\to X$, then
\[
j_! 1\to 1\to i_\ast 1
\]
is exact in $D(X)$.
\item[{\rm (3)}] The homology of $\mathbb A^1$ is trivial, i.e.~for $\pi: \mathbb A^1\to \ast$ the adjunction map
\[
1\to \pi_! \pi^! 1
\]
is an isomorphism.
\end{enumerate}
This $\infty$-category contains an initial object, which is the motivic $6$-functor formalism $SH_k$ of Morel--Voevodsky \cite{MorelVoevodsky}.
\end{theorem}

\begin{remark} Drew--Gallauer prove a slightly different statement: Namely, they do not a priori ask that proper maps are cohomologically proper. Condition (1) must then be rephrased as our notion of cohomologically smooth presumes a six-functor formalism. Indeed, Drew--Gallauer directly ask that pullbacks along smooth maps admit left adjoints satisfying base change and projection formula. This does not imply cohomological smoothness -- one still needs the invertibility of the dualizing complex. This they enforce using their last condition (4b): For the composite
\[
\ast\xrightarrow{s} \mathbb A^1\xrightarrow{p} \ast,
\]
the object $p_\sharp s_* 1\in D(\ast)$ is invertible. But this is automatic if we assume cohomological smoothness of smooth maps: Then $p_\sharp$ agrees with $p_!$ up to twist, and in any case $s_\ast=s_!$, so up to twist this agrees with $p_!s_! 1 = 1$.

Also, in (2), they ask for an apparently stronger condition that
\[
D(U)\xrightarrow{j_!} D(X)\xrightarrow{i^\ast} D(Z)
\]
is a Verdier sequence (or equivalently the one for the right adjoints), but this follows from our condition and the projection formula (which we already assumed for $i_\ast$). Similarly in (3) they make a statement about categories, but this follows again from the statement on the unit and the projection formula.
\end{remark}

We sketch the proof of the theorem, which is actually constructive and lets one recover the construction of the motivic $6$-functor formalism.

First, we enforce the existence of left adjoints $f_\sharp$ to smooth pullbacks $f^\ast$, satisfying base change and projection of formula. In this case, any $D(X)$ acquires objects $f_\sharp 1$ for smooth $f: Y\to X$, and in fact one can show that the free example is given by taking $X$ to the stable presheaf category on $\mathrm{Sm}_{/X}$.

Next, it turns out that excision implies Nisnevich descent. Indeed, Nisnevich maps are cohomologically smooth, so this follows from Proposition~\ref{prop:Dsuavedescent} and the excision axiom (which ensures that to check conservativity, we can check over a stratification; and any Nisnevich cover splits over a stratification). Thus, the next approximation is the functor taking any $X$ to $D_{\mathrm{Nis}}(\mathrm{Sm}_{/X},\mathbb S)$, the stable $\infty$-category of Nisnevich sheaves of spectra on $\mathrm{Sm}_{/X}$.

We need to enforce that the homology of $\mathbb A^1$ is trivial; this forces us to pass to $\mathbb A^1$-invariant Nisnevich sheaves $D_{\mathrm{Nis}}^{\mathbb A^1}(\mathrm{Sm}_{/X},\mathbb S)$. Finally, one has to ensure that for
\[
\ast\xrightarrow{s} \mathbb A^1\xrightarrow{p} \ast,
\]
the object $T=p_\sharp s_* 1\in D(\ast)$ is invertible. Again, as developed by Robalo \cite{RobaloSymMon}, there is a universal way to make objects $\otimes$-invertible; in this situation the object turns out to be symmetric, so one can just pass to spectrum objects. This yields the functor
\[
X\mapsto \mathrm{SH}(X) = \mathrm{Sp}_T(D_{\mathrm{Nis}}^{\mathbb A^1}(\mathrm{Sm}_{/X},\mathbb S)),
\]
of Morel--Voevodsky, and this is known to satisfy all conditions imposed.

The theorem of Drew--Gallauer makes it very easy to construct realizations of motives -- any usual six-functor formalism satisfies their axioms, thus yielding a map of six-functor formalisms from $\mathrm{SH}$.\footnote{Beware, however, that some formalisms like that of algebraic $D$-modules have their adjoints in the wrong direction, so it does not apply in this situation.} In particular, we get the following corollary, by combining it with the construction via ring stacks from the last lecture.

\begin{corollary}\label{cor:realizemotives} Let $R$ be a $k$-algebra stack over an analytic stack $S$, satisfying the hypotheses of Theorem~\ref{thm:conditionsonringstack}, yielding a transmutation functor $X\mapsto X_R$ from separated schemes of finite type over $k$ to analytic stacks over $S$. Then there is a unique map of six-functor formalisms
\[
\mathrm{SH}(X)\to D_{\mathrm{qc}}(X_R)
\]
on separated schemes of finite type $X$ over $\mathbb Z$.
\end{corollary}

\newpage

\section*{Appendix to Lecture XI: $2$-Motives and Ring Stacks}

We have seen that six-functor formalisms on schemes have, on the one hand, an initial example given by Morel--Voevodsky's $\mathrm{SH}$, i.e.~a version of motives; and on the other hand, are closely related to ring stacks. In fact, at the expense of passing to presentable symmetric monoidal $(\infty,2)$-categories, one can prove a very precise version of this.

Namely, for a noetherian ring $k$ of finite Krull dimension as in the lecture, let $\mathrm{Pr}_{\mathrm{SH}_k}$ be the presentable symmetric monoidal $(\infty,2)$-category associated to $\mathrm{SH}_k$. (We implicitly fix a regular cardinal $\kappa$ and work with $\mathrm{Pr}_\kappa$ everywhere.) There is a symmetric monoidal functor
\[
C^{\mathrm{op}}\to \mathrm{CAlg}(\mathrm{Pr}_{\mathrm{SH}_k}),
\]
and in particular the image of $\mathbb A^1_k$ is a $k$-algebra object in
\[
\mathrm{CAlg}(\mathrm{Pr}_{\mathrm{SH}_k})^{\mathrm{op}}.
\]
Psychologically, the reader is invited to think of $\mathrm{CAlg}(\mathrm{Pr}_{\mathrm{SH}_k})^{\mathrm{op}}$ as a generalized kind of geometric objects, roughly $1$-affine stacks (which are determined by their corresponding $D(-)\in \mathrm{CAlg}(\mathrm{Pr})$). The following result was suggested in \cite{ScholzeMotivesRingStacks}, and proved by Aoki in his thesis \cite{AokiThesis}.

\begin{theorem}\label{thm:aokithesis} The presentable symmetric monoidal $(\infty,2)$-category $\mathrm{Pr}_{\mathrm{SH}_k}$ is the initial presentable symmetric monoidal $(\infty,2)$-category $\mathbb C$ equipped with a $k$-algebra object
\[
[\mathbb A^1]\in \mathrm{CAlg}(\mathbb C)^{\mathrm{op}}
\]
satisfying the following conditions:
\begin{enumerate}
\item[{\rm (1)}] The map $f: [\mathbb A^1]\to \ast$ is cohomologically smooth, with trivial homology, i.e.~$f_\sharp f^\ast 1\cong 1$.
\item[{\rm (2)}] The zero section $i: \ast\xrightarrow{0} [\mathbb A^1]$ is a closed immersion, the units $j: [\mathbb A^1]^\times\subset [\mathbb A^1]$ yield an open immersion, and $j_! 1\to 1\to i_\ast 1$ is a fiber sequence.
\end{enumerate}
\end{theorem}

In other words, if we take the Tannakian perspective that any presentable symmetric monoidal $(\infty,2)$-category $\mathbb C$ gives rise to a notion of ``geometric spaces'' via $\mathrm{CAlg}(\mathbb C)^{\mathrm{op}}$, then the presentable symmetric monoidal $(\infty,2)$-category $\mathrm{Pr}_{\mathrm{SH}}$ precisely classifies ring stacks satisfying the (very simple) conditions of Theorem~\ref{thm:conditionsonringstack}.

Aoki proves many variants in his thesis. First, one can define the \'etale version of motives by enforcing surjectivity of the Kummer and Artin--Schreier maps. Moreover, one can define the theory of Berkovich motives \cite{ScholzeBerkovich} by adding a norm map
\[
N: [\mathbb A^1]\to \mathbb R_{\geq 0,\mathrm{Betti}}
\]
(multiplicative, $N^{-1}(0)=0$, and satisfying the triangle inequality) and asking for ball-invariance. In particular, it follows that the construction of ($2$-categorical, hence also $1$-categorical, compatible with $\otimes$, $f^\ast$, $f_!$) realizations of Berkovich motives is reduced to the construction of normed ring stacks satisfying some simple axioms.

\newpage

\section{Lecture XII: A formalism related to arithmetic $D$-modules}

In this final lecture, we sketch how one can define an analytic $\mathbb F_p$-algebra stack whose corresponding six-functor formalism is closely related to the classical theory of arithmetic $D$-modules. This was originally announced in \cite{ScholzeDarmstadt}. Some results towards this were obtained by Bambozzi--Chiarellotto--Vanni \cite{BambozziChiarellottoVanni}; 
our version is also related to the notion of isocrystals with log-decay as in work Kramer-Miller \cite{KramerMiller}.

We note that one version of such a formalism was already constructed using perfection and passing to the analytic de Rham stack of an untilt, see Example~\ref{ex:deRhamofuntilt}. However, the stack associated to $\mathbb A^1_{\mathbb F_p}$ is the analytic de Rham stack of the perfectoid
\[
\widetilde{\mathbb A^{1,\mathrm{an}}_{K'}} = \varprojlim_{T\mapsto T^p} \mathbb A^{1,\mathrm{an}}_{K'}.
\]
Its $\mathbb F_p$-algebra structure is given by the usual formula
\[
(x^{1/p^n})_n + (y^{1/p^n})_n = (z^{1/p^n})_n
\]
where $z = \mathrm{lim}_{n\to \infty} (x^{1/p^n}+y^{1/p^n})^{p^n}$ (with obvious $p$-power roots), i.e.~the formula defining the tilt of perfectoid rings. The limit is well-defined on analytic de Rham stacks; passing to analytic de Rham stacks is equivalent to completing with respect to the norm
\[
N: \widetilde{\mathbb A^{1,\mathrm{an}}_{K'}}\to \mathbb R_{\geq 0}.
\]
Indeed, $N^{-1}(0)$ is precisely the overconvergent neighborhood of $0$, which is the kernel of the map to the analytic de Rham stack.

However, analytic $D$-modules on $\widetilde{\mathbb A^{1,\mathrm{an}}_{K'}}$ is not what one expects arithmetic $D$-modules to look like. Rather, arithmetic $D$-modules should be $D$-modules on an overconvergent version of the unit disc $\mathbb D_{\mathbb Q_p}$ whose connection satisfies a very strong convergence condition. If they are equipped with a Frobenius structure, one can pull them back to an overconvergent perfectoid disc $\widetilde{\mathbb D_{K'}}$ and then one use the Frobenius structure to spread them to the whole perfectoid line $\widetilde{\mathbb A^{1,\mathrm{an}}_{K'}}$. However, without the Frobenius structure, one cannot pass from the usual theory of arithmetic $D$-modules to analytic $D$-modules on $\widetilde{\mathbb A^{1,\mathrm{an}}_{K'}}$.

Still, one can define a different analytic $\mathbb F_p$-algebra stack over $\mathbb Q_{p,\square}$, which is roughly speaking the quotient of an overconvergent disk by an open disk. In fact, taking the actual overconvergent $\mathbb D^\dagger_{\mathbb Q_p}$ and open unit disk $\mathbb D^\circ_{\mathbb Q_p}$, the quotient $\mathbb D^\dagger_{\mathbb Q_p}/\mathbb D^{\circ}_{\mathbb Q_p}$ is a well-defined $\mathbb F_p$-algebra stack that structurally behaves like the algebraic de Rham stack (excision holds, \'etale maps are cohomologically proper, proper maps are cohomologically smooth), and its theory of quasicoherent sheaves basically is the theory of arithmetic $D$-modules. However, we would like to find a formalism with the properties of Theorem~\ref{thm:conditionsonringstack} (which hence yields a realization of motives): we want to have a ring stack $R$ that is cohomologically smooth (not proper), with $0\subset R$ closed (not open).

Thus, we need to find an affinoid version of the open unit disk. Moreover, the connection of arithmetic $D$-modules should be convergent on this open unit disk. Thus, for example, horizontal sections of any Gau\ss--Manin connection should converge on this affinoid version of the open unit disk. Simple examples of this are the logarithm function
\[
\mathrm{log}(1+x) = \sum_{n\geq 1} (-1)^{n-1}\frac{x^n}{n}
\]
or also polylogarithms
\[
\mathrm{Li}_k(x) = \sum_{n\geq 1} \frac{x^n}{n^k}.
\]
These have coefficients that, as $p$-adic numbers, have only polynomial growth in $n$. This turns out to be a general phenomenon (which actually follows from having a Frobenius structure). Thus, we are led to consider the ``tempered'' affinoid open unit disk
\[
\mathbb D^{\circ,\mathrm{temp}}_{\mathbb Q_p} := \mathrm{AnSpec}(\{\sum_{n\geq 0} a_n T^n \mid a_n\in \mathbb Q_p, (|a_n|)_n\ \mathrm{has\ at\ most\ polynomial\ growth}\}).
\]
(Functions with such polynomial growth conditions are often called tempered.) It is a simple exercise to check that this defines an idempotent $\mathbb Q_{p,\square}[T]$-algebra, and hence a subspace of the affine line $\mathbb A^1$ (working throughout over $\mathrm{AnSpec}(\mathbb Q_{p,\square})$.

We need to define a corresponding open version of the closed unit disk. Actually, condition (2) of Theorem~\ref{thm:conditionsonringstack} forces that the open unit disk and the complementary closed unit disk together cover $\mathbb P^1$. This forces us to take as our tempered version of the closed unit disk the complement of $\mathbb D^{\circ,\mathrm{temp}}_{\mathbb Q_p}$ in $\mathbb P^1$. Concretely, up to switching $0$ and $\infty$ in $\mathbb P^1$, this is
\[
\mathbb D^{\mathrm{temp}}_{\mathbb Q_p} = \mathrm{colim}_{k>0} \mathrm{AnSpec}(\{\sum_{n\geq 0} a_n T^n \mid a_n\in \mathbb Q_p, |a_n| = o(n^{-k})\}),
\]
using algebras of power series whose coefficients decay like $n^{-k}$.

Now one checks that $\mathbb D^{\mathrm{temp}}_{\mathbb Q_p}$ is a subring of $\mathbb A^1$, its subspace $\mathbb D^{\circ,\mathrm{temp}}_{\mathbb Q_p}$ is an ideal in $\mathbb D^{\mathrm{temp}}_{\mathbb Q_p}$, and the quotient
\[
(\mathbb A^1_{\mathbb F_p})^{\mathrm{temp}}:=\mathbb D^{\mathrm{temp}}_{\mathbb Q_p} / \mathbb D^{\circ,\mathrm{temp}}_{\mathbb Q_p}
\]
is an $\mathbb F_p$-algebra stack (indeed, $p$ lies in the open unit disk over $\mathbb Q_p$), satisfying the conditions of Theorem~\ref{thm:conditionsonringstack}.

Thus, we get a general transmutation $X\mapsto X^{\mathrm{temp}}$ from separated schemes of finite type $X$ over $k$ towards analytic stacks over $\mathbb Q_{p,\square}$, and a full six-functor formalism. Moreover, one gets a realization of motives into this theory.

\newpage

\section{Miscellaneous}

This section consists of some further remarks and examples of $6$-functor formalisms.

\subsection{$C=\ast$}

As a first example, let us take for $C=\ast$ the trivial category with just one object (and one morphism). A $3$-functor formalism is a symmetric monoidal $\infty$-category $D$. There seems to be nothing to say in this case, but actually there is something to say once one passes to stacks. Namely, the category of sheaves of anima on $C$ is just the $\infty$-category of anima $\mathrm{An}$, and if $D$ is presentable, then $D$ determines uniquely a $6$-functor formalism on $\mathrm{An}$, taking any $X\in \mathrm{An}$ to
\[
D(X) = \mathrm{Fun}(X,D),
\]
i.e.~``$X$-parametrized objects of $D$''. If $D$ is the symmetric monoidal $\infty$-category of spectra, this is the notion of parametrized spectrum of May--Sigurdsson \cite{MaySigurdsson}.

\begin{exercise} Show that the class of morphisms $\tilde{E}$ allowed by Theorem~\ref{thm:stacky6functors} is the class of morphisms $f: X\to Y$ of anima that on every connected component of $X$ are $n$-truncated for some $n$.
\end{exercise}

\begin{exercise} Show that in the $6$-functor formalism from Theorem~\ref{thm:stacky6functors}, all maps in $\tilde{E}$ are cohomologically \'etale.
\end{exercise}

One can in fact directly construct a $6$-functor formalism on $\mathrm{An}$ by sending $X$ to $D(X)=\mathrm{Fun}(X,D)$ and taking for $I$ all maps (so also $E$ consists of all maps), while $P$ consists only of isomorphisms. This makes $f_!$ in general the left adjoint of $f^\ast$, and it extends the formalism coming from Theorem~\ref{thm:stacky6functors} to not necessarily truncated morphisms.

The class of cohomologically proper morphisms depends a lot on $D$. In general, if $D$ is stable (or just preadditive), then all $0$-truncated maps with finite fibres are $D$-cohomologically proper, and this is a way of saying that finite coproducts agree with finite products in $D$. This means that $1$-truncated maps whose fibres have finite automorphism groups are ``cohomologically separated'', yielding a natural transformation $f_!\to f_\ast$ for such $f$. In particular, for a finite group $G$ and $f: BG\to \ast$, one has an identification between $D(BG)$ and $G$-equivariant objects in $D$ (in practice, this is the derived category of $G$-representations), and $f_!$ is $G$-homology, $f_\ast$ is $G$-cohomology, and $f_!\to f_\ast$ yields the norm map from $G$-homology to $G$-cohomology. If the order of $G$ is invertible in $\mathrm{End}_D(1_D)$, this is an isomorphism, but there are also other situations, notably if $D$ is the $\infty$-category of $K(1)$-local spectra, where this is related to the telescopic Tate vanishing, see for example \cite{ClausenMathewTate} for a quick proof. If this map is an isomorphism, one gets resulting comparison maps for $2$-truncated maps, etc.

In particular, a phenomenon known as ``ambidexterity'' in algebraic topology gives the following theorem:

\begin{theorem}[{Hopkins--Lurie, \cite{HopkinsLurieAmbidexterity}}]\label{thm:ambidexterityKnlocal} Let $D$ be the symmetric monoidal $\infty$-category of $K(n)$-local spectra for $n\geq 1$ (and some implicit prime $p$). Then any $n$-truncated map $f: X\to Y$ of anima all of whose fibres have finite $\pi_i$, for $i=0,\ldots,n$, is $D$-cohomologically proper.
\end{theorem}

On another note, one can wonder for which $X\in\mathrm{An}$ the object $1_X\in D(X)$ is $f$-prim for the projection $f: X\to \ast$. Note that after Proposition~\ref{prop:checkprimsheaf}, we constructed a general transformation
\[
f_!(-\otimes \mathbb D_X)\to f_\ast
\]
for a certain object $\mathbb D_X\in D(X)$ which is in fact exactly the Spivak--Klein dualizing object when $D=\mathrm{Sp}$, and the map is exactly the twisted norm map; see \cite[Section I.4]{NikolausScholze} for an account. Here, the left-hand side is in fact the colimit-preserving approximation to $f_\ast$, and so the map is an isomorphism if and only if $f_\ast$ preserves all direct sums. This is the case, in particular, if $X$ is a compact object of $\mathrm{An}$. More generally, for any ``finitely dominated'' map of anima $f: X\to Y$, the sheaf $1_X$ is $f$-prim.

A lot of work in (parametrized) homotopy theory can then be recast in this language, and I will not attempt to do the vast literature justice due to ignorance on my side. Let me just cite, as one example with a similar point of view to these notes, the paper by Cnossen \cite{Cnossen}; see also its introduction and extended bibliography.

\begin{remark} Generalizing this example, one can take for $C$ the category of transitive $G$-sets, for some (abstract, say) group $G$, where one again takes for $I$ all morphisms (so also $E$ consists of all morphisms), while $P$ consists only of isomorphisms. In that case, Elmendorf's theorem \cite{Elmendorf} says that the $\infty$-category of presheaves of anima on $C$ is exactly the $\infty$-category of ``$G$-spaces''. Moreover, $6$-functor formalisms on $C$ with values in $\mathrm{Pr}^L$ are $G$-symmetric monoidal presentable $\infty$-categories, i.e.~functors from $C^{\mathrm{op}}$ to symmetric monoidal presentable $\infty$-categories. One can then again extend such $6$-functor formalisms to all ``$G$-spaces''. I believe much of the previous discussion can then be generalized, and is related to a lot of work on (genuine) $G$-equivariant homotopy theory.
\end{remark}

\subsection{$C=\mathrm{ProFin}$}

Next, let us consider $C=\mathrm{ProFin}$. As morphisms, we take for $P$ all maps (so also $E$ consists of all maps), while for $I$ we could in principle allow open immersions, but we could also just restrict to isomorphisms. There are many possible $6$-functor formalisms, so let us restrict to the case of the functor taking any $S\in \mathrm{ProFin}$ to the stable $\infty$-category $D(S,k)$ of sheaves on $S$ with values in modules over some $E_\infty$-ring $k$. Note that sheaves on $S$ are just functors taking any open closed subset $U\subset S$ to the value on $S$, subject to taking finite disjoint unions to finite products. This is very much a finitary condition, and using this it is trivial to prove proper base change, so this functor indeed extends to a $6$-functor formalism.

Moreover, the functor $S\mapsto D(S,k)$ is a hypersheaf for the Grothendieck topology used in condensed mathematics, i.e.~covers are generated by finite families of jointly surjective maps. We can thus extend $X\mapsto D(X,k)$ to all condensed anima; and if one possibly slightly restricts the Grothendieck topology (we did not know whether all covers in condensed mathematics satisfy universal $!$-descent) we also get an extended $6$-functor formalism. This applies in particular to locally profinite sets, and for open immersions of such, the functor $f_!$ is defined and agrees with the left adjoint of $f^\ast$ (i.e.~they are cohomologically \'etale). To see this, cover locally profinite sets via open and closed subsets of profinite sets, and use that $f_!$ commutes with all direct sums to see that it must be the expected functor.\footnote{Thanks to Clausen for explaining this!}

We also note that one can show that the morphism from the Cantor set to the interval is of universal $!$-descent, in fact it satisfies the hypotheses of Proposition~\ref{prop:Dprimdescent}. Thus, for finite-dimensional compact Hausdorff spaces, the $!$-functors are defined, and behave as expected. On the other hand, for a Hilbert cube, the $!$-functors are not defined.

This example of condensed anima is spelled out in detail in the work of Heyer--Mann \cite{HeyerMann}, with applications to smooth representation theory.

\subsection{$C=\mathrm{CHaus}$}

Now consider $C=\mathrm{CHaus}$, compact Hausdorff spaces. All morphisms are in $P$, but only isomorphisms in $I$, and we again use the functor $X\mapsto D(X,k)$, the $\infty$-category of sheaves on $X$ with values in $D(k)$, for some $E_\infty$-ring $k$. Again, we could also start with locally compact Hausdorff spaces, but on the level of sheaves of anima, the distinction disappears.

In this case, the resulting $D$-topology in fact depends on $k$, and certainly not all covers in the sense of condensed mathematics are allowed; for example, the cover of the Hilbert cube by a Cantor set is not allowed. However, if $X$ is finite-dimensional, then one can cover it by a Cantor set in the $D$-topology. Thus, the formalism from this section and the previous section agree when specialized to finite-dimensional spaces, but in general they differ.

We note that in both formalisms, with $C=\mathrm{ProFin}$ and $C=\mathrm{CHaus}$, the $\infty$-category of sheaves of anima mixes purely homotopy-theoretic spaces with actual topological spaces. So for a manifold $X$, one also has its associated anima $|X|$, with a map $f: X\to |X|$, and then $D(X)$ consists of all sheaves on $X$, while $D(|X|)$ is equivalent, via pullback along $f$, to the locally constant sheaves in $D(X)$. Moreover, the truncated map $f_n: X\to \tau_{\leq n} |X|$ is also such that $f_{n!}$ is defined. Unfortunately, we do not currently see how to define $f_!$.

\bibliographystyle{amsalpha}

\bibliography{SixFunctors}

\end{document}